\pdfoutput=1
\documentclass[reqno, a4paper]{amsart}
\usepackage{fixltx2e}
\usepackage{bookmark}
\usepackage{etoolbox}
\usepackage[usenames, dvipsnames]{color}
\usepackage{mathtools}

\definecolor{l-gray}{gray}{0.5}
\usepackage[english]{babel}
\usepackage{enumitem}

\usepackage[utf8]{inputenc}  
\usepackage[T1]{fontenc}
\usepackage{pifont}
\usepackage{gensymb}

\usepackage{rotating}
\usepackage{scalerel}

\usepackage{amsmath}
\usepackage{amssymb}
\usepackage{amsthm, remreset}
\usepackage{amscd}
\usepackage{centernot}
\usepackage{mathdots}
\usepackage[all,color,cmtip]{xy}

\makeatletter
\@namedef{subjclassname@2020}{%
  \textup{2020} Mathematics Subject Classification}
\makeatother

\newdir{ >}{{}*!/-7pt/\dir{>}}
\newdir{->-}{\dir{-}*\dir{>}*\dir{-}}
\newdir{  |}{{}*!/-20pt/\dir{|}}
\newdir{ |}{{}*!/-7pt/\dir{|}}
\newdir{|- }{!<2.5pt,0pt>\dir{|}*\dir{-}*{}}
\newdir{> }{!/7pt/\dir{>}*{}}
\newdir{>}{{}*:(1,-.2)@^{>}*:(1,+.2)@_{>}}
\newdir{<}{{}*:(1,+.2)@^{<}*:(1,-.2)@_{<}}

\newdir{>>}{{}*!/3.5pt/:(1,-.2)@^{>}*!/3.5pt/:(1,+.2)@_{>}*!/7pt/:(1,-.2)@^{>}*!/7pt/:(1,+.2)@_{>}}
\newdir{ >>}{{}*!/8pt/@{|}*!/3.5pt/:(1,-.2)@^{>}*!/3.5pt/:(1,+.2)@_{>}}
\newdir{ |>}{{}*!/-3.5pt/@{|}*!/-8pt/:(1,-.2)@^{>}*!/-8pt/:(1,+.2)@_{>}}
\newdir{ >}{{}*!/-8pt/@{>}}

\def\pullback{
 \ar@{-}[]+R+<3pt,-1pt>;[]+RD+<3pt,-3pt>%
 \ar@{-}[]+D+<1pt,-3pt>;[]+RD+<3pt,-3pt>}
 
 \def\lowpullback{
 \ar@{-}[]+R+<6pt,-6pt>;[]+D+<15pt,-16pt>;[]+RD+<6pt,-6pt>%
 \ar@{-}[]+D+<4pt,-10pt>;[]+RD+<6pt,-16pt>}

\def\pullbackdots{%
 \ar@{.}[]+R+<6pt,-1pt>;[]+RD+<6pt,-6pt>%
 \ar@{.}[]+D+<1pt,-6pt>;[]+RD+<6pt,-6pt>}
 
\def\pushout{%
 \ar@{-}[]+L+<-5pt,1pt>;[]+LU+<-5pt,5pt>%
 \ar@{-}[]+U+<-1pt,5pt>;[]+LU+<-5pt,5pt>}
 
\def\uppushout{%
 \ar@[Gray]@{-}[]+L+<-5pt,-1pt>;[]+LD+<-5pt,-5pt>%
 \ar@[Gray]@{-}[]+D+<-1pt,-5pt>;[]+LD+<-5pt,-5pt>}
 
 \def\lowuppushout{%
 \ar@{-}[]+L+<-10pt,-2pt>;[]+LD+<-10pt,-10pt>%
 \ar@{-}[]+D+<-2pt,-10pt>;[]+LD+<-10pt,-10pt>}

\def\splitpullback{%
 \ar@{-}[]+R+<6pt,-.51ex>;[]+RD+<6pt,-6pt>%
 \ar@{-}[]+D+<.51ex,-6pt>;[]+RD+<6pt,-6pt>}

\def\skewpullback{%
 \ar@{-}[]+R+<6pt,-1pt>;[]+RD+<-7pt,-8pt>%
 \ar@{-}[]+D+<-12pt,-8pt>;[]+RD+<-7pt,-8pt>}

\def\dottedpullback{%
 \ar@{.}[]+R+<6pt,-1pt>;[]+RD+<6pt,-6pt>%
 \ar@{.}[]+D+<1pt,-6pt>;[]+RD+<6pt,-6pt>}

\usepackage{geometry}
\let\origappendix\appendix 
\renewcommand\appendix{\clearpage\pagenumbering{roman}\origappendix}

\usepackage{graphicx}
\usepackage{epstopdf}
\newtheorem{theorem}[subsubsection]{Theorem}
	
\newtheorem{observation}[subsubsection]{Observation}	
\newtheorem{corollary}[subsubsection]{Corollary}	
\newtheorem{proposition}[subsubsection]{Proposition}	
\newtheorem{notation}[subsubsection]{Notation}	
\newtheorem{lemma}[subsubsection]{Lemma}
\newtheorem{definition}[subsubsection]{Definition}

\newtheorem{example}[subsubsection]{Example}

\newtheorem{remark}[subsubsection]{Remark}
\newtheorem{convention}[subsubsection]{Convention}

\newcommand{\Z}{\mathbb{Z}}

\newcommand{\N}{\mathbb{N}}
\newcommand{\EE}{\mathcal{E}}

\newcommand{\C}{\mathcal{C}}
\newcommand{\G}{\mathcal{G}}

\newcommand{\X}{\mathcal{X}}

\newcommand{\SET}{\mathsf{Set}}

\newcommand{\QND}{\mathsf{Qnd}}

\newcommand{\RCK}{\mathsf{Rck}}
\newcommand{\EXT}{\mathsf{Ext}}

\newcommand{\CEXT}{\mathsf{CExt}}

\newcommand{\ARR}{\mathsf{Arr}}

\newcommand{\GRP}{\mathsf{Grp}}
\newcommand{\GRPD}{\mathsf{Grpd}}

\newcommand{\AB}{\mathsf{Ab}}

\newcommand*{\defeq}{\mathrel{\vcenter{\baselineskip0.5ex \lineskiplimit0pt
                     \hbox{\scriptsize.}\hbox{\scriptsize.}}}%
                     =}

\newcommand*{\qndiop}{\triangleleft^{-1}}
\newcommand*{\qndop}{\triangleleft}

\newcommand\ttop{\scaleobj{0.6}{\top}}
\newcommand\pperp{\scaleobj{0.6}{\perp}}

\DeclareMathOperator{\Co}{C_{0}}
\DeclareMathOperator{\Ci}{C_{1}}

\DeclareMathOperator{\Fi}{F_{1}}
\DeclareMathOperator{\Fii}{F_{1}^i}
\DeclareMathOperator{\Fio}{F_{1}^0}

\DeclareMathOperator{\Fa}{F_{ab}}
\DeclareMathOperator{\Fg}{F_g}
\DeclareMathOperator{\Fgd}{F\degree_g}
\DeclareMathOperator{\Fq}{F_q}
\DeclareMathOperator{\Fr}{F_r}
\DeclareMathOperator{\Frq}{\, _rF_q}
\DeclareMathOperator{\id}{id}

\DeclareMathOperator{\Ss}{S}
\DeclareMathOperator{\ab}{ab}

\DeclareMathOperator{\Cc}{C}

\DeclareMathOperator{\Eq}{Eq}

\DeclareMathOperator{\Aut}{Aut}
\DeclareMathOperator{\Adj}{Adj} 
\DeclareMathOperator{\Inn}{Inn}

\DeclareMathOperator{\Conj}{Conj}
\DeclareMathOperator{\Pth}{Pth}
\DeclareMathOperator{\Loop}{Loop}
\DeclareMathOperator{\pth}{pth}
\DeclareMathOperator{\U}{U}
\DeclareMathOperator{\I}{I}

\DeclareMathOperator{\Ker}{Ker}
\DeclareMathOperator{\Cst}{Cst}

\usepackage{comment}
\includecomment{comment}
\definecolor{britishracinggreen}{rgb}{0.0, 0.26, 0.15}
\definecolor{darkpink}{rgb}{0.91, 0.33, 0.5}

\newcommand{\gr}[1]{{\color{darkpink}{\underline{#1}}}}

\newcommand{\red}[1]{{\color{Red}#1}}

\newcommand{\gray}[1]{{\color{Gray}#1}}
\newcommand{\R}[1]{{\rm(R#1)}}
\newcommand{\Q}[1]{{\rm(Q#1)}}
\newcommand{\K}[1]{{\rm K}_{#1}}
\newcommand{\Inv}{\rm(Inv)}

\begin{document}

\title[Higher coverings of racks and quandles -- Part I]{Higher coverings of racks and quandles -- Part I}
\author{Fran\c{c}ois Renaud}

\email[Fran\c{c}ois Renaud]{francois.renaud@uclouvain.be}

\address{Institut de Recherche en Math\'ematique et Physique, Universit\'e catholique de Louvain, che\-min du cyclotron~2 bte~L7.01.02, B--1348 Louvain-la-Neuve, Belgium}

\thanks{The author is a Ph.D.\ student funded by \textit{Formation à la Recherche dans l'Industrie et dans l'Agriculture} (\textit{FRIA}) as part of \textit{Fonds de la Recherche Scientifique - FNRS}.}

\begin{abstract}
This article is the second part of a series of three articles, in which we develop a higher covering theory of racks and quandles. This project is rooted in M.~Eisermann's work on quandle coverings, and the categorical perspective brought to the subject by V.~Even, who characterizes coverings as those surjections which are central, relatively to trivial quandles. We extend this work by applying the techniques from higher categorical Galois theory, in the sense of G.~Janelidze, and in particular we identify meaningful higher-dimensional centrality conditions defining our higher coverings of racks and quandles.

In this first article (Part I), we revisit and clarify the foundations of the covering theory of interest, we extend it to the more general context of racks and mathematically describe how to navigate between racks and quandles. We explain the algebraic ingredients at play, and reinforce the homotopical and topological interpretations of these ingredients. In particular we justify and insist on the crucial role of the left adjoint of the conjugation functor $\Conj$ between groups and racks (or quandles). We rename this functor $\Pth$, and explain in which sense it sends a rack to its group of homotopy classes of paths. We characterize coverings and relative centrality using $\Pth$, but also develop a more visual ``geometrical'' understanding of these conditions. We  use alternative generalizable and visual proofs for the characterization of central extensions of racks and quandles. We complete the recovery of M.~Eisermann's ad hoc constructions (weakly universal cover, and fundamental groupoid) from a Galois-theoretic perspective. We sketch how to deduce M.~Eisermann's detailed classification results from the fundamental theorem of categorical Galois theory. As we develop this refined understanding of the subject, we lay down all the ideas and results which will articulate the higher-dimensional theory developed in Part II and III.
\end{abstract}
\maketitle

\section{Introduction}
\subsection{Context}\label{SectionContext}
We like to describe racks as sets equipped with a self-distributive system of symmetries, attached to each point. The term \emph{wrack} was introduced by J.C.~Conway and G.C.~Wraith, in an unpublished correspondence of 1959. Their curiosity was driven towards the algebraic structure obtained from a group, when only the operations defined by conjugation are kept, and one forgets about the multiplication of elements. Sending a group to its so defined ``wreckage'' defines the \emph{conjugation functor} from the category of groups to the category of wracks. We use the more common spelling \emph{rack} as in \cite{FenRou1992} and  \cite{Ryd1993}. Other names in the literature are \emph{automorphic sets} by E.~Brieskorn \cite{Bri1988}, \emph{crossed G-sets} by P.J.~Freyd and D.N.~Yetter \cite{FreYet1989}, and \emph{crystals} by L.H.~Kauffman in \cite{Kau2012}. The former is (as far as we know) the first detailed published study of these structures.

The image of the conjugation functor actually lands in the category of those racks whose symmetries are required to fix the point they are attached to. Such structures were introduced and extensively studied by D.E.~Joyce in his PhD thesis \cite{Joy1979}, under the name of \emph{quandles}. Around the same time (1980's), S.~Matveev was studying the same structures independently, under the name of \emph{distributive groupoids} \cite{Mat1984}. D.E.~Joyce describes the theory of quandles as ``the algebraic theory of group conjugation'', and uses them to produce a complete knot invariant for oriented knots.

Over the last decades, racks and quandles have been applied to knot theory and physics in various works -- see for instance \cite{Joy1982,Bri1988,FenRou1992,Dri1992,FeRoSa1995,CaKaSa2001,CJKLS2003,Kau2012,Eis2014} and references there. More historical remarks are made in \cite{FenRou1992}, including references to applications in computer science. In geometry, the earlier notion of \emph{symmetric space}, as studied by O.~Loos in \cite{Loos1969}, gives yet another context for applications -- see \cite{Ber2008,Hel2012} for up-to-date introductions to the field. This line of work goes back to 1943 with M.~Takasaki's abstraction of a \emph{Riemannian symmetric space}: a \emph{kei} \cite{Tak1943}, which would now be called an \emph{involutive quandle}.

More recently (2007) M.~Eisermann worked on a covering theory for quandles (published in \cite{Eis2014}). He defines quandle coverings, and studies them in analogy with topological coverings. In particular, he derives several classification results for coverings, in the form of Galois correspondences as in topology (or Galois theory). In order to do so, he works with ad hoc constructions such as a (weakly) universal covering or a fundamental group(oid) of a quandle. Even though the link with coverings is unclear a priori, these constructions use the left adjoint of the conjugation functor, which is justified a posteriori by the fact that the theory produces the aforementioned classification results.

In his PhD thesis \cite{Eve2014}, V.~Even applies categorical Galois theory, in the sense of G.~Janelidze \cite{Jan1990}, to the context of quandles. By doing so, he establishes that M.~Eisermann's coverings arise from the admissible adjunction between \emph{trivial quandles} (i.e.~sets) and quandles, in the same way that topological coverings arise from the admissible adjunction between discrete topological spaces (i.e.~sets) and locally connected topological spaces (see Section 6.3 in \cite{BorJan1994}). He also derives that M.~Eisermann's notion of fundamental group of a connected, pointed quandle coincides with the corresponding notion from categorical Galois theory. This, in turn, makes the bridge with the fundamental group of a pointed, connected topological space. By doing so, V.~Even clarifies the analogy with topology even though his results still rely on ad hoc constructions such as M.~Eisermann's weakly universal covers.

\subsection{In this article} We explicitly extend M.~Eisermann and V.~Even's work to the more general context of racks, as it was already suggested in their articles. We then clarify and justify the use of the different algebraic and topological ingredients of their study with the perspective of developing a higher-dimensional covering theory. 

In Section \ref{SectionCategoricalGaloisTheory}, we  describe enough of categorical Galois theory to motivate the overall project and explain the results we seek. In particular, we specify which Galois structures (say $\Gamma$) we are interested in, which conditions on these Galois structures we need in order to achieve our higher-dimensional goals, and finally we comment on the use of projective presentations, and a global strategy to characterize the $\Gamma$-coverings (also called central extensions) arising in the context of such a suitable Galois structure $\Gamma$.

In Section \ref{SectionBiasedintroToRacksAndQuandles}, we introduce the fundamentals of the theory of racks and quandles, with the biased perspective of the covering theory which follows. We start (Section \ref{SectionAxiomsAndBasicConcepts}) with a short study of the axioms, our first comments relating groups, racks and quandles, and the basic concepts of symmetry, inner automorphisms, and their actions. Next (Section \ref{SectionFromAiomsToGeometricalFeatures}), we develop some intuition about the geometrical features of a rack. We illustrate our comments on the construction of the free rack, and recall the construction of the canonical projective presentation of a rack, which presents the elements in a rack with the geometrical features of those in the appropriate free rack. We then introduce the connected component adjunction (Section \ref{SectionConnectedComponentsAdjunction}), from which the covering theory of interest arises. The concepts of trivializing relation, connectedness, primitive path, orbit congruence, etc.~are recalled. We propose to derive the trivializing relation from the geometrical understanding of free objects via projective presentations. We recall the admissibility results for the connected component adjunction and comment on the non-local character of connectedness. We illustrate our visual approach of coverings on the characterization of trivial extensions. 

Section \ref{SubsectionGroupOfPath} follows with a description of the links between the construction of $\Pth$ (the group of paths functor), left adjoint of the conjugation functor, and the equivalence classes of \emph{tails} of formal terms in the language of racks. Again, we propose to look at the simple description of $\Pth$ on free objects, and extend this description to all objects, via the canonical projective presentations. We describe the action of the group of paths and how it relates to inner automorphisms and equivalence classes of primitive paths in general. The free action of this group on free objects is recalled. We emphasize the functoriality of $\Pth$ on all morphisms by contrast with the non-functorial construction of inner automorphisms. We describe the kernels of the induced maps between groups of paths, in preparation for the characterizations of centrality. We insist on the fact that the role of $\Pth$ (as left adjoint of the conjugation functor) is the same in racks and in quandles, although it is more intimately related to racks in design. We conclude this survey with a study of the adjunction between racks and quandles (Section \ref{SectionWorkingWithQuandles}). We derive its admissibility and deduce that all extensions are central with respect to this adjunction. We build the free quandle $\Fq(A)$ on a set $A$ in a way that illustrates best the journey from one context to the other. By doing so, we explain interest for pairs of generators with opposite exponents (the transvection group, understood via the functor $\Pth\degree$). We show that the normal subgroup $\Pth\degree(\Fq(A)) \leq \Pth(\Fq(A))$ of the group of paths acts freely on $\Fq(A)$ as expected. 

In Section \ref{SectionCoveringTheoryOfRacksAndQuandles}, we give a comprehensive Galois-theoretic account of the low-dimensional covering theory of quandles, which we extend to the suitable context of racks. Coverings are described, as well as their different characterizations, using the kernels of induced maps $\Pth(f)$ between groups of paths, but also via the concept of \emph{closing horns}. We recall that primitive extensions are coverings, and coverings are preserved and reflected by pullbacks along surjections (i.e.~central extensions are coverings). We find counterexamples for Theorem 4.2 in \cite{BLRY2010}, and finally illustrate our ``geometrical'' approach to centrality on the characterization of normal extensions. In Section \ref{SectionCharacterizingCentralExtensions}, we give proof(s) -- generalizable to higher dimensions -- for the characterization of central extensions of racks and quandles. We then theoretically understand how the concepts of centrality in racks and in quandles relate (Section \ref{SectionComparingAdjunctionsByFactorization}), using the factorization of the connected component adjunction through the adjunction between racks and quandles. Amongst other results, we derive that the centralizing relations, if they exist, should be the same in both contexts. We then prove (Section \ref{SectionCentralizingExtensions}) several characterizations of these centralizing relations, and extend the results from \cite{DuEvMo2017} on the reflectivity of coverings in extensions. In preparation for the admissibility in dimension 2, we show that coverings are closed under quotients along double extensions (towards ``Birkhoff'') and we show the commutativity property of the kernel pair of the centralization unit (towards ``strongly Birkhoff''). We then move to Section \ref{SectionWucAndFundamentalGroupoid}, and the construction of weakly universal covers from the centralization of canonical projective presentations. From there, we build the fundamental Galois groupoid of a rack and of a quandle, establishing the homotopical interpretations of $\Pth$ and $\Pth\degree$. In Section \ref{SectionFundamentalTheorem} we illustrate the use of the fundamental theorem of categorical Galois theory in this context. We conclude the article with a comment on the relationship between centrality in racks and in groups (Section \ref{SectionRelationshipToAbelianisation}).

\subsection{The point of view of categorical Galois theory}\label{SectionCategoricalGaloisTheory}

Categorical Galois theory (in the sense of \cite{Jan1990}, see also \cite{JK1994}) is a very general (and thus abstract) theory with rich and various interpretations depending on the numerous contexts of application. On a theoretical level, Galois theory exhibits strong links with, for example, factorization systems, commutator theory, homology and homotopy theory (see for instance \cite{JK2000b,CJKP1997,Jan2016}). Looking at applications, it unifies, in particular, the theory of field extensions from classical Galois theory (as well as both of its generalizations by A.~Grothendieck and A.~R.~Magid.), the theory of coverings of locally connected topological spaces, and the theory of central extensions of groups. The covering theory of racks and quandles \cite{Eis2014} is yet another example \cite{Eve2014}, which combines intuitive interpretations inspired by the topological example with features of the group theoretic case. A detailed historical account of the developments of Galois theory is given in \cite{BorJan1994}. In this introduction we avoid the technical details of the general theory, but hint at the very essentials needed by us.

Categorical Galois theory always arises from an \emph{adjunction} (say ``relationship'') between two categories (think ``contexts''). For our purposes, there shall be a ``primitive context'', say $\X$, which sits inside a ``sophisticated context'', say $\C$; such that moreover $\C$ \emph{reflects} back on $\X$ -- e.g. sets, considered as discrete topological spaces, sit inside locally connected topological spaces which reflect back on sets via the \emph{connected component functor} $\pi_0$ \cite[Section 6.3]{BorJan1994}. Under certain hypotheses on these contexts and their relationship, categorical Galois theory studies a (specific) \emph{sphere of influence} of the context $\X$ in the context $\C$, with respect to this relationship -- the idea of a relative notion of centrality \cite{Froe1963,Lue1967}. This influence (centrality) is discussed in terms of a chosen class of morphisms in these categories which we call \emph{extensions} (e.g.~the class of \emph{surjective étale maps}). 
\begin{figure}[hb!]
{\begin{center}
\includegraphics[width=0.48\linewidth]{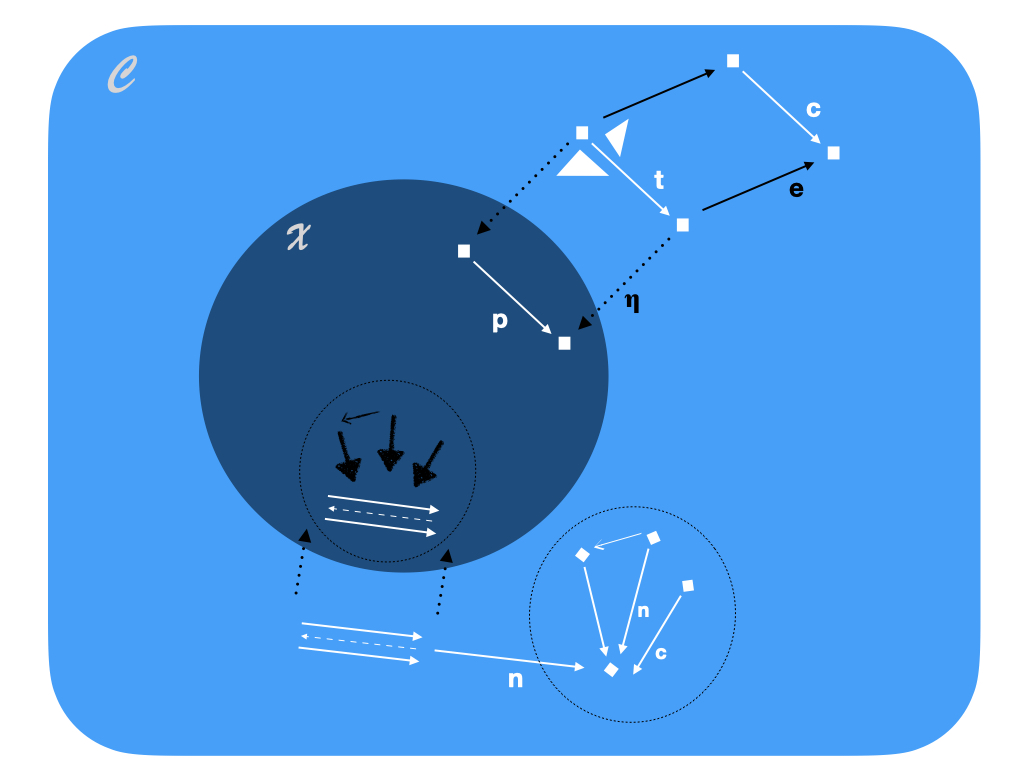}
\end{center}}\caption{A kid's drawing of categorical Galois theory}\label{FigureCategoricalGaloisTheory}
\end{figure}
\begin{convention}\label{RemarkGaloisStructure} For our purposes, a \emph{Galois structure} $\Gamma \defeq (\C,\X,F,\I,\eta,\epsilon,\EE)$ (see \cite{JanComo1991}), is the data of such an inclusion $\I$, of a \emph{full} (\emph{replete}) subcategory $\X$ in $\C$, with left adjoint $F\colon{\C \to \X}$, unit $\eta$, counit $\epsilon$ and a chosen class of \emph{extensions} $\EE$ in $\C$, such that $\EE$ contains all isomorphisms, $\EE$ is closed under composition and ``$F(\EE)\subset \EE$'' in the appropriate sense. Since the fundamental tool at play in the theory is that of taking \emph{pullbacks} \cite{McLane1997}, pullbacks along extensions need to exist and the pullback of an extension should be an extension. For our purposes $\EE$ shall always be a class of \emph{regular epimorphisms} and the components of the unit will be extensions: we say that $\X$ is $\EE$-reflective in $\C$.
\end{convention}

Given such a structure $\Gamma$, the idea is that extensions ``which live in $\X$'', which we call \emph{primitive extensions}, induce, in two steps, two other notions of extensions in $\C$, which are somehow related to primitive extensions ``in a tractable way''. The first step influence: \emph{trivial $\Gamma$-coverings} -- which in this work we refer to mostly as \emph{trivial extensions}, as in \cite{JK1994} -- are those extensions $t$ of $\C$ that are directly constructed from a primitive extension $p$ in $\X$, by pullback along a \emph{unit morphism} $\eta_{\centerdot}$ (see Figure \ref{FigureCategoricalGaloisTheory} -- this gives topological \emph{trivial coverings} in our example). Then the second step influence: \emph{$\Gamma$-coverings}, also called \emph{central extensions} (which are topological \emph{coverings} in our example), are those extensions which ``are locally trivial extensions'', i.e.~extensions which can be \emph{split} by another extension, where an extension $e$ \emph{splits} an extension
$c$ when the pullback $t$ of $c$ along $e$ gives a trivial $\Gamma$-covering (or trivial extension -- see Figure \ref{FigureCategoricalGaloisTheory}). In this article, we prefer to use the terminology \emph{trivial extension} (for trivial $\Gamma$-covering) and \emph{central extension} (for $\Gamma$-covering), as it is the case in \cite{JK1994}, since most of the Galois structures we study are examples of the Galois structures considered in \cite{JK1994}.

For these definitions to be meaningful, we work with Galois structures $\Gamma$ such that: pullbacks of unit morphisms are unit morphisms (admissibility -- see \cite{JK1994} for a precise definition), and moreover: extensions are of (\emph{relative}) \emph{effective descent} (i.e.~pulling back along extensions is an ``algebraic operation'' -- see \cite{JanTho1994,JanSoTho2004}). Under such conditions the central extensions ($\Gamma$-coverings) above a given object can be classified using data which is internal to $\X$ -- in a form which is classically called a \emph{Galois correspondence}, as in the theory of coverings in topology \cite[Theorem 1.38]{Hat2001}. We do not give further details about admissibility (or effective descent), since it is enough to understand it as the condition for Galois theory to be applicable. We actually work with a stronger property for Galois structures described in Section \ref{SectionAdmissibilityStronglyBirkhoff}.

More precisely, there is actually a third class of extensions, called \emph{normal $\Gamma$-coverings}, or for our purposes \emph{normal extensions}, which are those central extensions ($\Gamma$-coverings) that are split by themselves (the projections of their \emph{kernel pair} are trivial extensions). Now if $\G$ is the image by the reflector $F$ of the \emph{groupoid} induced by the kernel pair of a normal extension (normal $\Gamma$-covering) $n \colon{A \to B}$, then $\G$ is a groupoid in $\X$. Internal groupoids and internal actions are well explained in \cite{John1977}, read more about the use of groupoids in \cite{Bro2006}. The \emph{fundamental theorem of categorical Galois theory} then says that \emph{internal presheaves} over that groupoid $\G$ (think ``groupoid actions in $\X$'') yield a category which is equivalent to the category of those extensions above $B$ which are split by $n$. If $n$ splits all extensions above $B$, for instance, in the contexts of interest, when it is a \emph{weakly universal central extension} above $B$ (see Section \ref{SectionProjectivesAndWuc}), then $\G$ is the \emph{fundamental groupoid} of $B$, which thus classifies all central extensions above $B$. A \emph{weakly universal central extension above $B$} is a central extension with codomain $B$, which factors through any other central extension above $B$. Note that the conditions -- connectedness, local path-connectedness and semi-local simply-connectedness -- on the space $X$ in \cite[Theorem 1.38]{Hat2001} are there to guarantee the existence of a weakly universal covering above $X$ \cite[Section 6.6-8]{BorJan1994}.

In the case of groups, the adjunction of interest is $\ab \dashv \I$, the \emph{abelianization adjunction}, where the \emph{left adjoint} $\ab\colon{\GRP \to \AB}$ sends a group $G$ to the abelian group $G / [G,G]$, constructed by quotienting out the \emph{commutator subgroup} $[G,G]$ of $G$. In this context, the \emph{extensions} are chosen to be the \emph{regular epimorphisms}, which are merely the surjective group homomorphisms. Given this Galois structure, say $\Gamma_G$, the concept of a $\Gamma_G$-covering coincides with the concept of a \emph{central extension} from group theory. The fundamental theorem can for instance be used to show that given a \emph{perfect group} $G$, the \emph{second integral homology group} of $G$ can be \emph{presented} as a ``Galois group'' (see \cite{GraVdl2008,Jan2008}), which implies for instance that $G$ has a \emph{universal} central extension. Note that in \cite{JK1994}, G.~Janelidze and M.~Kelly introduce \emph{a general notion of central extension} using the concept of $\Gamma$-covering for those Galois structures $\Gamma$ that are alike $\Gamma_G$ from the group theoretical example.

Note that in Part I, the adjunction which gives rise to the covering theory of racks and quandles is related to $\ab \dashv \I$ and is also such that $\X$ is a \emph{subvariety of algebras} in $\C \defeq \RCK/\QND$. Such data always gives a Galois structure, by defining extensions to be the surjective maps \cite{JK1994}. Moreover, $\X = \SET$ is here equivalent to the category of sets, such that the left adjoint $F \defeq \pi_0\colon{\RCK/\QND = \C \to \X}$ can be interpreted as a connected component functor like in topology.

Now from the example of groups, and the aforementioned observation about links with homology, the development of Galois theory lead for instance to a generalisation \cite{EvGrVd2008} of the Hopf formulae for the (integral) homology of groups \cite{BroEll1988} to other non-abelian settings, leading to a whole new approach to non-abelian homology, by the means of higher central extensions \cite{Jan1991,EvGrVd2008,Ever2010}. This approach is compatible with settings such as the cotriple homology of Barr and Beck \cite{BarBec1969,EvVdl2013}, including, for instance, group homology with coefficients in the cyclic groups $\Z_n$. In order to access the relevant higher-dimensional information, as in \cite{EvGrVd2008}, one actually ``iterates'' categorical Galois theory. The increase in dimension consists in shifting from the context of $\C$ to the \emph{category of extensions} of $\C$: $\EXT\C$ defined as the full subcategory of the arrow category $\ARR\C$ with objects being extensions. A morphism $\alpha\colon{f_A \to f_B}$ in such a category of morphisms is given by a pair of morphisms in $\C$, which we denote $\alpha = (\alpha_{\ttop},\alpha_{\pperp})$ (the \emph{top} and \emph{bottom components} of $\alpha$), such that these form an (oriented) commutative square (on the left).
\[\vcenter{\xymatrix @R=7pt @ C=7pt{
A_{\ttop}  \ar[rr]^-{\alpha_{\ttop}}\ar[dd]_-{f_A} \ar@{{}{}{}}[rrdd]|-{(\rightarrow)}  & & B_{\ttop}  \ar[dd]^-{f_B} \\
\\
A_{\pperp}  \ar[rr]_-{\alpha_{\pperp}} &  & B_{\pperp}
}} \qquad  \qquad \vcenter{\xymatrix @R=0.6pt @ C=11pt{
A_{\ttop} \ar[rrr]^{\alpha_{\ttop}} \ar[rd] |-{p}  \ar[dddd]_{f_A} & & &  B_{\ttop}  \ar[dddd]|{f_B} \\
 & P  \ar[rru]|-{\pi_2} \ar[lddd]|-{\pi_1}  \\
\\
\\
A_{\pperp}  \ar[rrr]_{\alpha_{\pperp}} &  & & B_{\pperp}
}}
\]
We call the \emph{comparison map} of such a morphism (or commutative square) the unique map $p\colon{A_{\ttop} \to P}$ induced by the universal property of $P \defeq A_{\pperp} \times_{B_{\pperp}} B_{\ttop}$, the pullback of $\alpha_{\pperp}$ and $f_B$. Now from the study of the \emph{admissible adjunction} $F\dashv \I$ (within the Galois structure $\Gamma$), Galois theory produces the concept of a central extension ($\Gamma$-covering), and thus we may look at the \emph{full subcategory} $\CEXT\C$ of $\EXT\C$ whose objects are central extensions. The category of central extensions $\CEXT\C$ is not reflective, even less so admissible, in the category of extensions $\EXT\C$ in general (see \cite{JK1997}). In groups one can universally \emph{centralize} an extension, along a quotient of its domain, and there $\CEXT\C$ is actually a full replete \emph{(regular epi)-reflective} subcategory of $\EXT\C$. When such a reflection exists, one may further wonder whether there is a Galois structure behind it, and whether it is admissible. What is the sphere of influence of central extensions in extensions, and with respect to which class of \emph{extensions of extensions}, i.e.~can we re-instantiate Galois theory in this induced (two-dimensional) context? 

An appropriate class of morphisms to work with, in order to obtain an admissible Galois structure in such a two-dimensional setting, is the class of \emph{double extensions} \cite{Jan1991,CKP1993,Bou2003,EvGoeVdl2012}. A \emph{double extension} is a morphsim $\alpha=(\alpha_{\ttop},\alpha_{\pperp})$ in $\EXT\C$ such that both $\alpha_{\ttop}$ and $\alpha_{\pperp}$ are extensions and the comparison map of $\alpha$ is also an extension. Note that double extensions are indeed a subclass of regular epimorphisms in $\EXT\C$, provided $\C$ is a regular category (see \cite{BaGrOs1971}). \emph{Double central extensions of groups} were described in \cite{Jan1991}, and higher-dimensional Galois theory developed further \cite{Jan1995,EvGrVd2008}, leading to the aforementioned results in homology. 

Similarly in topology, higher homotopical information of spaces can be studied via the higher fundamental groupoids in the higher-dimensional Galois theory of locally connected topological spaces. A detailed survey about the study of higher-dimensional homotopy group(oid)s can be found in \cite{Bro1999}, see also \cite{BroHiSi2011}. Some insights are given in \cite{BroJan2004} where higher Galois theory is used to build a homotopy double groupoid for maps of spaces (see also \cite{BroJan1999}).

In this article we consolidate the understanding of the one-dimensional covering theory of racks and quandles, and introduce all the necessary ideas to start a higher-dimensional Galois theory in this context. In Part II we obtain an admissible Galois structure for the inclusion of coverings in extensions; we define and study double coverings, which are shown to describe the \emph{double central extensions of racks and quandles} (see \cite{Jan1991}). In Part III we generalize this to arbitrary dimensions.

\subsubsection{Admissibility via the strong Birkhoff condition, in two steps}\label{SectionAdmissibilityStronglyBirkhoff} Note that in the literature, most instantiations of higher categorical Galois theory are such that the ``base'' category $\C$ is a \emph{Mal'tsev category} (see \cite{CLP1991,CPP1992,CKP1993}), and such that moreover all the induced higher-dimensional categories of extensions ($\EXT\C$, $\EXT\EXT\C$, and so on) are also Mal'tsev categories. Admissibility conditions as well as computations with higher extensions are easier to handle in such a context. The categories we are interested in are not Mal'tsev categories. Showing how higher categorical Galois theory can apply in this more general setting thus requires some refinements on the arguments which are used in the existing examples.

The difficulty is in the induction for higher dimensions: the study of a given Galois structure is one thing, the study of which properties of a Galois structure induce good properties of the subsequent Galois structures in higher dimensions, is another. These subtleties will be discussed in Part III, see also \cite{Eve2015} for the most general example known to the author. In Part I, we lay down the necessary foundations for what comes in Part II and III. Let us sketch here, without technical details, which ingredients to focus on.

In Part I, our context is that of \cite{JK1994} which we refer to for more details. We look at the inclusion $\I\colon{\X \to \C}$ of $\X$, a \emph{full}, \emph{(regular epi)-reflective subcategory} of a \emph{finitely cocomplete} \emph{Barr-exact} category $\C$, such that $\X$ is \emph{closed under isomorphisms} and \emph{quotients}. In short \emph{Barr exactness} means that $\C$ has finite limits; every morphism factors uniquely, up to isomorphism, into a regular epimorphism, followed by a monomorphism, and these factorizations are stable under pullbacks; and, moreover, every \emph{equivalence relation} is the kernel pair of its \emph{coequalizer} \cite{Bar1971}. Here \emph{(regular epi)-reflectiveness} refers to the fact that the unit $\eta$ of the adjunction $F \dashv \I$ (with left adjoint $F\colon{\C \to \X}$) is a regular epimorphism (surjection), which implies (more generally) that $\X$ is also closed under subobjects. Finally, observe that (in general) \emph{monadicity} of $\I$ implies that $\X$ is also closed under \emph{finite limits} in $\C$. 

The fact that $\X$ is closed under quotients is then the remaining condition for $\X$ to be called a \emph{Birkhoff subcategory} of $\C$ \cite{BuSan1981,JK1994}. Given a more general Galois structure $\Gamma=(\C, \X, F, \I, \eta, \epsilon, \EE)$ such as in Convention \ref{RemarkGaloisStructure}, we say that $\Gamma$ is \emph{Birkhoff} if $\X$ is closed in $\C$ under quotients along extensions. In the Galois structures of interest (see for instance \cite{JK1994}), this condition is shown to be equivalent to the fact that the \emph{reflection squares} of extensions are \emph{pushouts}. Given $f\colon{A\to B}$ in $\C$, \emph{the reflection square at} $f$ (with respect to $\Gamma$) is the morphism $(\eta_A, \eta_B)$ with domain $f$ and codomain $\I F(f)$ in $\ARR(\C)$. Finally, $\X$ is said to be \emph{strongly Birkhoff} in $\C$ if moreover these reflection squares of extensions are themselves double extensions.
\begin{equation}\label{DiagramReflectionSquare} \vcenter{\xymatrix @R=0.6pt @ C=11pt{
A \ar[rrr]^{\eta_A} \ar[rd] |-{p}  \ar[dddd]_{f} & & &  \I F(A)  \ar[dddd]|{\I F(f)} \\
 & P  \ar[rru]|-{\pi_2} \ar[lddd]|-{\pi_1}  \\
\\
\\
B  \ar[rrr]_{\eta_B} &  & & \I F(B)
}}
\end{equation}
Proposition 2.6 in \cite{EvGrVd2008} implies that if $\Gamma$ is \emph{strongly Birkhoff}, then it is in particular admissible.

Now observe that in the Barr-exact context from above, Proposition 5.4 in \cite{CKP1993} implies that if $\Gamma$ is Birkhoff, it is strongly Birkhoff if and only if, for any object $A$ in $\C$, the kernel pair of $\eta_A$ \emph{commutes} with any other equivalence relation on $A$ (in the sense of \cite{MalSbo1954,CKP1993}). For instance, in the category of groups, any two equivalence relations commute with each other (see \emph{Mal'tsev categories} \cite{CKP1993}). Hence since $\AB$ is a Birkhoff subcategory of $\GRP$, it is actually strongly Birkhoff in $\GRP$, which implies the \emph{admissibility} of $\ab \dashv \I$ (see \cite[Theorem 3.4]{JK1994}). However, working in a Mal'tsev category is not necessary, as it was for instance noticed by V.~Even in \cite{Eve2014} and \cite{Eve2015}, where he uses the permutability property of the kernel pairs of unit morphisms to conclude the admissibility of his Galois structure. In Part I, we briefly re-discuss these results and illustrate the argument on a new adjunction. In higher dimensions, we shall also aim to obtain strongly Birkhoff Galois structures by splitting the work in two steps: (1) closure by quotients along higher extensions and (2) the permutability condition on the kernel pairs of (the non-trivial component of) the unit morphisms.

\subsubsection{Splitting along projective presentations and weakly universal covers}\label{SectionProjectivesAndWuc}
Remember that in any category, an object $E$ is \emph{projective} -- with respect to a given class of morphisms, which we always take to be our extensions -- if for any extension $f\colon{A \twoheadrightarrow B}$ and any morphism $p\colon{E \to B}$, there exits a factorization of $p$ through $f$ i.e.~$g\colon{E \to A}$ such that $f \circ g = p$. A projective presentation of an object $B$ is then given by an extension $p\colon{E \twoheadrightarrow B}$ such that $E$ is projective (with respect to extensions). For instance, in \emph{varieties of algebras} (in the sense of universal algebra), there are \emph{enough projectives}, i.e.~each object has a canonical projective presentation given by the \emph{counit} of the ``free-forgetful'' \emph{monadic} adjunction with sets \cite{McLane1997}. 

We may assume that in the Galois structures $\Gamma$ (as it is the case in groups or in Part II-III) the ``sophisticated context'' $\C$ has \emph{enough projectives}. Then any central extension (i.e. $\Gamma$-covering) $f$ is in particular split by any projective presentation $p$ of its codomain. We have
\begin{equation}\label{DiagramSplitByProjectivePresentations}
\xymatrix @R=1pt @C=7pt {
 E\times_{B} A  \ar[rr]^{p_{A}} \ar@{{}{..}{>}}[rd]_-{t}  \ar[ddd]_{p_E}  &  & A \ar[ddd]^{f}\\
 & T\times_{B} A   \ar[ddd]^(.36){p_T} \ar[ur]_{}   & \\
\\
 E \ar[rr]^(.4){p} \ar@{{}{--}{>}}[rd]_{p'}   &  &   B \\
 &   T   \ar[ur]_{}&  
}
\end{equation}
where $p'$ is induced by $E$ being projective, $t$ is induced by the universal property of $T\times_{B} A$ and $p_T$ is a trivial extension by assumption. Then with no assumptions on $\C$, the left hand face is a pullback since the back face and the right hand face are. Assuming that the Galois structure we consider is admissible; trivial, central and normal extensions are then pullback stable (see for instance \cite{JK1994}), and thus $p_E$ is a trivial extension, since it is the pullback of a trivial extension. 
Hence if $\C$ has enough projectives, then for any object $B$ in $\C$ the category of central extensions $\CEXT(B)$ above~$B$ is the same as the category of those extensions which are split by one given morphism such as the foregoing projective presentation $p$ of $B$.

Now when central extensions are reflective in extensions, a weakly universal central extension can always be obtained from the centralization of a projective presentation. One can for example recover this idea from \cite{Mil1971}. Consider an extension $f\colon{A \to B}$, and the centralization of a projective presentation of $B$:
\[\xymatrix@R=5pt@C=12pt{
 & E \ar[ddd]|(.3){p} \ar[ld]_-{\text{centralization}}  \ar@{{}{--}{>}}[rd]^-{a}  & \\
E' \ar[rdd]_-{p'}   \ar@/_2ex/@{{}{..}{>}}[rr]|(.35){b} &  & A  \ar[ldd]^-{f} \\
\\
& B & 
}\]
We get $a$ since $E$ is projective and $b$ by the universal property of $p'$. In the contexts of interest (see for instance Proposition \ref{PropositionSplitByCentralization}), a central extension is split by each weakly universal central extension of its codomain. Such weakly universal central extensions above an object $B$ are then split by themselves which makes them normal extensions. The reflection of the kernel pair of such is then the \emph{fundamental Galois groupoid of $B$}, which classifies central extensions above $B$.

\subsubsection{General strategy for characterizing central extensions} \label{SectionStrategy}
Finally we describe our general strategy, suggested by G. Janelidze, when it comes to identifying a property which characterizes central extensions. Observe that if a central extension $f$ is split by a split epimorphism $p$, then it is a trivial extension. \[ \xymatrix @R=1pt @C=11pt {
A \ar@{=}[rr]^{ } \ar@{{>}{-}{>}}[rd]_-{\bar{s}}  \ar[ddd]_{f}  &  & A  \ar[ddd]^{f}\\
 & E \times_{B} A   \ar[ddd]_(.35){\overline{f}} \ar[ur]_{\bar{p}}   & \\
\\
 B \ar@{=}[rr]^(.3){ } \ar@{{>}{-}{>}}[rd]_{s}   &  &   B \\
 &   E   \ar[ur]_{p}&  
}
\] Indeed, if the pullback $\overline{f}$ of $f$ along $p$ is a trivial extension by assumption, then the pullback of $\overline{f}$ along the splitting $s$ of $p$ is again a trivial extension and isomorphic to $f$. As a consequence, split epimorphic normal extensions are trivial extensions. Also, those central extensions that have projective codomains are trivial extensions. Now suppose one has identified a special class of extensions, called \emph{candidate-coverings}, such that candidate-coverings are preserved and reflected by pullbacks along extensions. Provided primitive extensions are candidate-coverings, then all trivial extensions are candidate-coverings and also central extensions are. Moreover, given a candidate-covering $f\colon {A\to B }$, pulling back $f$ along a projective presentation $p$ of $B$ yields a candidate-covering with projective codomain. Since $f$ is central if and only if it is split by such a $p$, we see that candidate-coverings are central extensions if and only if all candidate-coverings with projective codomains are actually trivial extensions, which is usually easier to check.

\section{An introduction to racks and quandles}\label{SectionBiasedintroToRacksAndQuandles}
We introduce all the fundamental ingredients of the theory of racks and quandles, which we describe and develop from the perspective inspired by the covering theory of interest.
\subsection{Axioms and basic concepts}\label{SectionAxiomsAndBasicConcepts}

\subsubsection{Racks and quandles as a system of symmetries}
Symmetry is classically modeled/studied using groups. Informally speaking: given a space $X$, one studies the group of automorphisms $\Aut(X)$ of $X$. In his PhD thesis \cite{Joy1979}, D.E.~Joyce describes quandles as another algebraic approach to symmetry such that, locally, each point $x$ in a space $X$ would be equipped with a global symmetry $S_x$ of the space $X$. Groups themselves always come with such a system of symmetries given by conjugation and the definition of inner automorphisms. Quandles, and more primitively racks, can be seen as an algebraic generalisation of such.

\subsubsection{Describing the algebraic axioms}
Consider a set $X$ that comes equipped with two functions
\[ \xymatrix{X \ar@<+1ex>[r]^-{\Ss} \ar@<-1ex>[r]_-{\Ss^{-1}} & X^X, }\]
which assign functions $\Ss_x$ and $\Ss^{-1}_x$ in $X^X$ (the set of functions from $X$ to $X$) to each element $x$ in $X$. Each element $x$ then acts on any other $y$ in $X$ via those functions $\Ss_x$ and $\Ss^{-1}_x$. By convention we shall always write actions on the right:
\[ y \cdot \Ss_x \defeq \Ss_x(y) \qquad \qquad y \cdot \Ss^{-1}_x \defeq \Ss_x^{-1}(y)  \]

The functions $\Ss_x$ and $\Ss_x^{-1}$ at a given point $x \in X$ are required to be inverses of one another, in particular for all $y$ in $X$ we have
\[
(y \cdot \Ss^{-1}_x) \cdot \Ss_x = y = (y \cdot \Ss_x) \cdot \Ss^{-1}_x.
\]
Note that, under this assumption, $\Ss^{-1}$ and $\Ss$ determine each other. Now we want to call such bijections $\Ss_x$ \emph{symmetries} (or \emph{inner automorphisms}) of $X$. But observe that the set $X$ is now equipped with two binary operations 
\[ \xymatrix{X\times X \ar@<+1ex>[r]^-{\qndop} \ar@<-1ex>[r]_-{\qndiop} & X,}\] 
defined by $x \qndop y \defeq x \cdot \Ss_y$ and $x \qndiop y \defeq x \cdot \Ss^{-1}_y $ for each $x$ and $y$ in $X$. Read ``$y$ acts on $x$ (positively or negatively)''. Automorphisms of $X$ should then preserve these operations. In particular we thus require that for each $x$, $y$ and $z$ in $X$:
\[
(x  \qndop y ) \qndop z = (x \qndop y) \cdot \Ss_z = (x \cdot \Ss_z) \qndop (y \cdot \Ss_z) = (x \qndop z) \qndop (y \qndop z).
\]

\subsubsection{Defining a rack}
Any set $X$ equipped with such structure, i.e.~two binary operations $\qndop$ and $\qndiop$ on $X$ such that for all $x$, $y$ and $z$ in $X$:
\begin{enumerate}
\item[\R1] $ (x \qndop y) \qndiop y = x = (x \qndiop y) \qndop y$; 
\item[\R2]$(x \qndop y ) \qndop z =(x \qndop z) \qndop (y \qndop z)$;
\end{enumerate}
is called a \emph{rack}. We write $\RCK$ for the category of racks with rack homomorphisms defined as usual (functions preserving the operations). 

We refer to the axiom \R2 as \emph{self-distributivity}. For each $x$ in $X$, the \emph{positive} (resp.~\emph{negative}) \emph{symmetry at $x$} is the automorphism $\Ss_x$ (resp.~$\Ss_x^{-1}$) defined before. A \emph{symmetry}, also called \emph{right-translation}, of $X$ is $\Ss_x$ or $\Ss_x^{-1}$ for some $x$ in $X$. The \emph{symmetries} of $X$ refers to the set of those.

\subsubsection{Racks from group conjugation}
One crucial class of examples is given by group conjugation. D.E.~Joyce describes quandles as ``the algebraic theory of conjugation'' \cite{Joy1979}. We have the functor:
\[ \xymatrix{\GRP \ar[r]^-{\Conj} & \RCK}, \]
which sends a group $G$ to the rack $\Conj(G)$ with same underlying set, and whose rack operations are defined by conjugation:
\[ x \qndop a \defeq a^{-1}xa \text{  and  } x \qndiop a \defeq axa^{-1}, \]
for $a$ and $x$ in $G$. Group homomorphisms are sent to rack homomorphisms by just keeping the same underlying function. The forgetful functor $\U \colon{\GRP \to \SET}$ thus factors through $\U \colon{ \RCK \to \SET}$ via $\Conj$. However the functor $\Conj$ is not \emph{full}, since given groups $G$ and $H$, there are more rack homomorphisms between $\Conj(G)$ and $\Conj(H)$ than there are group homomorphisms between $G$ and $H$. 

This peculiar ``inclusion'' functor consists in ``forgetting an operation'' in comparison with subvarieties which are about ``adding an equation''. When forgetting an operation, an obvious question is to ask: what equations should the remaining operations satisfy? Racks form one candidate theory. We will see that quandles (Subsection \ref{SubsectionIdempotencyAxiom}) give another option. In which sense is one different/better than the other? Can we characterize (as a subcategory) those racks which arise from groups? An important ingredient for answering those questions and understanding the relationship between groups, racks and quandles is the left adjoint of $\Conj$ (Subsection \ref{SubsectionGroupOfPath}). The thorough study and understanding of this left adjoint (first defined by D.E.~Joyce as $Adconj$, see also $\Adj$ in \cite{Eis2014}) is central to this piece of work, also with respect to its crucial role in the covering theory of racks and quandles. 

In what follows, we often consider groups as racks without necessarily mentioning the functor $\Conj$.  

\subsubsection{Other identities}\label{ParagraphOtherIdentities}
Note that for the symmetries $\Ss_x$ to define automorphisms of racks, one needs distributivity of $\qndop$ on $\qndiop$, distributivity of $\qndiop$ on $\qndop$, and self-distributivity of $\qndiop$. All these identities are induced by the chosen axioms. Besides, it suffices for a function $f$ to preserve one of the operations in order for it to preserve the other. We do not give a detailed survey of rack identities here. Bear in mind that in the theory of racks, the roles of $\qndop$ and $\qndiop$ are interchangeable. Swapping them in a given equation, gives again a valid equation. Finally we focus on an important characterization of \R2 using \R1:

\subsubsection{Self-distributivity}\label{ParagraphSelfDistr}
\begin{lemma}\label{LemmaSelfDistr}
Under the axiom \R1, the axiom \R2 is equivalent to
\begin{enumerate}
\item[\R{2'}]  $x \qndop (y \qndop z) = ((x \qndiop z) \qndop y ) \qndop z$. 
\end{enumerate}
\end{lemma}
\begin{proof} Given \R1, we formally show that
\begin{align*}
\text{\R2} \Rightarrow \text{\R{2'}:} \qquad x \qndop (y \qndop z)   &= ((x \qndiop z)\qndop z ) \qndop  (y \qndop z)  & \text{(by \R1)} & \\
									&= ((x \qndiop z) \qndop y ) \qndop z & \text{(by \R2)} & 
\end{align*}
\begin{align*}
\text{\R{2'}} \Rightarrow \text{\R{2}:} \qquad (x \qndop z) \qndop (y \qndop z)  &= (((x \qndop z)  \qndiop z) \qndop y) \qndop z   & \text{(by \R{2'})}& \\
									&= (x \qndop y) \qndop z   & \text{(by \R1)} & \qedhere
\end{align*}
\end{proof}
Similarly \R2 is also equivalent to  \R{2''}:  $\ x\qndop (y \qndiop z)  =((x \qndop z)\qndop  y) \qndiop z $. From the preceding discussion we also have
\[ x\qndiop (y \qndiop z)  =((x \qndop z)\qndiop  y) \qndiop z, \quad \text{and finally} \quad x\qndiop (y \qndop z)  =((x \qndop z)\qndop  y) \qndiop z . \]
Considering these as identities between \emph{formal terms in the language of racks} (see for instance Chapter II, Section 10 in \cite{BuSan1981}), we say that the term on the right-hand side is \emph{unfolded}, whereas the term on the left hand side isn't. 

\subsubsection{Composing symmetries -- inner automorphisms}\label{SubsubsectionInnerAutomorphisms}
By construction (see Paragraph \ref{ParagraphOtherIdentities}), given a rack $X$, the images of $\Ss$ and $\Ss^{-1}$ (defined as above) are in the group of automorphisms of $X$. Define the \emph{group of inner automorphisms} as the subgroup $\Inn(X)$ of $\Aut(X)$ generated by the image of $\Ss$. For each rack $X$, we then restrict $\Ss$ to the morphism 
\[ \xymatrix{X \ar[r]^-{\Ss} & \Inn(X).} \]

An inner automorphism is thus a composite of symmetries. Remember that we write actions on the right, hence we use the notation $z \cdot (\Ss_x \circ \Ss_y) \defeq \Ss_y(\Ss_x(z))$ for $x$, $y$, and $z$ in $X$. We use the same notation $\Ss$ for different racks $X$ and $Y$. Note that the construction of the group of inner automorphisms $\Inn$ does not define a functor from $\RCK$ to $\GRP$. It does so when restricted to surjective maps (see for instance \cite{BLRY2010}).

Observe that if $z = x \qndop y$ in $X$, then $\Ss_z = \Ss_y^{-1} \circ \Ss_x \circ \Ss_y$ by self-distributivity \R{2'}. The function $\Ss$ is actually a rack homomorphism from $X$ to $\Conj(\Inn(X))$. Again this describes a natural transformation in the restricted context of surjective homomorphisms.

Of course inner automorphisms of a group coincide with the inner automorphisms of the associated conjugation rack. However, observe that for a group $G$, a composite of symmetries is always a symmetry, whereas in a general rack, the composite of a sequence of symmetries does not always reduce to a one-step symmetry. Indeed, given $a$ and $b$ in a group $G$, then for all $x\in G$:
\[(x\qndop a) \qndop b = b^{-1}a^{-1} x ab = x \qndop (ab) \qquad \text{and, moreover,} \qquad x\qndiop a = x \qndop a^{-1}.\]
So, given a group $G$, the morphism $\xymatrix{G \ar[r]^-{\Ss} & \Conj(\Inn(G)) = \Inn(G)}$ is surjective.

\subsubsection{Acting with inner automorphisms -- representing sequences of symmetries}\label{SubsectionActionByInnerAut}
Given a rack $X$, we have of course an action of $\Inn(X)$ on $X$ given by evaluation. Elements of the group of inner automorphisms $\Inn(X)$ allow for a ``representation'' of successive applications of symmetries, seen as a composite of the automorphisms $\Ss_x$. 

More explicitly, any $g \in \Inn(X)$ decomposes as a product $g = \Ss_{x_1}^{\delta_n} \circ \cdots \circ \Ss_{x_n}^{\delta_1}$ for some elements $x_1,\ \ldots,\ x_n$ in $X$ and exponents $\delta_1,\ \ldots, \ \delta_n$ in $\lbrace -1,\ 1 \rbrace$. Such a decomposition is not unique, but for any $x$ in $X$ the action of $g$ on $x$ is well defined by 
\[x \cdot g \defeq x \cdot (\Ss_{x_1}^{\delta_n} \circ \cdots \circ \Ss_{x_n}^{\delta_1}) = x  \qndop^{\delta_1} x_1 \qndop^{\delta_2} x_2 \cdots\qndop^{\delta_n}  x_n , \]
where we omit parentheses using the convention that one should always compute the left-most operation first.

\paragraph{}\label{ParagraphConventionSymmetries}As we shall see, successive applications of symmetries play an important role in racks. For our purposes, using the group of inner automorphisms for their study is not satisfactory. Note that given $x \neq y$ in a rack $X$, two symmetries $\Ss_x$ and $\Ss_y$ are identified in $\Inn(X)$ if they define the same automorphism. From now on, and informally speaking, we consider the data of the base-point $x$ to be part of the data which we refer to as a \emph{symmetry} denoted $\Ss_x$ or $\Ss_x^{-1}$.

In what follows we study different ways and motivations to organize the set of symmetries into a group acting on $X$. Note that we may understand the definition of \emph{augmented quandles} (or racks) \cite{Joy1979}, see Paragraph \ref{ParagraphActionByInnerAutomorphisms}, as a tool to abstract away from ``representing'' sequences of symmetries via composites of such (in the sense of the group of inner automorphisms).

\subsubsection{Quandles, the idempotency axiom}\label{SubsectionIdempotencyAxiom}
As explained by D.E.~Joyce, it is reasonable (in reference to applications) to require that a symmetry at a given point fixes that point. If for each $x$ in a rack $X$ we have moreover that
\begin{enumerate}
\item[\Q1] $x\qndop x = x$; 
\end{enumerate}
then $X$ is called a \emph{quandle}. We have the category of quandles $\QND$ defined as before. Again, \Q1 is equivalent to  
\Q{1'}: $\ x\qndiop x = x$, under the axiom \R1. 

For the purpose of this article, we shall mainly be working in the more general context of racks since these exhibit all the necessary features for the covering theory of interest. Actually all concepts of centrality and coverings remain the same whether one works with the category of racks or of quandles. Directions for a systematic conceptual understanding of these facts will be provided. The addition of the idempotency axiom still has certain consequences on ingredients of the theory such as the fundamental groupoid or the \emph{homotopy classes of paths}. We shall always make explicit these differences and similarities, also using the enlightening study of the ``free-forgetful'' adjunction between racks and quandles. 

\subsubsection{Idempotency in racks}\label{ParagraphIdempotencyInRacks} An essential observation to make is that, even though \Q1 doesn't hold in each rack, a weaker version of the idempotency axiom still holds in all racks as a consequence of self-distributivity. Indeed, racks and quandles are very close -- which we shall illustrate throughout this article. The axiom \Q1 requires the $\qndop$ operations to be idempotent: $x \qndop x = x$. Now observe that in a rack $X$, such identities can be deduced by self-distributivity in ``the tail of a term'': given any $y$ and $x \in X$, we have
\[x \qndop (y \qndop y) = x \qndiop y \qndop y \qndop y  = x \qndop y.\]
The symmetries $\Ss_y$ and $\Ss_{(y\qndop y)}$, at $y$ and $y\qndop y$ are always identified in $\Inn(X)$, even when $y \neq (y \qndop y)$ in $X$. Similarly, for $x$ and $y$ in $X$ any chain $y \qndop^k y$ (for $k \in \Z$, the action of $y$ on $y$, repeated $|k|$ times -- use $\qndiop$ when $k<0$) is such that $ x \qndop (y \qndop^k y)= x \qndop y$. For more details, the left adjoint $\Frq \colon{\RCK \to \QND}$ to the inclusion $\I \colon {\QND \to \RCK}$ will be described in Section \ref{SubSectionFreeQuandleOnRack}. In what follows, the present comment translates in several different ways, such as in Example \ref{ExampleUnitFrqIsCovering} for instance.

\subsection{From axioms to geometrical features}\label{SectionFromAiomsToGeometricalFeatures}
\paragraph*{We informally highlight two additional elementary features of the axioms which play an important role in what follows. We then illustrate them in the characterization of the free rack on a set $A$}
\subsubsection{Heads and tails -- detachable tails}\label{ParagraphHeadsAndTails}
Observe that on either side of the identities defining racks, the \emph{head} $x$ of each term is the same and does not play any role in the described identifications.
\[ \text{ \R1 } \textcolor{Red}{x}\qndop y \qndiop y = \textcolor{Red}{x} = \textcolor{Red}{x} \qndiop y \qndop y \qquad \qquad  \text{ \R{2'} } \textcolor{Red}{x}  \qndop(y \qndop z) = \textcolor{Red}{x}\qndiop z \qndop y \qndop z \]
Now consider any \emph{formal term} in the language of racks (built inductively from atomic variables and the rack operations --  see Chapter II Section 10 in \cite{BuSan1981}), such as for instance
\begin{equation}\label{EquationFormalTermExample}
 (x \qndop y) \qndiop (\cdots ((a \qndop b) \qndiop  c) \qndop  d ) \cdots \qndop z.
\end{equation}
Remember that roughly speaking, the elements of the \emph{free rack} on a set $A$ can be constructed as equivalence classes of such formal terms, built inductively from the atomic variables in $A$, where two terms are identified if one can be obtained from the other by replacing subterms according to the axioms, or according to any provable equations derived from the axioms.

Given any term such as above, we shall distinguish the \emph{head} $x$ of the term from the rest of it which is called the \emph{tail} of the term. The informal idea is that the ``behaviour'' of the tail is independent from the head it is attached to. It thus makes sense to consider the tails (or equivalence classes of such) separately from the heads these tails might act upon.

Observe that the idempotency axiom plays a slightly different role in that respect since, although the heads of terms are left unchanged under the use of \Q1, the identifications in the tails of terms might depend on the heads these are attached to. We shall however see that the discussion about racks still lays a clear foundation for understanding the case of quandles which we discuss in Section \ref{SectionWorkingWithQuandles}.

\subsubsection{Tails as sequences of symmetries}\label{ParagraphHeadsAndTails2}
By Paragraph \ref{ParagraphSelfDistr}, acting with a symmetry of the form $\Ss_{(x\qndop y)}$ translates into successive applications of $\Ss_y^{-1}, \Ss_x, \Ss_y$ from left to right. 
\[\xymatrix@R=13pt@C=13pt{ \red\bullet \ar@{{}{-}{}}[r]|-{\dir{>}} ^-{\Ss_y^{-1}}  \ar@{{}{-}{}}[d]|-{\dir{>}} _-{\Ss_{x\qndop y}} & \bullet  \ar@{{}{-}{}}[d]|-{\dir{>}} ^-{\Ss_{x}} \\
\color{ForestGreen}\bullet & \bullet  \ar@{{}{-}{}}[l]|-{\dir{>}} ^-{\Ss_{y}}
} \]

Now consider any formal term such as in Equation~\eqref{EquationFormalTermExample} for instance. Using \R{2'} repeatedly, we may \emph{unfold} the tail of a term into a string of successive actions of the form
\[  x\qndop y \qndiop c \qndop  c \qndiop b \qndiop a \qndop b \qndiop c \qndop c \qndop d  \cdots \qndop z. \] 
We can then interpret the tail as a \emph{path} of successive actions of the symmetries which are applied to the head $x$. Using \R1 repeatedly again, we may also discard all possible occurrences of the successive application of a symmetry and its inverse 
\[  x\qndop y  \qndiop b \qndiop a \qndop b \qndop d  \cdots \qndop z. \] 
It is then possible to show that such \emph{unfolded} and \emph{reduced} terms provide normal forms (unique representatives) for elements in the free rack. The elements of a free rack on a set $A$ are thus described with this architectural feature of having a head in $A$ and an independent tail, such that the tail is a sequence of ``representatives'' of the symmetries which organize themselves as the elements of the free group on $A$.

\subsubsection{The free rack}\label{ParagraphTheFreeRack} The following construction can be found in \cite{FenRou1992}. It was also studied in \cite{Kru1998}.

Given a set $A$, the free rack on $A$ is given by
\[ \Fr(A) \defeq A  \rtimes \Fg(A) \defeq \lbrace (a,g) \mid g \in \Fg(A);\ a \in A \rbrace ,\]
where $\Fg(A)$ is the free group on $A$ and the operations on $\Fr(A)$ are defined for $(a,g)$ and $(b,h)$ in $A  \rtimes \Fg(A) $ by 
\[(a,g) \qndop (b,h) \defeq (a,gh^{-1}\gr{b}h) \qquad \text{and} \qquad (a,g) \qndiop (b,h) \defeq (a,gh^{-1}\gr{b}^{-1}h). \]
In order to distinguish elements $x$ in $A$ from their images under the injection $\eta^g_A \colon{A \to \Fg(A)}$, we shall use the convention to write \[\gr{a} \defeq \eta^g_A(a).\]

Looking for the unit of the adjunction, we then have the injective function which sends an element in $A$ to the trivial path starting at that element, i.e.~$\eta^r_A \colon{A \to \Fr(A)} \colon a \mapsto (a,e)$, where $e$ is the empty word (neutral element) in $\Fg(A)$.

Note that since any element $g \in \Fg(A)$ decomposes as a product $g = \gr{g_1}^{\delta_1} \cdots \gr{g_n}^{\delta_n} \in \Fg(A)$ for some $g_i \in A$ and exponents $\delta_i = 1$ or $-1$, with $1 \leq i \leq n$, we have, for any $(a,g) \in \Fr(A)$, a decomposition as
\[
 (a,g) = (a,\gr{g_1}^{\delta_1} \cdots \gr{g_n}^{\delta_n}) = (a,e) \qndop^{\delta_1} (g_1,e) \qndop^{\delta_2} (g_2,e) \cdots \qndop^{\delta_n}  (g_n,e). \]
 
As we discussed before, if we have moreover that $g_i = g_{i+1}$ and $\delta_i = - \delta_{i+1}$ for some $1 \leq i \leq n$, then
\begin{align*}
(a,e) \qndop^{\delta_1} &(g_1,e) \cdots \qndop^{\delta_{i-1}} (g_{i-1},e) \qndop^{\delta_{i}} (g_{i},e) \qndop^{\delta_{i+1}} (g_{i+1},e)\qndop^{\delta_{i+2}} (g_{i+2},e) \cdots \qndop^{\delta_n}  (g_n,e) = \\
&= (a,\gr{g_1}^{\delta_1} \cdots \gr{g_{i-1}}^{\delta_{i-1}} \gr{g_i}^{\delta_i}\gr{g_{i+1}}^{\delta_{i+1}}\gr{g_{i+2}}^{\delta_{i+2}} \cdots \gr{g_n}^{\delta_n})\\
 &= (a,\gr{g_1}^{\delta_1} \cdots \gr{g_{i-1}}^{\delta_{i-1}}\gr{g_{i+2}}^{\delta_{i+2}} \cdots \gr{g_n}^{\delta_n}) \\
 &= (a,e) \qndop^{\delta_1} (g_1,e) \cdots \qndop^{\delta_{i-1}} (g_{i-1},e) \qndop^{\delta_{i+2}} (g_{i+2},e) \cdots \qndop^{\delta_n}  (g_n,e) 
\end{align*}
which expresses the first axiom of racks, using group cancellation. 

From there we derive the universal property of the unit: given a function $f \colon A \to X$ for some rack $X$, we show that $f$ factors uniquely through $\eta^r_A$. Given an element $(a,g)\in \Fr(A)$, we have that for any decomposition $g =\gr{g_1}^{\delta_1} \cdots \gr{g_n}^{\delta_n}$ as above, we must have  
\[f(a,g) \ =\ f(a,\gr{g_1}^{\delta_1} \cdots \gr{g_n}^{\delta_n}) 
		\ =\ f\big( (a,e) \qndop^{\delta_1} (g_1,e) \cdots \qndop^{\delta_n}  (g_n,e) \big) 
		\ =\ f(a) \qndop^{\delta_1} f(g_1) \cdots \qndop^{\delta_n}  f(g_n)\]
which uniquely defines the extension of $f$ along $\eta^r_A$ to a rack homomorphism $f \colon{\Fr(A) \to X}$. This extension is well defined since two equivalent decompositions in $\Fr(A)$ are equivalent after $f$ by the first axiom of racks as displayed in Paragraph \ref{ParagraphTheFreeRack}.

The left adjoint $\Fr \colon {\SET \to \RCK}$ of the forgetful functor $\U \colon \RCK \to \SET$ with unit $\eta^r$ is then defined on functions $f\colon {A \to B}$ by 
\[\Fr(f) \defeq f \times  \Fg(f)  \colon {A \rtimes \Fg(A) \to B \rtimes \Fg(B)}.\]
This is easily seen to define a rack homomorphism. Functoriality of $\Fr$ and naturality of $\eta^r$ are immediate.

\paragraph{Terminology and visual representation}\label{ParagraphTerminologyAndVisualRepr}
In order to emphasise its visual representation, we call an element $(a,g) \in \Fr(A)$  a \emph{trail}. We call $g$ the \emph{path} (or \emph{tail}) component and $a$ the \emph{head} component of the trail $(a,g)$. It is understood that the path $g$ formally acts on $a$ to produce an \emph{endpoint} of the trail (see Paragraph \ref{ParagraphTheFreeRack}). Formally $(a,g)$ stands for both the trail and its endpoint:
\[\xymatrix{ a \ar@{{}{-}{}}[r]|-{\dir{>}} ^-{g} & (a,g).} \]
The action of a trail $(b,h)$ on another trail $(a,g)$ consists in adding, at the end of the path $g$, the contribution of the \emph{symmetry associated to the endpoint} of $(b,h)$ (see Subsection \ref{SubsectionCanonicalPresentation} and further). We say that a trail acts on another \emph{by endpoint}, as in the diagram below, where composition of arrows is computed by multiplication in the path component:
\begin{equation}\label{EquationActionByCodomainRacks}
\vcenter{\xymatrix@C=15pt @R=13pt {a \ar@{{}{-}{}}[dd]|{\dir{>}} _{g} \ar@{{}{}{}}[ddr]|{\qndop} & b \ar@{{}{-}{}}[dd]|{\dir{>}} _{h} \ar@{{}{}{}}[ddr]|{=} & a \ar@{{}{-}{}}[d]|{\dir{>}} ^-{g}    \\ 
 & & (a,g) \ar@{{}{-}{}}[d]|-{\dir{>}} ^-{h^{-1}\gr{b}h}  \\
(a,g) & (b,h) & (a,gh^{-1}\gr{b}h) 
}}
\end{equation}

\subsubsection{Canonical projective presentations}\label{SubsectionCanonicalPresentation}
Since $\RCK$ is a variety of algebras, any object $X$ can be canonically presented as the quotient
\[ \xymatrix@C=60pt{ \Fr \Fr  X \ar@<5pt>[r]^-{\Fr \epsilon^r_X}  \ar@<-5pt>[r]_-{\epsilon^r_{\Fr X}} & \Fr  X \ar[l]|-{\Fr\eta^r_{ X}} \ar[r]^-{\epsilon^r_X}  & X }\]
where we have omitted the forgetful functor $\U \colon{\RCK \to \SET}$ (understand $X$ alternatively as a rack or a set), and $\epsilon^r_X$ is the counit of the ``free-forgetful'' adjunction $\Fr \dashv \U$. This counit $\epsilon^r_X$ is the coequalizer of the reflexive graph on the left. This canonical presentation of racks allows us to capture a sense in which the geometrical features of free objects are carried through to any general rack. We shall illustrate this on the important functorial constructions of the Galois theory of interest. Let us make explicit these objects and morphisms to exhibit some of the mechanics at play. Think of what this \emph{right-exact fork} represents for groups, where the operation is associative. 

First of all we may exhibit heads and tails and rewrite this\emph{ right-exact fork} as
\[ \xymatrix@C=60pt{ (X \rtimes \Fg(X)) \rtimes \Fg(X \rtimes \Fg(X)) \ar@<5pt>[r]^-{\epsilon^r_X \times \Fg[\epsilon^r_X]}  \ar@<-5pt>[r]_-{\epsilon^r_{\Fr X}} & X \rtimes \Fg X \ar[l]|-{\Fr\eta^r_{ X}} \ar[r]^-{\epsilon^r_X}  & X }\]

Then it is immediate from Paragraph \ref{ParagraphTheFreeRack} that the counit $\epsilon^r_X$ should send a pair $(x,g)=(x,\gr{g_1}^{\delta_1} \cdots \gr{g_n}^{\delta_n})$ for $g_i \in X$ to the element in the rack $X$ given by:
\[\epsilon^r_X(x,g) = x\cdot g \defeq x \qndop^{\delta_1} g_1 \cdots \qndop^{\delta_n}  g_n.\] 
Hence the canonical projective presentation $\epsilon^r_X$ of a rack $X$ covers each element $x \in X$ by all possible formal decomposition $(x_0,g)$ of that element $x$, such that $x$ \emph{is the endpoint of the trail} $(x_0,g)$, i.e.~the result of the action of a \emph{path} on a \emph{head}: $x=x_0 \cdot g$. Now this head $x_0$ and each ``representative of a symmetry'' $\gr{g_i}^{\delta_i}$ in the path component $g=\gr{g_1}^{\delta_1} \cdots \gr{g_n}^{\delta_n}$ may itself be expressed as the endpoint of some trail (i.e.~$x_0 = x_{00} \cdot h$, and $g_i = y_i \cdot k_i$ for $h$ and $k_i$ in $\Fg X$). This is what is captured by the object $\Fr\Fr(X)$ on the left of the fork.

Then from the definition of the counit, we may derive the two projections. These may be understood as expressing two things:

First observe that an element $t=[(a,g); e]$ in $\Fr\Fr(X)$ (i.e.~an element which has a trivial path component, but an interesting head) is sent to $((a \cdot g), e)$ by the first projection and to $(a,g)$ by the second projection. 
The two projections thus allow us to move part of the tail of a trail towards the head of that trail and part of the head towards the tail. 

Then an element $[(a,e); \gr{(b,h)}]$ -- i.e.~an element with a trivial head component and a non trivial (but simple) tail -- is sent by the first projection to $(a,\gr{(b\cdot h)})$, and by the second projection to $(a,h^{-1}\gr{b}h)$. Coequalizing these two projections expresses self-distributivity (see Paragraphs \ref{ParagraphSelfDistr} and \ref{ParagraphHeadsAndTails2}). In other words it illustrates how to compute the representative of the symmetry associated to the endpoint of a trail. This is already part of the definition of the rack operation in the free rack. We have the rack homomorphism on the left
\[\vcenter{\xymatrix@R=10pt{X \rtimes \Fg(X) \ar[r]^-{i_X} & \Fg(X) \\
(x,g) \ar@{{|}{-}{>}}[r]^-{i_X} & g^{-1}\gr{x}g }} \qquad \qquad \vcenter{\xymatrix@C=15pt@R=15pt{ X \ar[rr]^-{\eta^g_{X}} \ar[rd]_-{\eta^r_{X}} & &  \Conj(\Fg(X)) \\
& \Fr(X) \ar@{-->}[ru]_-{i_X} &
}}\]
which sends a path to the symmetry associated to its endpoint. It is actually induced by the universal property of free racks as displayed in the diagram on the right.

\subsection{The connected component adjunction}\label{SectionConnectedComponentsAdjunction}

\subsubsection{Trivial racks and trivializing congruence}
Another important theoretical example of racks is given by the so-called \emph{trivial racks} (or trivial quandles) for which each symmetry at a given point is chosen to be the identity. Each point acts trivially on the rest of the rack. This may be expressed as an additional axiom:
\begin{enumerate}
\item[(Triv)] $ x \qndop y = x$. 
\end{enumerate} 

Since each set comes with a unique structure of trivial rack and each function between trivial racks is a homomorphism, we get an isomorphism between the category of sets ($\SET$) and the category of trivial racks. The category of sets is thus a subvariety of algebras within racks. 

The inclusion functor $\I \colon{\SET \to \RCK}$ sends a set to the trivial rack on that set. Now this inclusion functor should have a left adjoint which sends a rack to the \emph{freely trivialized} rack. Since trivial racks are those which satisfy (Triv), a good candidate for the trivialization of a rack $X$ is thus by quotienting out the congruence $\Co X$ generated by the pairs \[(x,x\qndop y).\] 

Using the comments of Section \ref{SectionFromAiomsToGeometricalFeatures}, it is not too hard to show that it actually suffices to consider the transitive closure of the set of pairs $\lbrace (x,x), (x,x\qndop y), (x,x\qndiop y) \mid x,y \in X\rbrace$ which gives the congruence $\Co X$ when endowed with the rack structure of the cartesian product. Symmetry and compatibility with rack operations are obtained for free. This further yields the concepts of \emph{connectedness} and \emph{primitive path} of Paragraph \ref{ParagraphConnectedComponents}. 

\begin{convention} For the purpose of this work, understand sets, or trivial racks, to be the \emph{0-dimensional coverings} of the covering theory of racks (and quandles), in the same way that abelian groups and central extensions of groups are respectively the 0-dimensional coverings and 1-dimensional coverings in groups. Similarly $\Co$ is the \emph{centralizing relation in dimension 0}. In Section \ref{SectionCoveringTheoryOfRacksAndQuandles} we study the subsequent 1-dimensional covering theory of racks and quandles.
\end{convention}

\subsubsection{Connectedness and primitive paths}\label{ParagraphConnectedComponents}
Given two elements $x$ and $y$ in a rack $A$, we say that $x$ and $y$ are \emph{connected} ($[x] = [y]$) if there exists $n \in \N$ and elements $a_1$, $a_2$, $\ldots$, $a_n$ in $A$ such that 
\[ y= x \qndop^{\delta_1}  a_1 \qndop^{\delta_2} a_2  \cdots\qndop^{\delta_n} a_n, \]
for some coefficients $\delta_i \in \lbrace -1,\, 1 \rbrace$ for $1 \leq i \leq n$. 

Such a sequence of elements together with the choice of coefficients is viewed as a formal sequence of symmetries (see Paragraph \ref{ParagraphConventionSymmetries}). Bearing in mind Paragraphs \ref{ParagraphHeadsAndTails} and \ref{ParagraphHeadsAndTails2}, we call such a formal sequence of symmetries $(a_i,\delta_i)_{1\leq i \leq n}$ a \emph{primitive path} of the rack $A$. In particular this specific primitive path \emph{connects} $x$ to $y$ but may be applied to different elements in the rack. We call the data of such a pair $T = (x,(a_i,\delta_i)_{1\leq i \leq n})$ a \emph{primitive trail} in $X$, where $x$ is the head of $T$ and $y$ the \emph{endpoint} of $T$.

We have that $(x,y)$ is in $\Co A$ if and only if there exists a primitive path which connects $x$ to $y$. For the sake of precision, and following the point of view of \cite{Joy1979}, let us take this as definition for $\Co A$. 

\subsubsection{Left adjoint $\pi_0$}
Then any rack homomorphism $f\colon {A \to X}$ for some trivial rack $X$ is such that $\Co A \leq \Eq(f)$ since given $y= x \qndop^{\delta_1}  a_1  \cdots\qndop^{\delta_n} a_n$ in $A$ we must have in $X$:
\[f(y)= f(x) \qndop^{\delta_1}  f(a_1)  \cdots\qndop^{\delta_n} f(a_n) = f(x).\]
Hence we define the functor $\pi_0 \colon{\RCK\to \SET}$ such that $\pi_0(A) \defeq A / (\Co A)$ is the set of connected components of $A$ (i.e.~the set of $\Co A$-equivalence classes) and $\pi_0 \dashv \I$ with unit 
\[ \xymatrix{A \ar[r]^-{\eta_A} & \pi_0(A),}\]
sending an element $a \in A$ to its connected component $\eta_A(a)$ (also denoted $[a]$) in $\pi_0(A)$. For any $f\colon {A \to X}$ as before, there is a unique function $f' \colon{\pi_0(A) \to X}$ defined on a connected component by the image under $f$ of any of its representatives.

\subsubsection{From free objects to all -- definition as a colimit}
Observe that the composite 
\[ \xymatrix{\SET \ar[r]^-{\I} & \RCK  \ar[r]^-{\U} & \SET } \]
gives the identity functor. As a consequence, the composite of left adjoints $\pi_0 \Frq$ also gives the identity functor. More precisely we may deduce from the composite of adjunctions that, given a set $X$, the unit $\eta_{\Fr(X)}\colon{ X \rtimes \Fg(X) \to X}$ is ``projection on $X$'', i.e.~the connected component of a trail $(x,g) \in \Fr(X)$ is given by projection on its head $x$.

Since $\pi_0$ is a left adjoint, it should preserve colimits, hence $\pi_0(X)$ should be the coequalizer, in $\SET$, of the pair:
\[ \xymatrix@C=60pt{ \pi_0((X \rtimes \Fg(X)) \rtimes \Fg(X \rtimes \Fg(X))) \ar@<5pt>[r]^-{\pi_0(\epsilon^r_X \times \Fg[\epsilon^r_X])}  \ar@<-5pt>[r]_-{\pi_0(\epsilon^r_{\Fr\U X})} & \pi_0(X \rtimes \Fg X), }\]
which indeed reduces to being the coequalizer of
\[ \vcenter{\xymatrix@C=60pt{ X \times \Fg(X) \ar@<5pt>[r]^-{p_1}  \ar@<-5pt>[r]_-{p_2} & X }}\qquad \text{where} \qquad
\vcenter{\xymatrix@R=1pt{p_1(x,\gr{g_1}^{\delta_1} \cdots \gr{g_n}^{\delta_n} ) = x \qndop^{\delta_1} g_1 \cdots \qndop^{\delta_n}  g_n; \\ p_2(x,\gr{g_1}^{\delta_1} \cdots \gr{g_n}^{\delta_n} ) = x.\qquad \qquad \qquad \quad}} \]
\subsubsection{Equivalence classes of primitive paths}\label{ParagraphEquivalenceClessesOfPrimitivePaths}
The term \emph{primitive path} is used to express the idea that it is the most unrefined way we shall use to acknowledge that two elements are connected. Literally it is just a formal sequence of symmetries.

As explained in Paragraph \ref{SubsectionActionByInnerAut}, inner automorphisms also ``represent'' sequences of symmetries. Again, each primitive path naturally reduces to an inner automorphism simply by composing all the symmetries in the sequence. We also have that $(x,y)$ is in $\Co A$ if and only if there exists $g \in \Inn(A)$ such that $x \cdot g = y$. In other words, $\Co A$ is the congruence generated by the action of $\Inn(A)$. We call it the \emph{orbit congruence} of $\Inn(A)$ (see Paragraph \ref{ParagraphOrbitCongruences}). In what follows, we like to view inner automorphisms as equivalence classes of primitive paths. As mentioned earlier we shall consider other such equivalence classes of primitive paths which lie in between formal sequences of symmetries and composites of such. Each of these represent different witnesses of how to connect elements in a rack. All of these generate the same trivializing congruence $\Co$.

\subsubsection{Conjugacy classes} Observe that for a group $\Conj(G)$, its set of connected components is given by the set of conjugacy classes in $G$. In this case the congruence $\Co(\Conj(G))$ is characterised as follows: $(a,b) \in \Co(\Conj(G))$ if and only if there exists $c \in G$ such that $b = c^{-1}ac$. Again, any primitive path, or sequence of symmetries, can be described via a single symmetry obtained as the symmetry of the product of the elements in the sequence.

Note that if $H$ is an abelian group, then $\Conj(H)$ is the trivial rack on the underlying set of $H$. More precisely the restriction to $\AB$ of the functor $\Conj$ yields the forgetful functor to $\SET$:
\[\xymatrix{ \AB \ar[rrr]^-{\Conj \text{ restricts to }\U} & & & \SET.} \]

\subsubsection{Racks and quandles have the same connected components}\label{ParagraphRckAndQndHvTheSameConnComp} The functor $\pi_0$ may be restricted to the domain $\QND$ and is then left adjoint to the inclusion functor $\I \colon {\SET \to \QND}$ by the same arguments as above. More precisely we have for any rack $X$ that $\pi_0 \Frq(X) = \pi_0(X)$, where $\Frq(X)$ is the free quandle on the rack $X$.

\subsubsection{Orbit congruences permute}\label{ParagraphOrbitCongruences} In order to obtain the admissibility of $\SET$ in $\QND$, V.~Even shows that certain classes of congruences commute with all congruences. As for quandles, we define \emph{orbit congruences} \cite{BLRY2010} as the congruences induced by the action of a normal subgroup of the group of inner automorphisms.
More precisely, if $X$ is a rack, and $N$ a normal subgroup of $\Inn(X)$ we shall write $\sim_N$ for the \emph{$N$-orbit congruence} defined for elements $x$ and $y$ in $X$ by: $x \sim_N y$ if and only if there exists $g \in N$ such that $x \cdot g = y$. As it is explained in \cite{Eve2015} (see Proposition 2.3.9), this is well defined and yields a congruence (also in $\RCK$).

We then have the following -- see \cite{EveGr2014} and \cite[Lemma 3.1.2]{Eve2015} for the proof, which also holds in $\RCK$.

\begin{lemma}\label{LemmaOrbitCongPermute}
Let $X$ be a rack, $R$ a reflexive (internal) relation on $X$ and $N$ a normal subgroup of $\Inn(X)$, then the relations $\sim_N$ and $R$ permute:
\[ \sim_N \circ R = R \circ \sim_N. \]
\end{lemma}

\subsubsection{Admissibility for Galois theory}\label{ParagraphAdmissibilityForGalois}
Of course the kernel pair of the unit $ \eta_X \colon { X \to \pi_0(X)}$ is an orbit congruence, since by Paragraph \ref{ParagraphEquivalenceClessesOfPrimitivePaths}, two elements are in the same connected component if and only if they are in the same orbit under the action of $\Inn(X)$.

As it was recalled in Section \ref{SectionAdmissibilityStronglyBirkhoff} (see also \cite{JK1994}), this yields Theorem 1 of \cite{Eve2014}:

\begin{proposition}
The subvariety $\SET$ is strongly Birkhoff and thus admissible in $\RCK$. Similarly for $\SET$ in $\QND$.
\end{proposition}

The Galois structure $\Gamma \defeq (\RCK, \SET, \pi_0,\I, \eta, \epsilon, \EE)$ (respectively $\Gamma^q \defeq (\QND, \SET, \pi_0,\I, \eta, \epsilon, \EE)$) (see \cite{JK1994}) where $\EE$ is the class of surjective morphisms of racks (respectively quandles), is thus admissible, i.e.~the study of Galois theory is relevant in this context and gives rise, in principle, to a meaningful notion of relative centrality. 

\subsubsection{Connected components are not connected}\label{ParagraphConnectedComponentsAreNotConnected}
Given an element $a$ in a rack $A$, we may consider its \emph{connected component} $\Cc_a$, i.e.~the elements of $A$ which are connected to $a$. The set $\Cc_a$ is actually a \emph{subrack} of $A$ as it is closed under the operations in $A$. We may construct the rack $\Cc_a$ as a pullback in $\RCK$:
\begin{equation}\label{EquationConnectedComponent}
\vcenter{\xymatrix@R=15pt{ \Cc_a \pullback \ar[r] \ar[d] & 1 \ar[d]^-{[a]} \\
A \ar[r]^-{ \eta_A} & \pi_0(A),
}}
\end{equation}
where $1 = \lbrace * \rbrace$ is the one element set, which is the terminal object in $\RCK$ and also the free quandle on the one element set. Note that if $A$ is connected, then by definition $\pi_0(A) = \lbrace \ast \rbrace$ and thus $\Cc_a = A$. However if $\Cc_a \subset A$, then $\Cc_a$ might have more than one connected component itself (i.e.~$\pi_0(\Cc_a)$ has cardinality $|\pi_0(\Cc_a)| > 1$), since the existence of a primitive path between some $c$ and $b$ in $\Cc_a$, might depend on elements which are not connected to $a$.
\begin{example}\label{ExampleaConnectedComponents}
A rack $A$ is called \emph{involutive} if the two operations $\qndop$ and $\qndiop$ coincide. The subvariety of \emph{involutive racks} is thus obtained by adding the axiom
\begin{enumerate}
\item[\Inv] $x \qndop y \qndop y = x$.
\end{enumerate}
We define the involutive quandle $Q_{ab\star}$ with three elements $a$, $b$ and $\star$ such that the operation $\qndop$ is defined by the following table (see $Q_{(2,1)}$ from \cite[Example 1.3]{Eis2014}).
\begin{center}
\begin{tabular}{c | c c c}
$\triangleleft$    & $a$ & $b$ & $\star$ \\
\hline
$a$ & $a$ & $a$ & $b$ \\
$b$ & $b$ & $b$ & $a$ \\
$\star$ & $\star$ & $\star$ & $\star$ \\
\end{tabular}
\end{center}
The connected component of $a$ is the trivial rack $\Cc_a = \lbrace a,\, b \rbrace$ which has itself two connected components $\lbrace a \rbrace$ and $\lbrace b \rbrace$.
\end{example}
We like to say that, for racks (and quandles) the notion of connectedness is not local. In categorical terms, we may say that the functor $\pi_0$ is not \emph{semi-left-exact} \cite{CaHeKe1985,CJKP1997}. This property is indeed characterised, in this context, by the preservation of pullbacks such as in Equation~\eqref{EquationConnectedComponent} above, i.e.~$\pi_0$ is semi-left-exact if and only if any such connected component ($\Cc_a$) is connected ($\pi_0(\Cc_a) = \lbrace \ast \rbrace$) (see for instance \cite{BorJan1994}). This is an important difference with the case of topological spaces for instance, where the connected components are connected and thus the corresponding $\pi_0$ functor is semi-left-exact. See also \cite{EveGr2014} for further insights on connectedness.

Finally note that the same comments apply to the context of $\QND$. Looking at \cite[Corollary 2.5]{dCosta2013}, we compute that $\pi_0(\Fr(1) \times \Fr(1)) = \Z$ and thus that $\pi_0\colon{\RCK \to \SET}$ does not preserve finite products; wheareas $\pi_0\colon{\QND \to \SET}$ does, as was shown in \cite[Lemma 3.6.5]{Eve2014}.

\subsubsection{Towards covering theory}

Knowing that $\Gamma$ is admissible, we may now wonder what is the ``sphere of influence'' of $\SET$ in $\RCK$, with respect to surjective maps, and start to develop the covering theory. Since $\SET$ is strongly Birkhoff in $\RCK$, trivial extensions (first step influence) are easy to characterize as those surjections which are ``injective on connected components'':

\begin{proposition}(See also \cite{Eve2014,Eve2015})
Given a surjective morphism of racks $t \colon{X \to Y}$, the following conditions are equivalent:
\begin{enumerate}[label=(\roman*)]
\item $t$ is a trivial extension;
\item $\Eq(t) \cap \Co X = \Delta_X$;
\item if $a$ and $b$ in $X$ are connected, then $t(a)=t(b)$ implies $a=b$.
\end{enumerate} 
\end{proposition}

Recall that the construction of inner automorphisms ($\Inn$) induces a functor on surjective morphisms: given a surjective morphism $t\colon{X \to Y}$, we write $\hat{t}$ or $\Inn(t)\colon{ \Inn(X) \to \Inn(Y)}$ for the induced homomorphism between the inner automorphism groups (see first two sections of \cite{BLRY2010}).

We may then also describe a trivial extension as an extension which \emph{reflects loops}:
trivial extensions are those extensions such that for any $a$ in $A$, if $g$ in $\Inn(A)$ is such that $t(a) \cdot \hat{t}(g) = t(a)$, then $a \cdot g=a$. 
\[ \xymatrix{ (a \ar@{{}{-}{}}[r]|-{\dir{>}} ^-{g} & a \cdot g)\quad  \ar@{{|}{-}{>}}[r]^-{t} & \quad t(a) = t(a \cdot g) \quad \ar@{{}{-}{}}@(ul,ur)[]|-{\dir{>}}^-{\hat{t}(g)}} \quad \Rightarrow  \xymatrix{\quad a = a \cdot g \quad \ar@{{}{-}{}}@(ul,ur)[]|-{\dir{>}}^-{g}} \] 

In what follows, we shall use such geometrical interpretations to make sense of the algebraic conditions of interest for the covering theory. However, the non-functoriality of $\Inn$ on general morphisms appears as a serious weakness (see for instance the need for Remark \ref{RemarkFunctorialityOfPth} in the proof of Proposition \ref{PropositionCoversingsAboveFreeObjects}). It will become clear from what follows that a more suitable way to represent sequences of symmetries is needed. This is achieved by the \emph{group of paths} which we motivate and describe in the next section. It is not a new concept, but our name for the left adjoint of the conjugation functor, which was described by D.E.~Joyce and then used by M.~Eisermann to construct weakly universal covers and an ad hoc fundamental groupoid for quandles. However, we provide a hopefully enlightening description of the construction and the role of this functor, which naturally arises from the geometrical features described in Section \ref{SectionFromAiomsToGeometricalFeatures}.

\subsection{The group of paths}\label{SubsectionGroupOfPath}

\subsubsection{Definition}
Consider a rack $X$ and two elements $x$ and $y$ in $X$ which are connected by a primitive path $\Ss^{\delta_1}_{x_1},\, \ldots,\, \Ss^{\delta_n}_{x_n}$: 
\[ x \cdot (\Ss^{\delta_1}_{x_1},\, \ldots,\, \Ss^{\delta_n}_{x_n}) \defeq x \qndop^{\delta_1} x_1 \cdots \qndop^{\delta_n} x_n = y. \]
Because of \R1, we discussed that it makes sense to identify such formal sequences so as to obtain elements of the free group on $X$. Now in the same way that we used Paragraph \ref{ParagraphSelfDistr} to unfold formal terms, we still have that whenever $x_i = b \qndop c$ for $1\leq i \leq n$ and $b$, $c$ in $X$, acting with $\Ss_{x_i}$ amounts to successively acting with $\Ss_c^{-1}$, $\Ss_b$ and $\Ss_c$. From a rack $X$ we may thus build the quotient:
\[ \xymatrix{ \Fg(X) \ar[r]^-{q_X} & \Pth(X) \defeq  \Fg(X)/\langle \gr{c}^{-1}\gr{a}^{-1}\gr{x}\,\gr{a} \mid a,x,c \in X \text{ and } c= x \qndop a \rangle,}\]
which is understood as a group of equivalence classes of primitive paths. Two primitive paths are identified in the group of paths if and only if one can be formally obtained from the other, using the identities induced by the graph of the rack operations (such as $c= x \qndop a$), as well as the axioms of racks (or more precisely the axiom-induced identities between tails of formal terms).

\subsubsection{Unit and universal property}\label{ParagraphUnitAndUniversalPropertyPth}
The function $\eta^g \colon {X \to \Fg(X)}$ composed with this quotient $q_X \colon{\Fg(X) \to \Pth(X)}$ yields a rack homomorphism
\[\xymatrix{ X \ar[r]^-{\pth_X} & \Conj(\Pth(X))}\]
which sends each element $x$ of $X$ to $\pth_X(x)$ in $\Pth(X)$, such that $\pth_X(x)$ ``represents'' the positive symmetry at $x$ in the same way $\Ss_x$ does in $\Inn(X)$ (see Paragraph \ref{ParagraphActionByInnerAutomorphisms}). As for the inclusion in the free group, we shall use the convention
 \[\gr{x} \defeq \pth_X(x).\]

Now given a rack homomorphism 
$f \colon {X \to \Conj(G)}$ for some group $G$, there is a unique group homomorphism $f'$ induced by the universal property of the free group, which, moreover, factors uniquely through the quotient $q_X \colon{\Fg(X) \to \Fg(X)/\langle (\gr{x\qndop a})^{-1}\gr{a}^{-1}\gr{x}\,\gr{a}\mid a,x \in X \rangle}$, since for any $a$ and $x$ in $X$, $f(x\qndop a) = f(a)^{-1}f(x)f(a)$ in $G$:
\[ \xymatrix{ \gray{X} \ar@[Gray][r]^-{\gray{\eta^g_{X}}} \ar@[Gray][rd]_-{\gray f} &  \Fg(X) \ar@{..>}[d]^-{\exists ! f'} \ar[r]^-{q_X} & \Pth(X) \ar@{-->}[ld]^{\exists ! \bar{f}} \\
& G &  
}\]
Hence, the construction $\Pth$ uniquely defines a functor which is the left adjoint of $\Conj$ with unit $\pth\colon{1_{\RCK} \to \Conj \Pth}$. As usual, given $f\colon{X\to Y}$ in $\RCK$, there is a unique morphism $\Pth(f)$, such that
\[ \xymatrix@R=14pt@C=35pt{ X \ar[r]^-{\pth_{X}} \ar[d]_-{f} &  \Conj(\Pth(X)) \ar@{..>}[d]^-{\exists ! \Conj(\Pth(f))} \\
Y  \ar[r]^-{\pth_{Y}} & \Conj(\Pth(Y)),
}\]
which defines the functor $\Pth$ on morphisms.
\begin{notation}\label{NotationPthF} In what follows, we write $\vec{f}$ for the image $\Pth(f)$ of a morphism $f$ from $\RCK$.
\end{notation}

\subsubsection{From free objects to all -- construction as a colimit}\label{ParagraphConstructionOfPthAsColimit}

Again, observe that the composite $\Pth \Fr$ is left adjoint to the forgetful functor $\U\colon{\GRP \to \SET}$, i.e.~$\Pth(\Fr(X)) = \Fg(X)$. More precisely, we may interpret $\pth$ as the extension to all objects of the functorial construction on free objects
\[i_X\colon{X \rtimes \Fg(X) \to \Fg(X) \  \colon \ (x,g) \mapsto g^{-1}\gr{x}g}\]
which sends a trail to the ``representative of the symmetry'' associated to its endpoint (Subsection \ref{SubsectionCanonicalPresentation}). Indeed, by the composition of adjunctions, as before, this $i$ is easily seen to define the restriction to free objects of the unit $\pth$ of the $\Pth \dashv \Conj$ adjunction:
\begin{equation}\label{EquationReprOfSymmInFreeRacks} \vcenter{\xymatrix@!0@C=63pt@R=25pt{ X \ar[rr]^-{\eta^g_{X}} \ar[rd]_-{\eta^r_{X}} & &  \Conj(\Fg(X)) \ar@{-->}[dd]^-{\exists ! \Conj(f'')} \\
& \Fr(X) \ar[rd]_-{\forall f} \ar[ru]|-{i_X = \pth_{\Fr(X)}} &  \\
 & & \Conj(G)
}} \quad \text{where} \quad 
 i_X(x,e) = i_X \eta^r_X (x) = \eta^g_X(x) = \gr{x}.
\end{equation}
Then since $\Pth$ is a left adjoint, $q_X\colon {\Fg(X) \to\Pth(X)}$ should be the coequalizer of the pair
\[ \xymatrix@C=60pt{ \Pth((X \rtimes \Fg(X)) \rtimes \Fg(X \rtimes \Fg(X))) \ar@<5pt>[r]^-{\Pth(\epsilon^r_X \times \Fg[\epsilon^r_X])}  \ar@<-5pt>[r]_-{\Pth(\epsilon^r_{\Fr\U X})} & \Pth(X \rtimes \Fg X) }\]
which, using $i$ above, we compute to be
\[ \xymatrix@C=60pt{ \Fg(X \times \Fg(X)) \ar@<5pt>[r]^-{p_1}  \ar@<-5pt>[r]_-{p_2} & \Fg(X) }\]
where $p_1$ and $p_2$ are defined by 
\[ p_1(x,g)= i_X(x \cdot g,e) = \eta^g_X(x \cdot g) = \gr{x \cdot g}\quad \text{ and }\quad p_2(x,g) = i_X(x,g)= g^{-1}\gr{x}g. \]

The universal property of the unit and definition on morphisms then follows easily as before. We insist on the tight relationship between the left adjoint $\Pth$ of the conjugation functor $\Conj$, and the geometrical features of the free racks as described in Subsection \ref{SectionFromAiomsToGeometricalFeatures}. 

We also use this detailed construction of $\Pth$ as a colimit, in the proof of Proposition \ref{PropositionKernelOfPthF}.

Finally, note that this pair $p_1, p_2$ is reflexive and thus from the coequalizer $q_X$ we also get the pushout $q_X,q_X \colon {\Fg(X) \rightrightarrows \Pth(X)}$ of $p_1$ and $p_2$. Even though the original fork in $\RCK$ is not necessarily a \emph{double extension}, the resulting fork in $\GRP$ is a \emph{double extension} (because $\GRP$ is an exact Mal'tsev category \cite{CKP1993}) i.e.~the comparison map $p\colon{\Fg(X \times \Fg(X)) \to \Eq(q_X)}$ to the kernel pair of the coequalizer $q_X$, is a surjection. 

\subsubsection{Action by inner automorphisms}\label{ParagraphActionByInnerAutomorphisms}

It is already clear from the construction of $\Pth$ that the group of paths $\Pth(X)$ acts on the rack $X$ ``via representatives of the symmetries''. For any $x$ and $y$ in $X$ we have
\[ x \cdot (\gr{y}) = x\qndop y, \]
which uniquely defines the action of any element in $\Pth(X)$.

Compare this action with the action by inner automorphisms: for each rack~$X$, the universal property of $\pth_{X}$ on $\Ss \colon{X \to \Inn(X)}$ (defined in Subsection \ref{SubsubsectionInnerAutomorphisms}) gives
\[ \xymatrix@R=15pt{ X \ar[r]^-{\pth_{X}} \ar[rd]_-{\Ss} & \Pth(X) \ar@{-->}[d]^-{s} \\
& \Inn(X), 
}\]
where we have omitted $\Conj$, and $s$ is the group homomorphism which relates the representatives of symmetries in $\Pth(X)$ to those in $\Inn(X)$. Then the action of $g \in \Pth(X)$ on $X$ is also uniquely described by the action of the inner automorphism $s(g)$. If preferred, the reader may use this as the definition of action by the group of paths. The morphism $s$ is called the \emph{excess} of $X$ in \cite{FenRou1992}. It is shown to be a central extension of groups in \cite[Proposition 2.26]{Eis2014}. Note that if $N \triangleleft \Pth(X)$ is a normal subgroup of $\Pth(X)$, then $s(N)$ is a normal subgroup of $\Inn(X)$. Hence the congruence $\sim_N$ induced by the action of $N$ on $X$ always defines an \emph{orbit congruence} ($\sim_N\ =\ \sim_{s(N)}$) in the sense of Paragraph \ref{ParagraphOrbitCongruences}.

We extend the concept of a \emph{trail} from Paragraph \ref{ParagraphTerminologyAndVisualRepr}.
\begin{definition} Given a rack $X$, a \emph{trail} (in $X$) is the data of a pair $(x,g)$ given by a \emph{head} $x \in X$ and a \emph{path} $g \in \Pth(X)$. The \emph{endpoint} of such a trail is then the element obtained by the action $x \cdot g$, of $g$ on $x$.
\end{definition}

In some sense, $\Pth(X)$ is the initial such group containing representatives of the symmetries of $X$ and acting via those symmetries on $X$ -- whereas $\Inn(X)$ is the terminal such. This can be described via the notion of an \emph{augmented rack}. Those are given by a group $G$ and a rack homomorphism $\iota \colon {X \to G}$ together with a right action of $G$ on $X$ such that for $g$, $h$ in $G$ and $x$, $y$ in $X$,
\begin{enumerate}\label{IdentitiesAugmentedRacks}
\item if $e$ is the neutral element in $G$, then $x \cdot e = x$;
\item $x \cdot (gh) = (x \cdot g) \cdot h$;
\item $(x\qndop y) \cdot g = (x \cdot g) \qndop (y \cdot g)$;
\item $\iota(x \cdot g) = g^{-1} \iota(x)  g$.
\end{enumerate}
Morphisms of such are as $s$ above. Now if we fix $X$, then $\pth_X \colon {X \to \Pth(X)}$ is initial amongst augmented racks (on $X$) whereas $\Ss \colon{X \to \Inn(X)}$ is terminal. This describes why $\Inn$ can be used as the reference to define such \emph{actions by representatives of the symmetries}, described as \emph{actions by inner automorphisms}. On the other hand, it also illustrates that $\Pth(A)$ is the \emph{freest way} to produce an augmented rack. Note the resemblance between the concept of an augmented rack and the concept of a \emph{crossed module} (see for instance \cite{McLane1997}). 

\begin{remark}\label{RemarkFunctorialityOfPth}
As mentioned before, $\Pth$ has the crucial advantage of functoriality, i.e.~for any morphism of racks $f \colon{X \to Y}$ (including non-surjective ones), and for any $x \in Y$, $g= \gr{g_1}^{\delta_1} \cdots \gr{g_n}^{\delta_n} \in \Pth(X)$, we have that 
\[ x \cdot (\vec{f}(g)) = x \cdot (\vec{f}(\gr{g_1}^{\delta_1} \cdots \gr{g_n}^{\delta_n})) = x \cdot (\gr{f(g_1)}^{\delta_1} \cdots \gr{f(g_n)}^{\delta_n})= x \qndop^{\delta_1} f(g_1) \cdots \qndop^{\delta_n} f(g_n). \]
\end{remark}

In the next paragraph, we observe that in the case of free objects $\Fr(X)$, these two constructions coincide ($\Pth(\Fr(X)) = \Inn(\Fr(X))$ is $\Fg(X)$) and, most importantly for what follows, they act freely on $\Fr(X)$ (also see \cite{FenRou1992,Kru1998}, where these results were first discussed). Our hope is that, in view of the preceding discussion, these results do not take the reader by surprise any more. 

\subsubsection{Free actions on free objects} By Paragraph \ref{ParagraphConstructionOfPthAsColimit}, and for any set $X$, the group of paths $\Pth(\Fr(X)) \cong \Fg(X)$ is freely generated by the elements \[\pth_{\Fr(X)}[\eta^r_X(x)] = \pth_{\Fr(X)}[(x,e)] = \gr{(x,e)} \] for $x \in X$.  Using the identification $\gr{(x,e)} \leftrightarrow \gr{x}$, for any element $(x,g)$ of $\Fr(X)$ and any word $h =\gr{h_1}^{\delta_1} \cdots \gr{h_n}^{\delta_n}$ in $\Pth(\Fr(X)) = \Fg(X)$, with $h_i \in X$ and $\delta_i \in \lbrace -1,\, 1 \rbrace$ for each $1 \leq i \leq n$, we have that
\[ (x,g) \cdot h = (x,g) \cdot (\gr{h_1}^{\delta_1} \cdots \gr{h_n}^{\delta_n}) = (x,g) \qndop^{\delta_1} (h_1,e) \cdots \qndop^{\delta_n} (h_n,e) =(x,gh). \]

\begin{proposition}
The action of $\Fg(X)=\Pth(\Fr(X))$ on $\Fr(X) = X \rtimes \Fg(X)$ corresponds to the usual $\Fg(X)$ right action in $\SET$
\[(X\times \Fg(X)) \times  \Fg(X)  \to X \times \Fg(X) \colon ((a,g),h) \mapsto  (a,g) \cdot h = (a,gh), \] 
given by multiplication in $\Fg(X)$. Such an action is free, since if $(a,hg)=(a,g)$, then $hg = g$ and thus $h=e$. 
\end{proposition}

Observe that $\Inn(\Fr(X))$ is generated as a group by the elements in the image of $\Ss \eta^r_X$. Indeed for each
\[(a,g) = (a,g_1^{\delta_1} \cdots g_n^{\delta_n}) = (a,e)  \qndop^{\delta_1} (g_1,e) \cdots \qndop^{\delta_n}  (g_n,e) = (a,e) \cdot g, \]
in $\Fr(A)$, as before, we have
\[ \Ss_{(a,g)} = \Ss_{(g_n,e)}^{-\delta_n} \cdots \Ss_{(g_1,e)}^{-\delta_1} \Ss_{(a,e)}\, \Ss_{(g_1,e)}^{\delta_1} \cdots \Ss_{(g_n,e)}^{\delta_n}; \]
see identity (4) from page \pageref{IdentitiesAugmentedRacks}: $\Ss_{(a,e) \cdot g} = g^{-1} \Ss_{(a,e)} g$.

We conclude that $\Inn(\Fr(X))$ is actually freely generated. Indeed, the group homomorphism \[s \colon { \Pth(\Fr(X)) = \Fg(X) \to  \Inn(\Fr(X))}\] defined in Subsection \ref{ParagraphActionByInnerAutomorphisms}, is such that:
\begin{itemize}
\item it is surjective, since the generating set $s(X) = \lbrace \Ss_{(x,e)} \mid x \in X\rbrace \subset \Inn(\Fr(X))$ is the image of $X \subset \Fg(X)$ by $s$;
\item it is injective, since $s(\gr{h_1}^{\delta_1} \cdots \gr{h_n}^{\delta_n}) = e$ for some $h_i \in X$ and $\delta_i \in \lbrace -1,\, 1 \rbrace$ for $1 \leq i \leq n$, if and only if 
\[(x,g) =(x,g) \cdot (\Ss^{\delta_1}_{(h_1,e)}\cdots \Ss^{\delta_n}_{(h_n,e)})= (x,g) \cdot (\gr{h_1}^{\delta_1} \cdots \gr{h_n}^{\delta_n}),\] for all $(x,g) \in \Fr(X)$, which implies that $\gr{h_1}^{\delta_1} \cdots \gr{h_n}^{\delta_n}= e$ since the action of $\Fg(X)$ is free.
\end{itemize}

\begin{proposition}
We may always identify $\Inn(\Fr(X))$, $\Pth(\Fr(X))$ and $\Fg(X)$ as well as their action on $\Fr(X)$, which is free. We refer to them as the \emph{group of paths of $\Fr(X)$}. 
\end{proposition}

\subsubsection{The kernels of induced morphisms $\vec{f}$}
In this section we introduce the results which we use to describe the relationship between the group of paths $\Pth$, and the central extensions (coverings) and centralizing relations of racks and quandles.

Our Lemma \ref{LemmaGroupCharacterization1} is only a slight generalization of a Lemma in \cite{BonSta2021}. We further generalize to \emph{higher dimensions} in Part II.

\begin{definition}\label{DefinitionK1}
Given a group homomorphism $f \colon{G \to H}$, and a generating set $A \subseteq G$ (i.e.~such that $G = \langle a \, \mid\, a \in A \rangle_G$), we define (implicitly with respect to $A$) 
\begin{enumerate}[label=(\roman*)]
\item two elements $g_a$ and $g_b$ in $G$ are \emph{$f$-symmetric (to each other)} if there exists $n\in \N$ and a sequence of pairs $(a_1,b_1)$, $\ldots$, $(a_n,b_n)$ in the set $(A\times A)\cap \Eq(f)$, such that \[ g_a = a_1^{\delta_1} \cdots  a_n^{\delta_n}, \quad \text{ and } \quad g_b = b_1^{\delta_1} \cdots  b_n^{\delta_n},\]
for some $\delta_i \in \lbrace -1,\, 1\rbrace$, where $1\leq i \leq n$. Alternatively say that $g_a$ and $g_b$ are an \emph{$f$-symmetric pair}.
\item $\K{f}$ is the set of \emph{$f$-symmetric paths} defined as the elements $g \in G$ such that $g = g_a g_b^{-1}$ for some $g_a$ and $g_b \in G$ which are $f$-symmetric to each other.
\end{enumerate}
\end{definition}

\begin{lemma}\label{LemmaGroupCharacterization1}
Given the hypotheses of Definition \ref{DefinitionK1}, the set of $f$-symmetric paths $\K{f} \subseteq G$ defines a normal subgroup in $G$. More precisely it is the normal subgroup generated by the elements of the form $ab^{-1}$ such that  $a$, $b \in A$, and $(a,b) \in \Eq(f)$:
\[ \K{f} = G_f \defeq \langle\langle ab^{-1} \mid (a,b) \in (A \times A) \cap \Eq(f) \rangle \rangle_G. \]
\end{lemma}
\begin{proof}
First we show that $\K{f}$ is a normal subgroup of $G$. Let $g_a$ and $g_b$ be $f$-symmetric (to each other). Observe that $g_b^{-1}$ and $g_a^{-1}$ are also $f$-symmetric, and thus $\K{f}$ is closed under inverses. Moreover, if $h_a$ and $h_b$ are $f$-symmetric, and $g=g_ag_b^{-1}$, $h=h_ah_b^{-1}$, then $gh=k_a k_b^{-1}$, with $k_a = h_ah_a^{-1}g_a$ and $k_b =h_bh_a^{-1} g_b$ which are $f$-symmetric. Finally since $A$ generates $G$, for any $k \in G$, $kg_a$ and $kg_b$ are $f$-symmetric to each other, and thus $kgk^{-1} \in \K{f}$ is an $f$-symmetric path.

Since the generators of $G_f$ are in the normal subgroup $\K{f}$, it suffices to show that $\K{f} \leq G_f$. Given an $f$-symmetric pair $g_a$ and $g_b$, we show that $g=g_ag_b^{-1} \in G_f$ by induction, on the minimum length $n_g$ of the sequences $(a_i,b_i)_{1\leq i \leq  n}$ in the set $(A\times A)\cap \Eq(f)$ such that $g_a = a_1^{\delta_1} \cdots  a_n^{\delta_n}$ and $g_b = b_1^{\delta_1} \cdots  b_n^{\delta_n}$ for some $\delta_i \in \lbrace -1,\, 1\rbrace$. If $n_g=1$, then $g$ is a generator of $G_f$. Suppose that $g=g_ag_b^{-1} \in G_f$ for all such $f$-symmetric pair with $n_g < n $ for some fixed $n \in \N$. Then given a pair $g_a = a_1^{\delta_1} \cdots  a_n^{\delta_n}$ and $g_b = b_1^{\delta_1} \cdots  b_n^{\delta_n}$ for some $(a_1,b_1)$, $\ldots$, $(a_n,b_n)$ in the set $(A\times A)\cap \Eq(f)$, and $\delta_i \in \lbrace -1,\, 1\rbrace$, we have that $h_a \defeq a_1^{-1} g_a$ and $h_b \defeq b_1^{-1} g_b$ are such that $h=h_ah_b^{-1} \in G_f$ by assumption. Moreover, $g = a_1 h a_1^{-1} a_1 b_1^{-1}$ which is a product of elements in $G_f$.
\end{proof}

\begin{observation}\label{ObservationPairingElementsInKernelWords}
Consider a function $f\colon {A \to B}$, and a word $\nu = a_1^{\delta_1} \cdots a_n^{\delta_n}$ with $a_i \in A$ and $\delta_i \in \lbrace -1,\, 1\rbrace$, for $1 \leq i \leq n$. This word represents an element $g$ in the free group $\Fg(A)$. As usual, a \emph{reduction} of $\nu$ consists in eliminating, in the word $\nu$, an adjacent pair $a_i^{\delta_i}a_{i+1}^{\delta_{i+1}}$ such that $\delta_{i} = -\delta_{i+1}$ and $a_{i} = a_{i+1}$. Every element $g \in \Fg(A)$ represented by a word $\nu$ admits a unique \emph{normal form} i.e.~a word $\nu'$ obtained from $\nu$ after a sequence of reductions, such that there is no possible reduction in $\nu'$, but $\nu'$ still represents the same element $g$ in $\Fg(A)$.

Suppose that $\nu$ represents an element $g$ which is in the kernel $\Ker(\Fg(f))$. The normal form of the word $f[\nu] \defeq f(a_1)^{\delta_1} \cdots f(a_n)^{\delta_n}$ (which represents $\Fg(f)(g) = e \in \Fg(B)$) is the empty word $\emptyset$, and thus there is a sequence of reductions of $f[\nu]$ such that the end result is $\emptyset$. From this sequence of reductions, we may deduce that $n=2m$ for some $m \in \N$ and the letters in the word (or sequence) $\nu$ organize themselves in $m$ pairs $(a_i^{\delta_i},a_j^{\delta_j})$ (the pre-images of those pairs that are reduced at some point in the aforementioned sequence of reductions) such that $i < j$, $f(a_i)= f(a_j)$, $\delta_i = - \delta_j$, each letter of the word $g$ appears in only one such pair and finally given any two such pairs $(a_i^{\delta_i},a_j^{\delta_j})$ and $(a_l^{\delta_l},a_m^{\delta_m})$, then $l < i$ (respectively $l> i$) if and only $m> j$ (respectively $m< j$),  i.e.~drawing lines which link those letters of the word $\nu$ that are identified by the pairing, none of these lines can cross. 
\[\xymatrix@R=20pt@C=4pt{ \\ 
a_1^{\delta_1}\ar@{{}{-}{}}@/^7ex/[rrrrr] & a_2^{\delta_2} \ar@{{}{-}{}}@/^4ex/[r] & a_3^{\delta_3} & a_4^{\delta_4} \ar@{{}{-}{}}@/^4ex/[r] & a_5^{\delta_5}  & a_6^{\delta_6} & a_7^{\delta_7} \ar@<1ex>@{{}{-}{}}@/^8ex/[rrrrrrr] & a_8^{\delta_8}\ar@{{}{-}{}}@/^4ex/[r]  & a_9^{\delta_9} & a_{10}^{\delta_{10}} \ar@{{}{-}{}}@/^6ex/[rrr] & a_{11}^{\delta_{11}} \ar@{{}{-}{}}@/^4ex/[r]  & a_{12}^{\delta_{12}} & a_{13}^{\delta_{13}} & a_{14}^{\delta_{14}}}\]
Given such a pairing of the letters of $\nu$, for each $k \in \lbrace 1,\ \ldots,\ n \rbrace$ we write $(a_{i_k}^{\delta_{i_k}},a_{j_k}^{\delta_{j_k}})$ for the unique pair such that either $i_k = k$ or $j_k=k$. Note that, conversely, any element $g$ in $\Fg(A)$ which is represented by a word $\nu$ which admits such a pairing of its letters, is necessarily in $\Ker(\Fg(f))$.
\end{observation}

Using this observation, we characterize the kernels of maps between free groups.

\begin{proposition}\label{LemmaKernelOfFgF}
Given a function $f\colon {A \to B}$, the kernel $\Ker(\Fg(f))$ of the induced group homomorphism $\Fg(f) \colon {\Fg(A) \to \Fg(B)}$ is given by the normal subgroup $\K{\Fg(f)}$ of $\Fg(f)$-symmetric paths (as in Definition \ref{DefinitionK1}): $\Ker(\Fg(f)) = \K{\Fg(f)}$.
\end{proposition}
\begin{proof}
The inclusion $\Ker(\Fg(f)) \supseteq \K{\Fg(f)}$ is obvious. Consider a reduced word $\nu = a_1^{\delta_1} \cdots a_n^{\delta_n} $ of length $n\in \N$ which represents an element $g$ in $\Fg(A)$ with $\delta_i \in \lbrace -1,\, 1\rbrace$, for $1 \leq i \leq n$ and suppose that $g \in \Ker(\Fg(f))$. Then the letters $a_k^{\delta_k}$ of the sequence (or word) $\nu \defeq (a_k^{\delta_k})_{1\leq k \leq n}$ organize themselves in pairs $(a_{i_k}^{\delta_{i_k}},a_{j_k}^{\delta_{j_k}})$ as in Observation \ref{ObservationPairingElementsInKernelWords}. Define the word $\nu' = b_1^{\delta_1} \cdots b_n^{\delta_n}$ such that for each $1 \leq k \leq n$, $b_k \defeq a_{i_k}$. Then by construction $\nu'$ represents an element $h$ which reduces to the empty word in $\Fg(A)$, so that $g = gh^{-1}$. Moreover, $g$ and $h$ form an $f$-symmetric pair, which shows that $g \in \K{\Fg(f)}$. \qedhere

\end{proof}

Finally the same characterization holds for kernels of maps $\Pth(f)=\vec{f}\colon{ \Pth(X) \to \Pth(Y)}$ induced by a surjective morphism of racks $f\colon{X \to Y}$.

\begin{proposition}\label{PropositionKernelOfPthF}
Given a surjective morphism of racks $f\colon {X \twoheadrightarrow Y}$, the kernel $\Ker(\vec{f})$ of the group homomorphism $\vec{f} \colon {\Pth(X) \twoheadrightarrow \Pth(Y)}$ is given by the normal subgroup $\K{\vec{f}}$ of $\vec{f}$-symmetric paths (as in Definition \ref{DefinitionK1}): 
\[ \Ker(\vec{f}) = \K{\vec{f}} =  \langle\langle ab^{-1} \mid (a,b) \in \Eq(f) \rangle \rangle_{\Pth(X)}.\]
\end{proposition}
\begin{proof}
From Subsection \ref{ParagraphConstructionOfPthAsColimit}, we reconstruct the image $\vec{f}$ as in the following diagram, where we also draw the kernels of $\Fg(f)$ and $\vec{f}$:
\[ \xymatrix@R=20pt@C=50pt{ & \Fg(X \times \Fg(X)) \ar@<+1ex>[d]^-{} \ar@<-1ex>[d]_-{} \ar@{{}{-}{>>}}[r]^-{\Fg(f \times \Fg(f))}  & \Fg(Y \times \Fg(Y)) \ar@<+1ex>[d]^-{} \ar@<-1ex>[d]_-{} \\
\Ker(\Fg(f)) \ar@{{}{--}{>}}[d]_-{k_1} \ar@{{>}{-}{>}}[r]^-{\ker(\Fg(f)} & \Fg(X) \ar@{{}{}{}}[rd]|-{(*)} \ar@{{}{-}{>>}}[d]|-{q_X} \ar@{{}{-}{>>}}[r]^-{\Fg(f)} & \Fg(Y)  \ar@{{}{-}{>>}}[d]|-{q_Y} \\
 \Ker(\vec{f})\ar@{{>}{-}{>}}[r]^-{\ker(\Fg(f)} & \Pth(X) \ar@{{}{--}{>>}}[r]^-{\vec{f}} & \Pth(Y). }\]
Since $q_X$ and $q_Y$ are the coequalizers of the pairs above (see Subsection \ref{ParagraphConstructionOfPthAsColimit} for more details), and the map $\Fg(f \times \Fg(f))$ is surjective, by Lemma 1.2 in \cite{Bou2003}, the square $(*)$ is a double extension (regular pushout), and thus the comparison map $k_1$ is surjective. Then $\Ker(\vec{f})$ coincides with the image $\ker{\Fg(f)}$ along of $q_X$, by uniqueness of (regular epi)-mono factorizations in $\GRP$. We may compute this image to be $\K{\vec{f}}$. Indeed, in elementary terms, any $g \in \Pth(X)$ such that $\vec{f}(g)=e$ can be ``covered'' by an element $h \in \Fg(X)$ such that $q_X(h)=g$ and $\Fg(f)[h]=e$ as well. Then by Lemma \ref{LemmaKernelOfFgF}, we have that $h = h_ah_b^{-1}$ for some $h_a$ and $h_b$ in $\Fg(X)$ which are $\Fg(f)$-symmetric to each other. The images $q_X(h_a)$ and $q_X(h_b)$ are then $\vec{f}$-symmetric to each other by commutativity of $(*)$, hence the quotient $g = q_X(h) = q_X(h_a) q_X(h_b)^{-1} \in \K{\vec{f}}$ is an $\vec{f}$-symmetric path. 
\end{proof}

\begin{notation}
For a morphism of racks $f$, we often write \emph{$f$-symmetric} (pair or path) instead of $\vec{f}$-symmetric (pair or path). A \emph{$f$-symmetric trail} $(x,g)$ is a trail with an $f$-symmetric path $g$. 
\end{notation}

\subsubsection{The left adjoint $\Pth$ is not faithful}
Observe that given a set $A$, the morphism 
\[\xymatrix{\Fr(A) \ar[rr]^-{i_A = \Pth_{\Fr(A)}} & & \Fg(A)},\]
is not injective. Indeed the elements $(a,\gr{a}g)$ and $(a,g)$ have the same image. We shall see that the kernel pair of $i_A$ yields the quotient producing the free quandle from the free rack.
Then the free quandle $\Fq(A)$ on the set $A$ embeds in the group $\Conj(\Fg(A))$, which is why Joyce calls quandles the \emph{algebraic theory of conjugation}. Observe, though, that not all quandles embed in a group.

\begin{example}\label{ExamplePthOfQab}
In the involutive quandle $Q_{ab\star}$ defined in Example \ref{ExampleaConnectedComponents}, the elements $\gr a$ and $\gr b$ are identified in $\Pth(Q_{ab\star})$. Indeed, $\gr a$ and $\gr b$ act trivially on $Q_{ab\star}$, hence they are in the center of the group $\Pth(Q_{ab\star})$. Moreover, $a$ and $b$  are in the same connected component, and thus they are also sent to conjugates in $\Pth(Q_{ab\star})$, which yields $\gr{a} = \gr{b}$. Note that from there we have $\Pth(Q_{ab\star}) = \Fg(\lbrace a,\, \star \rbrace)/\langle\langle a^{-1}\star^{-1}a\star  \rangle \rangle_{\Fg(\lbrace a,\, \star \rbrace)} = \Fa(\lbrace a,\, \star \rbrace) = \Z \times \Z$, where $\Fa$ is the free abelian group functor, and in $\Z \times \Z$, we have $\gr a = \gr b = (1,0)$ and $\gr {\star} = (0,1)$ (also see \cite[Proposition 2.27]{Eis2014}). 

In particular, the unit of the adjuntion $\Pth \dashv \Conj$ is not injective and $\Pth$ is not faithful (note that the right adjoint $\Conj$ is faithful, but not full). As a consequence $Q_{ab\star}$ is not a subquandle of a quandle in $\Conj(\GRP)$ since this would imply that $\pth_{Q_{ab\star}}$ is injective. We may also observe that a subquandle of a conjugation quandle is such that $(x \qndop y = x) \Leftrightarrow (y \qndop x = y)$.
\end{example}

\subsubsection{Racks and quandles have the same group of paths}\label{SubsectionRckAndQndHaveTheSameGrpOfPth} Observe that we may restrict $\Pth$ to the domain $\QND$. By the same argument $\Pth \I \colon {\QND \to \GRP}$ (which we denote $\Pth$) is then left adjoint to $\Conj\colon{\GRP \to \QND}$. We may conclude by uniqueness of left adjoints that if $\Frq $ is the left adjoint to the inclusion $\I\colon{\QND \to \RCK}$, then $\Pth\, \Frq  \cong \Pth \colon {\RCK \to \GRP}$. The adjunction between racks and groups factorizes into
\[\xymatrix@C=35pt@R=20pt{ \RCK \ar@<+3pt>@/^7pt/[rr]^-{\Frq}  \ar@<-3pt>@/_7pt/[rdd]|(.36){\Pth} \ar@{{}{}{}}[rr]|-{\perp}  & & \QND  \ar@<+3pt>@/^7pt/[ll]^-{\I} \ar@<-3pt>@/_7pt/[ddl]|(.6){\, \Pth\ \, \dashv} \\
\\
 & \GRP \ar@<-3pt>@/_7pt/[uur]|(.61){\Conj} \ar@<-3pt>@/_7pt/[uul]|(.4){\dashv\ \ \Conj\ }    &
}\]
in which all possible triangles of functors commute. Considering the comment of Paragraph \ref{ParagraphIdempotencyInRacks} about the idempotency axiom, we may want to rephrase this as follows: for each rack $X$, the quotient defining $\Pth(X)$ always identifies generators that would be identified in the free quandle on $X$.

More informally, considering the way $\Pth$, the left adjoint of $\Conj$, is constructed from equivalence classes of tails in the theory of racks, we may wonder in which sense racks could be a better context to study group conjugation. From the perspective of their respective covering theories, we further describe the relationship between groups, racks and quandles in what follows (see for instance Section \ref{SectionRelationshipToAbelianisation}).

\subsection{Working with quandles}\label{SectionWorkingWithQuandles}
We introduce the necessary material to make the transition from the context of racks to the context of quandles. See also the \emph{associated quandle} in \cite{FenRou1992}.
\subsubsection{The free quandle on a rack}\label{SubSectionFreeQuandleOnRack}
Remember from Paragraph \ref{ParagraphIdempotencyInRacks} that the idempotency axiom is a consequence of the axioms of racks ``for elements in the tail of a term''. In order to turn a rack into a quandle the identifications that matter are thus of the form 
\[   x \qndop^{\delta_x} x\qndop^{\delta_x} \cdots x\qndop^{\delta_x} x \qndop^{\delta_1} a_1 \cdots \qndop^{\delta_{n}}  a_{n} = x \qndop^{\delta_1} a_1 \cdots \qndop^{\delta_{n}}  a_{n},  \]
where a use of the idempotency axiom cannot be avoided. Now by self-distributivity of the operations, we may write $ y\defeq (x \qndop^{\delta_1} a_1 \cdots \qndop^{\delta_{n}}  a_{n})$, and then rewrite these identities as
\[  y \qndop^{\delta_x}  y \qndop^{\delta_x} \cdots  y \qndop^{\delta_x} y = y. \]
\begin{definition}
Given a rack $X$, define $Q_X$ as the relation (in $\SET$) defined for $(x,y) \in X \times X$ by $(x,y) \in Q_X$ if and only if $x = y \qndop^{k} y$ for some integer $k$ (see Paragraph \ref{ParagraphIdempotencyInRacks}), where $y \qndop^0 y \defeq y$.
\end{definition}
\begin{lemma}
Given a rack $X$, the relation $Q_X$ defines a congruence on $X$.
\end{lemma}
\begin{proof}
\begin{enumerate}
\item The relation $Q_X$ is reflexive by definition.
\item As aforementioned, for $x$ and $a$ in some rack, any chain $a \qndop^k a$ for some $k \in \Z$ is such that $x  \qndop (a \qndop^k a) = x \qndop a$. Hence $Q_X$ is symmetric since $b=a \qndop^k a$ implies that $b \qndop^{-k} b = b \qndop^{-k} a = a$.
\item Now $Q_X$ is transitive by self-distributivity.
\item And finally it is internal since if $a=b\qndop^k b$ and $c = d \qndop^l d$ then $a \qndop c = (b \qndop^k b) \qndop ( d \qndop^l d) = (b \qndop^k b) \qndop d = (b\qndop d) \qndop^k (b \qndop d)$. \qedhere
\end{enumerate}
\end{proof}

\begin{lemma}\label{KACharacterization}
Given a rack $X$, then a pair of elements $(x,y) \in X \times X$ is in the kernel pair $\Eq(^r\eta^q_X)$ of $^r\eta^q_X \colon { X \to \Frq(X)}$ if and only if $y = x \qndop^n x$ for some integer $n$, i.e.~$Q_X = \Eq(^r\eta^q_X)$.
\end{lemma}
\begin{proof}
Since $\RCK$ is a Barr-exact category \cite{Bar1971}, it suffices to show that the quotient of $X$ by the equivalence relation $Q_X$ (on the left) is the same as the quotient of $X$ by $\Eq(^r\eta^q_X)$ (on the right): \[\xymatrix{X \ar[r]^-q & X/Q_X }\qquad  \qquad \xymatrix{X \ar[r]^-{^r\eta^q_X} & \Frq(X).}\]
For this we show that $X/Q_X$ is a quandle and that $q$ has the same universal property as $^r\eta^q_X$. Indeed we have that $q(a) \qndop q(a) = q(a \qndop a) = q(a)$ since $(a, a\qndop a) \in Q_X$ for each $a$.
Finally observe that if $f \colon{X \to Q}$ is a rack homomorphism such that $Q$ is a quandle, then we necessarily have that $f$ coequalizes the projections $\pi_1,\ \pi_2 \colon{Q_X\rightrightarrows X}$ of the congruence $Q_X$. We then conclude by the universal property of the coequalizer.
\end{proof}

\subsubsection{Galois theory of quandles in racks}
We study the Galois structure $_r\Gamma_q \defeq $($\RCK$, $\QND$, $\Frq$, $ ^r\eta^q$, $^r\epsilon^q$, $\EE$) where $\EE$ is the class of surjective morphisms (see Section \ref{SectionCategoricalGaloisTheory} and \cite{JK1994}).

Since $\QND$ is a Birkhoff subcategory of $\RCK$, for $_r\Gamma_q$ to be admissible, it suffices to show that for each rack $X$ the kernel pair $\Eq(^r\eta^q_X)$ of the unit permutes with other congruences on $X$ (see Section \ref{SectionAdmissibilityStronglyBirkhoff}). Observe that this is not a consequence of Lemma \ref{LemmaOrbitCongPermute}.

\begin{lemma}\label{KACommute}
Given a rack $X$, then the congruence $Q_X = \Eq(^r\eta^q_X)$ commutes with any other internal relation $R$ on $X$.
\end{lemma}
\begin{proof}
We prove that a pair $(a,b) \in X \times X$ is in $\Eq(^r\eta^q_X) \, R$ if and only if it is in $R\, \Eq(^r\eta^q_X)$. As in Lemma \ref{LemmaOrbitCongPermute}, we show that if there is $c \in X$ such that $(a,c)$ is in one of these relations (say for instance $\Eq(^r\eta^q_X) \, R$) and $(c,b)$ in the other one, then there is a $c' \in X$ such that $(a,c')$ is in the latter ($R\, \Eq(^r\eta^q_X)$) and $(c',b)$ in the former ($\Eq(^r\eta^q_X) \, R$). Now observe that if $(x,y) \in R$, then $(x,y) \qndop^{k} (x,y) = (x\qndop^{k}x, y\qndop^{k} y) $ is in $R$ for any integer $k$. The result then follows from reading the following diagram for any $k \in \Z$, where horizontal arrows represent membership in $\Eq(^r\eta^q_X)$ and vertical arrows represent membership in $R$. Indeed from the top right corner below we construct the bottom left corner and the other way around:
\[ 
\xymatrix @C= 60pt@R=20pt {
c_1 \qndop^{-k} c_1 = a \ar@<+2pt>[r]^{\Ss_a^{k}}  \ar@{--}[d] & c_1 = a \qndop^{k} a  \ar@{--}[d] \ar@<+2pt>[l]^{\Ss_{c_1}^{-k}}\\
b \qndop^{-k} b  = c_2  \ar@<+2pt>[r]^{\Ss_{c_2}^{k}}  & b = c_2 \qndop^{k} c_2  \ar@<+2pt>[l]^{\Ss_b^{-k}}
}\]
where we use the fact that if $x = y \qndop^{k} y$ then $\Ss_x = \Ss_y$. Algebraically we read $(a,c_1) \in Q_X$ implies $c_1 \qndop^{-k} c_1 = a$ for some $k \in \Z$ and $(c_1,b) \in R$ implies $(c_1 \qndop^{-k} c_1,b \qndop^{-k} b) \in R$, thus choosing $c_2 = b \qndop^{-k} b$ yields one of the implications. The other direction translates similarly.
\end{proof}

\begin{remark}
Given a rack $X$, the congruence $Q_X$ is not an orbit congruence in general. For instance, observe that $Q_{\Fr(\lbrace a, b \rbrace)}$ contains the pairs $(a,a\qndop a)$ and $(b, b \qndop b)$. Suppose by contradiction that there is a normal subgroup $N \leq \Inn(\Fr(\lbrace a, b \rbrace)) = \Fg(\lbrace a, b \rbrace)$ for which $\sim_N = Q_{\Fr(\lbrace a, b \rbrace)}$. Then since $\Fg(\lbrace a, b \rbrace)$ acts freely on $\Fr(X)$, both inner automorphisms $\Ss_{a}$ and $\Ss_{b}$ need to be in $N$. This leads to a contradiction since $a \sim_N (a \qndop b)$ but $(a, a \qndop b) \centernot\in Q_{\Fr(\lbrace a, b \rbrace)}$. By contrast $Q_{\Fr(\lbrace \ast \rbrace)}$ is of course an orbit congruence.
\end{remark}

\begin{corollary}\label{CorollaryFrqStronglyBirkhoff}
Quandles form a strongly Birkhoff (and thus admissible) subcategory of $\RCK$.
\end{corollary}
\begin{proof}
By Proposition 5.4 in \cite{CKP1993}, the reflection squares of surjective morphisms are double extensions (see Section \ref{SectionAdmissibilityStronglyBirkhoff}). This implies the admissibility of the Galois structure $_r\Gamma_q$, for instance by \cite[Proposition 2.6]{EvGrVd2008}.
\end{proof}

Note that the left adjoint $\Frq$ is actually semi-left-exact as we may deduce from the fact that ``connected components are connected'' (see Paragraph \ref{ParagraphConnectedComponentsAreNotConnected}).

\begin{proposition}\label{PropositionSemiLeftExactnessFrq}
Any pullback of the form 
\[ \xymatrix@R=14pt@C=10pt{
C_a \pullback \ar[rr]^{p_2}  \ar[d]_{p_1} & & 1 \ar[d]^{[a]}\\
X \ar[rr]_{^r\eta^q_X}  & & \Frq(X),
}\]  
in $\RCK$, is preserved by the reflector $\Frq$, i.e.~$\Frq(C_a) = 1$, and thus $\Frq$ is \emph{semi-left-exact} in the sense of \cite{CaHeKe1985,CJKP1997}.
\end{proposition}
\begin{proof}
Observe that $X \times 1 \cong X$ and thus elements of the pullback $C_a$ are merely elements $x\in X$ such that that $\, ^r\eta^q(x) = [a] \in \Frq(X)$ i.e.~all elements $x$ and $y$ in $C_a$ are such that there is $k\in \Z$ such that $x = y \qndop^{k} y$. Hence by Lemma \ref{KACharacterization} the image of this pullback by $\Frq$ gives indeed $1$, which concludes the proof. 
\end{proof}

As a consequence we could have used \emph{absolute Galois theory} in this context \cite{Jan1990}. In order to not overload this article, we stick to the relative approach which we developed here.

Observe that there is a limit to the exactness properties satisfied by $\Frq$: we already saw in Paragraph \ref{ParagraphConnectedComponentsAreNotConnected} that $\Frq$ cannot preserve finite products, since $\pi_0\colon \QND \to \SET$ does but $\pi_0 \Frq \colon \RCK \to \SET$ does not. Moreover, since $\QND$ is an idempotent subvariety of $\RCK$, Proposition 2.6 of \cite{dCosta2013} induces that $\Frq$ does not have \emph{stable units} (in the sense of \cite{CaHeKe1985}).

To conclude, we show that, besides semi-left-exactness, the $\Frq$-covering theory is ``trivial'' in the sense that all surjections are $\Frq$-central. We use the general strategy which was stated in Section \ref{SectionStrategy}. Since the Galois structure is strongly Birkhoff, the ``first step influence'' is as usual:

\begin{lemma}
A surjective morphism $f \colon {X \to Y}$, in the category of racks, is $\Frq$-trivial if and only if $Q_X \cap \Eq(f) = \Delta_X$.
\end{lemma}
\begin{proof}
The morphism $f$ is trivial if and only if the reflection square at $f$ is a pullback (see Section \ref{SectionAdmissibilityStronglyBirkhoff}, Diagram \eqref{DiagramReflectionSquare}). Since this reflection square is a double extension, it suffices for the comparison map to be injective. Since the square is a pushout, the kernel pair of the comparison map is given by the intersection $Q_X \cap \Eq(f)$ of the kernel pairs of $q_X$ and $f$ respectively.
\end{proof}

\begin{proposition}\label{PropositionAllSurjAreFrqCentral}
All surjections $f \colon {X \to Y}$ in the category of racks are $\Frq$-central.
\end{proposition}
\begin{proof}
In order to show this, consider the canonical projective presentation $\epsilon^r_Y \colon {\Fr(UY) \to Y}$, and take the pullback of $f$ along $\epsilon^r_Y$. This yields a morphism \[\bar{f} \colon { X \times_{Y} \Fr(UY) \to \Fr(UY)}.\]

Now any morphism $ g\colon {X \to \Fr(Y)} $ with free codomain is $\Frq$-trivial since if $x = x \qndop^{k} x$ in $X$ for some integer $k$ and if, moreover, $f(x) = f(x) \qndop^{k} f(x)$ in $\Fr(Y)$, then $f(x)^k = e$ by the free action of $\Pth(\Fr(Y))$ on $\Fr(Y)$. However this can only be if $k=0$, which implies that $Q_X \cap \Eq(f) = \Delta_X$.
\end{proof}

\subsubsection{Towards the free quandle}\label{SubsectionFreeQuandle}
Given a set $A$, in order to develop a good candidate description for the free quandle on $A$ (see also \cite{Joy1979}), we may now consider $\Fq(A)$ as the free quandle on the rack $\Fr(A)$. As aforementioned and roughly speaking, the following identifications between terms:
\begin{equation}\label{EquationQuandleInRackIdentity}
x \qndop^{\delta_x} x\qndop^{\delta_x} \cdots x\qndop^{\delta_x} x \qndop^{\delta_1} a_1 \cdots \qndop^{\delta_{k}}  a_{k} = x \qndop^{\delta_1} a_1 \cdots \qndop^{\delta_{k}}  a_{k},  
\end{equation}   
define the relation $Q_{\Fr(A)}$ such that $\Fq(A) = \Fr(A)/Q_{\Fr(A)}$.

We want to select one representative $(a,g) \in A \rtimes \Fg(A)$ for each equivalence class determined by these identifications. Thinking in terms of trails, we observe that if $(a,g)$ and $(b,h)$ are identified, then they must have the same head $a=b$. We thus focus on the paths and use a clever semi-direct product decomposition of $\Fg(A)$. 

\paragraph{Characteristic of a path} We have the following commutative diagram in $\SET$,
\[ \xymatrix@R=15pt{ A \ar[r]^-{\eta^g_{A}} \ar[d]_-{\Cst} &  \Fg(A) \ar@{..>}[d]^-{\chi \defeq \Fg(\Cst)} \\
1 \ar[r]_-{\eta^g_{1}} & \Z = \Fg(1),
}\]
where $\Z$ is the underlying set of the additive group of integers, and the composite $\eta^g_1 \Cst$ is the constant function with image $1 \in \Z$. Given an element $g \in \Fg(A)$, there exists a decomposition $g = g_1^{\delta_1} \cdots g_n^{\delta_n}$ for some $g_i \in A$ and exponents $\delta_i = \lbrace -1, \, 1 \rbrace$, with $1 \leq i \leq n$. The characteristic function sums up the exponents $\chi(g) = \sum_{i=1}^n \delta_i$ (of course the result doesn't depend on the chosen decomposition of $g$). We may then classify paths in $\Fg(A)$ in terms of their characteristic (i.e.~their image by $\chi$). Looking at Equation~\eqref{EquationQuandleInRackIdentity}, two terms with same head, and same characteristic, that are moreover identified by $Q_{\Fr(A)}$, must actually be equal. In other words, given a fixed head $a$ each equivalence class $[(a,g)]$ in $\Fq(A)$ has only one representative $(a,g')$ such that the path $g'$ is of a given characteristic.

\paragraph{Characteristic zero and semi-direct product decomposition} The kernel of $\chi$ defines a normal subgroup $\Fgd(A) \leq \Fg(A)$ which is characterized (see \cite{Joy1979} and Lemma \ref{LemmaKernelOfFgF}) by 
\[\Fgd(A) = \langle ab^{-1} \mid a,b \in A \rangle_{\Fg(A)}.  \]
Then for each $a \in A$, we may identify $\Z$ with the subgroup $\langle a^n \mid n \in \Z \rangle \leq \Fg(A)$ which may be seen as the subgroup of $\Fg(A)$ which fixes $[(a,e)] \in \Fq(A) \defeq \Fr(A)/Q_{\Fr(A)}$. This then gives a splitting for $\chi$, on the left, yielding the split short exact sequence on the right:
\[\iota_a\colon{\Z \to \Fg(A)} \ \colon \ k \mapsto a^k \qquad \qquad \xymatrix{
\Fgd(A) \ar @{{>}{-}{>}}[r]^-{\nu_A} & \Fg(A) \ar@{{}{-}{>>}}[r]^-{\chi} & \Z \ar @{{>}{-}{>}} @<2pt>@/^2ex/[l]^-{\iota_a}
}\]

\paragraph{Characteristic zero representatives}
Given an element $a\in A$, any $g \in \Fg(A)$ decomposes uniquely as $\gr a^{\chi(g)} g_0$, where $g_0 = \gr a^{-\chi(g)}g$. This defines a function sending equivalence classes $[(a,g)] \in \Fq(A)$, to their representatives of characteristic zero $(a,g_0)$. Note that, for two different $a$ and $b$ in $A$, the construction of $g_0$ will vary, however elements of $\Fr(A)$ with different heads are always sent to different equivalence classes in $\Fq(A)$. 

\paragraph{Transporting structure} This function is indeed bijective, and thus we may transport the quandle structure from the quotient $\Fr(A)/Q_{\Fr(A)}$ to the set of representatives $A \times \Fgd(A)$. More explicitly we compute for $(b,h)$ and $(a,g)$ in $\Fr(A)$ that
\[ (a,g_0) \qndop (b,h_0) = (a, g_0 h_0^{-1}\gr b h_0), \]
where $w \defeq g_0 h_0^{-1}\gr b h_0$ is not of characteristic zero. We then want to take $w_0 =\gr a^{-1} g_0 h_0^{-1}\gr b h_0$ and define in $\Fq(A)$:
\[ (b,h_0) \qndop (a,g_0) \defeq (a,w_0). \]

\subsubsection{The free quandle}\label{ParagraphDefinitionFreeQuandle}
After this analysis, we may confidently build the free quandle (first described in \cite{Joy1979}) as follows.

Given a set $A$ the free quandle on $A$ is given by
\[ \Fq(A) \defeq A \rtimes \Fgd(A)\defeq \lbrace (a,g) \mid g \in \Fgd(A);\ a \in A \rbrace,\]
where the operations on $\Fq(A)$ are defined for $(a,g)$ and $(b,h)$ in $A \rtimes \Fgd(A)$ by 
\[(a,g) \qndop(b,h) \defeq (a,\gr a^{-1}gh^{-1}\gr bh)\quad \text{and} \quad (a,g) \qndiop(b,h) \defeq (a,\gr a^{-1}gh^{-1}\gr b^{-1}h). \]
As before, $g$ is the \emph{path} component and $a$ is the \emph{head} component of the so-called \emph{trail} $(a,g) \in \Fq(A)$ and we say that an element \emph{$(b,h)$ acts on an element $(a,g)$ by endpoint}. 
These operations indeed define a quandle structure.

From there, we translate all main results from the construction of free racks. Looking for the unit of the adjunction, we have the injective function $\eta^q_A \colon {A \to \Fq(A)}  :  a \mapsto (a,e)$.

Moreover, since any element $g \in \Fgd(A)$ decomposes as a product $g = \gr{g_1}^{\delta_1} \cdots \gr{g_n}^{\delta_n} \in \Fg(A)$ for some $g_i \in A$ and exponents $\delta_i \in \lbrace -1,\, 1 \rbrace$, with $1 \leq i \leq n$, and $\sum_i \delta_i = 0$, we have, for any $(a,hg) \in \Fq(A)$ with $g$ and $h \in \Fgd(A)$, a decomposition as
\begin{align*}
 (a,hg)   &= (a,h\gr {g_1}^{\delta_1} \cdots \gr {g_n}^{\delta_n}) 
 			= (a,\gr a^{\sum_i -\delta_i} h\gr {g_1}^{\delta_1} \cdots \gr {g_n}^{\delta_n}) 
 			= (a, \gr a^{-\delta_n}\cdots \gr a^{-\delta_1}h \gr {g_1}^{\delta_1} \cdots \gr {g_n}^{\delta_n})\\
 			&= (a,h) \qndop^{\delta_1} (g_1,e) \cdots \qndop^{\delta_n}  (g_n,e). 
\end{align*}
 
Observing that if $\gr{g_i}^{-\delta_i} = \gr{g_{i+1}}^{\delta_{i+1}}$ for some $(a,g) = (a,\gr{g_1}^{\delta_1} \cdots \gr{g_n}^{\delta_n}) \in \Fq(A)$ as above, then
\begin{align*}
 (a,e) \qndop^{\delta_1} & (g_1,e)  \cdots \qndop^{\delta_{i-1}} (g_{i-1},e) \qndop^{\delta_{i+2}} (g_{i+2},e)  \cdots \qndop^{\delta_n} (g_n,e)  = \\
&= (a,\gr{g_1}^{\delta_1} \cdots \gr{g_{i-1}}^{\delta_{i-1}}  \gr{g_{i+2}}^{\delta_{i+2}} \cdots \gr{g_n}^{\delta_n})
= (a,\gr{g_1}^{\delta_1} \cdots \gr{g_{i-1}}^{\delta_{i-1}} \gr{g_i}^{\delta_i} \gr{g_{i+1}}^{\delta_{i+1}} \gr{g_{i+2}}^{\delta_{i+2}} \cdots \gr{g_n}^{\delta_n}) \\
&= (a,e) \qndop^{\delta_1} (g_1,e) \cdots \qndop^{\delta_n} (g_n,e) ,
\end{align*}
which expresses the first axiom of racks, using group cancellation, as before. 

From there we derive the universal property of the unit: given a function $f \colon A \to Q$ for some quandle $Q$, we show that $f$ factors uniquely through $\eta^q_A$. Given an element $(a,g)\in \Fq(A)$, we have that for any decomposition $g =\gr{g_1}^{\delta_1} \cdots \gr{g_n}^{\delta_n}$ as above, we must have  
\[
f(a,g)\ =\  f(a,\gr{g_1}^{\delta_1} \cdots \gr{g_n}^{\delta_n}) 
		\ =\  f((a,e)  \qndop^{\delta_1} (g_1,e) \cdots \qndop^{\delta_n} (g_n,e) ) 
		\ =\  f(a)\qndop^{\delta_1}f(g_1) \cdots  \qndop^{\delta_n}  f(g_n)\]
which uniquely defines the extension of $f$ along $\eta^q_A$ to a quandle homomorphism $f \colon{\Fq(A) \to Q}$. This extension is well defined since equal such decompositions in $\Fq(A)$ are equal after $f$ by the first axiom of racks as displayed in Paragraph \ref{ParagraphDefinitionFreeQuandle}.

Finally the left adjoint $\Fq \colon \SET \to \RCK$ of the forgetful functor $\U \colon \RCK \to \SET$ with unit $\eta^q$ is then defined on functions $f\colon {A \to B}$ by 
\[\Fq(f) \defeq f \times \Fgd(f)  \colon {A \rtimes \Fgd(A)\to B \rtimes \Fgd(B)},\]
where $ \Fgd(f)$ is the restriction of $ \Fg(f)$ to the normal subgroup $ \Fgd(A) \leq  \Fg(A)$, whose image is in $\Fgd(B)$. This defines quandle homomorphisms. Also functoriality of $\Fq$ and naturality of $\eta^q$ are immediate.

\paragraph{Free action of $\Fgd(A)$}
Now remember the action by inner automorphisms of $\Fg = \Pth(\Fq(A))$ defined by the commutative diagram in $\SET$:
\[ \xymatrix@!0@R=23pt@C=65pt{ A \ar[rr]^-{\eta^g_{A}} \ar[rd]_-{\eta^q_{A}} & &  \Fg(A) \ar@{..>}[dd]^-{s} \\
& \Fq(A) \ar[ru]_-{\pth_{\Fq(A)}}  \ar[rd]_-{\Ss} &  \\
 & & \Inn(\Fq(A)),
}\]
where $s$ is the group homomorphism induced by the universal property of $\eta^g_A$ or equivalently that of $\pth_{\Fq(A)}$.

This action is \emph{not} in general given by left multiplication in $\Fgd(A)$, since in particular an $h$ in $\Fg(A)$ is of course not always of characteristic zero... However, from Paragraph \ref{ParagraphDefinitionFreeQuandle} we deduce that whenever $h\in \Fgd(A)$, the action of $h$ on an element $(a,g) \in \Fq(A)$  gives $(a,gh)$ as before. 
\begin{proposition}\label{CorollaryFreeActionOfFreeGroupQnd}
The action of $\Fgd(A)$ on $\Fq(A)$ given via the restriction 
\[ \xymatrix{ \Fgd(A) \ar[r]^-{s\degree} & \Inn\degree(\Fq(A)),}\]
of $s$ thus corresponds to the usual left-action of $\Fgd(A)$ in $\SET$: $(A \times \Fgd(A)) \times \Fgd(A)) \to A \times \Fgd(A)$, given by multiplication in $\Fgd(A)$. Such an action is free since if $(a,gh)=(a,g)$, then $gh= g$ and thus $h=e$. 
\end{proposition}

\subsubsection{The group of paths of a quandle} 
Observe that the construction of $\chi$ for the free group $\Fg(A) = \Pth(\Fr(A))$ generalizes to any rack $X$. The function $\Cst \colon {X \to 1}$ is actually a rack homomorphism to the trivial rack $1$. It thus induces a group homomorphism $\chi = \Pth(\Cst)$: 
\[ \xymatrix@R=15pt@C=40pt{ X \ar[r]^-{\pth_{X}} \ar[d]_-{\Cst} &  \Pth(X) \ar@{..>}[d]^-{\chi = \Pth(\Cst)} \\
1 \ar[r]_-{\pth_{1}} & \Z = \Pth(1).
}\]
As in the case of the free rack, we have the short exact sequence of groups:
\[\xymatrix{
\Pth\degree(X) \ar @{{>}{-}{>}}[r]^-{\nu_X} & \Pth(X) \ar@{{}{-}{>>}}[r]^-{\chi} & \Z = \Pth(1),
}\]
where $\nu_X\colon {\Pth\degree(X) \to \Pth(X)}$ is the kernel of $\chi$. This construction defines a functor $\Pth\degree\colon \RCK \to \GRP$. Most importantly it defines a functor $\Pth\degree\colon{\QND \to \GRP}$ which can be interpreted as sending a quandle to its group of equivalence classes of primitive paths, such that two primitive paths are identified if one can be obtained from the other with respect to the axioms defining quandles. In the same way that $\Pth$ describes homotopy classes of paths in racks, $\Pth\degree$ describes homotopy classes of paths in quandles, as it was already explained in \cite{Eis2014} and we shall rediscover in the covering theory described below.

\paragraph{The transvection group} As in the case of free groups, given a rack $X$, Proposition \ref{PropositionKernelOfPthF} implies that the kernel $\Pth\degree(X)$ of $\chi$ is characterized as the subgroup: 
\begin{equation}\label{EquationPthDegreeCharact}
\Pth\degree(X) = \langle \gr{a}\,\gr{b}^{-1}\, \mid\, a,\, b \in X \rangle_{\Pth(X)}, 
\end{equation}
which is the definition that was used by D.E.~Joyce in \cite{Joy1979}.
Then the restriction of the quotient $s\colon {\Pth(X) \to \Inn(X)}$ (defined in Subsection \ref{SubsectionActionByInnerAut}) yields the normal subgroup 
\[\Inn\degree(X) \defeq \langle  \gr{a}\,\gr{b}^{-1}\, \mid\, a,\, b \in X \rangle_{\Inn(X)},\]
which was called the \emph{transvection group} of $X$ by D.E.~Joyce.

This transvection group plays an important role in the literature. In the context of this work, we understand that the construction $\Pth\degree$ has better properties such as functoriality, and is of more significance to the theory of coverings than its image $\Inn\degree$ within inner automorphisms.

\paragraph{The case of free quandles} Observe that for a set $X$, $\Pth\degree(\Fq(X))=\Fgd(X)$ (for instance by Equation~\eqref{EquationPthDegreeCharact}). As in the case of free racks we get that:
\begin{proposition}
Given a set $A$, we may identify the groups $\Inn\degree(\Fq(A)) = \Pth\degree(\Fq(A)) = \Fgd(A)$, and their actions on $\Fq(A)$. We refer to them as the \emph{group of paths of $\Fq(A)$}. This group acts freely on $\Fq(A)$ by Corollary \ref{CorollaryFreeActionOfFreeGroupQnd}.
\end{proposition}
\begin{proof}
Given a set $A$, the morphism $s\degree \colon { \Fgd (A) \to  \Inn\degree(\Fq(A))}$ is a group isomorphism: 
\begin{itemize}
\item it is surjective, since $ \Inn\degree(\Fq(A))$ is generated by the  set $s(A)s(A)^{-1} = \lbrace \Ss_{(a,e)}(\Ss_{(b,e)})^{-1} \mid a,b \in A\rbrace \subset \Inn\degree(\Fq(A))$ which is the image of $AA^{-1}\subset \Fgd(A)$ by $s$;
\item it is injective, as before because of the free action of $\Fgd(A)$ via $s\degree$. \qedhere
\end{itemize}
\end{proof}

\paragraph{Inner automorphism groups}
In the case of quandles, the group of inner automorphisms $\Inn(\Fq(A))$ is not isomorphic to $\Fg(A)$ in general. However, the only counter-example is actually the case $A=\lbrace1\rbrace$: $\Fq(\lbrace1\rbrace) = \lbrace1\rbrace$ is the trivial quandle on one element and $\Inn(\lbrace1\rbrace)=\lbrace e\rbrace$ is the trivial group, whereas $\Fg(\lbrace1\rbrace)$ is $\Z$. Of course we do have $\Fgd(\lbrace1\rbrace) = \lbrace e\rbrace$. Now in all the other cases $\Inn(\Fq(A)) \cong \Fg(A)$. The case $A= \emptyset$ is trivial. Then whenever 
\[   x \qndop^{\delta_x} x\qndop^{\delta_x} \cdots x\qndop^{\delta_x} x  \qndop^{\delta_1} a_1 \cdots \qndop^{\delta_{k}}  a_{k}= x \qndop^{\delta_1} a_1 \cdots \qndop^{\delta_{k}}  a_{k},  \]
it suffices to pick $y\neq x \in A$ and then $y \qndop^{\delta_x} x \qndop^{\delta_x} x\qndop^{\delta_x} \cdots x \qndop^{\delta_1} a_1 \cdots \qndop^{\delta_{k}}  a_{k} \neq y \qndop^{\delta_1} a_1 \cdots \qndop^{\delta_{k}}  a_{k}$, showing that in $\Inn(\Fq(A))$: $\gr x^{\delta_x} \gr x^{\delta_x} \cdots \gr x^{\delta_x} \gr{a_{1}}^{\delta_{1}}  \cdots \gr{a_k}^{\delta_k} \neq \gr{a_{1}}^{\delta_{1}}  \cdots \gr{a_k}^{\delta_k}$, just as in $\Inn(\Fr(A))$.

\section{Covering theory of racks and quandles}\label{SectionCoveringTheoryOfRacksAndQuandles}

In this section we study the relative notion of centrality induced by the sphere of influence of $\SET$ in $\RCK$, with respect to extensions (surjective homomorphisms). Remember that pullbacks of primitive extensions (surjections in $\SET$) along the unit $\eta$ induce the concept of trivial extensions, which we saw are those extensions which reflect \emph{loops}. Central extensions in $\RCK$ are those from which a trivial extension can be reconstructed by pullback along another extension. Equivalently, central extensions are those extensions whose pullback, along a projective presentation of their codomain, is trivial. In Section \ref{SectionOneDimensionalCoverings} we thus look for a condition (C) such that, if a surjective rack homomorphism $f\colon{A \to B}$ satisfies (C), then the pullback $t$ of $f$ along $\epsilon^r_B\colon{\Fr(B) \to B}$ reflects loops (see Section \ref{SectionCategoricalGaloisTheory} and references there).

\subsection{One-dimensional coverings}\label{SectionOneDimensionalCoverings}

Quandle coverings were defined in \cite{Eis2014}, and shown to characterize $\Gamma_q$-central extensions of quandles in \cite{Eve2014}. We give the same definition for rack coverings (already suggested in M.~Eisermann's work), which we then characterize in several ways. In Section \ref{SectionCharacterizingCentralExtensions} we further show that these are exactly the central extensions of racks.

Remember that in dimension zero, a rack $A$ is actually a set, if \emph{zero-dimensional data}, i.e.~an element $a \in A$, acts trivially on any element $x \in A$~: $x \qndop a = x$.
We saw that this may be expressed by the fact that $\Pth(A)$ acts trivially on $A$ or alternatively by the fact that any two elements which are connected by a primitive path are actually equal.

Now in \emph{dimension one}, an extension $f\colon{A \twoheadrightarrow B}$ is a covering if \emph{one-dimensional data}, i.e.~a pair $(a,b)$ in the kernel pair of $f$, acts trivially on any element $x \in A$:
\begin{definition}\label{DefinitionRackCovering}
A morphism of racks $f \colon{A \to B}$ is said to be a \emph{covering} if it is surjective and for each pair $(a,b) \in \Eq(f)$, and any $x\in A$ we have 
\[ x \qndop a \qndiop b = x. \]
\end{definition}

Of course a trivial example is given by surjective functions between sets (the primitive extensions). The following implies that central extensions are coverings:
\begin{lemma}\label{LemmaCoveringsAndPullbacks}
Coverings are preserved and reflected by pullbacks along surjections in $\RCK$.
\end{lemma}
\begin{proof}
Same proof as in \cite{Eve2015} see also \cite{Eve2014}.
\end{proof}

\subsubsection{Coverings and the group of paths} Observe that given data $f$, $x$, $a$ and $b$, such as in Definition \ref{DefinitionRackCovering}, we have in particular that $x \qndiop a = x\qndiop a \qndop a \qndiop b = x\qndiop b$.  In fact we can easily deduce that $f$ is a covering if and only if for all such $x$, $a$ and $b$ as before
\[ x \qndiop a \qndop b = x. \]
This is to say that $f$ is a covering if and only if any path of the form $\gr{a}\,\gr{b}^{-1}$ or $\gr{a}^{-1}\gr{b} \in \Pth(A)$, for $a$ and $b$ in $A$, such that $f(a)=f(b)$, acts trivially on elements in $A$. But then $f$ is a covering if and only if the subgroup of $\Pth(A)$ generated by those elements acts trivially on elements of $A$. Now, given $g\in \Pth(A)$, if $z \cdot g = z$ for all $z$ in $A$, then also $x \cdot a^{-1} \cdot g \cdot a = (x \qndiop a) \cdot g \cdot a = (x\qndiop a) \cdot a = x$ for all $a \in A$. Hence we conclude that $f$ is a covering if and only if the normal subgroup $\langle\langle ab^{-1} \mid (a,b) \in \Eq(f) \rangle \rangle_{\Pth(A)}$ acts trivially on elements of $A$. Finally by Proposition \ref{PropositionKernelOfPthF} we get the following result which illustrates the importance of $\Pth$ in the covering theory of racks and quandles.

\begin{theorem}\label{PropositionCoveringsAsOrbitCongruences}
Given a surjective morphism $f\colon {A \to B}$ in $\RCK$ (or in $\QND$), the following conditions are equivalent:
\begin{enumerate}
\item $f$ is a covering;
\item the group of $\vec{f}$-symmetric paths $\K{\vec{f}}$ acts trivially on $A$ (as a subgroup of $\Pth(A)$) -- i.e. any $f$-symmetric trail loops in $A$;
\item $\Ker(\vec{f})$ acts trivially on $A$ (as a subgroup of $\Pth(A)$);
\item $\Ker(\vec{f})$ is a subobject of the kernel $\Ker(s)$, where $s\colon{\Pth(A) \to \Inn(A)}$ is the canonical quotient described in Paragraph \ref{ParagraphActionByInnerAutomorphisms}.
\end{enumerate} 
\end{theorem}
\begin{proof}
The statements $(1)$, $(2)$ and $(3)$ are equivalent by the previous paragraph (and thus by Proposition \ref{PropositionKernelOfPthF}). Statement $(4)$ is merely a way to rephrase $(3)$ using the fact that elements of the inner automorphism groups are defined by their action. 
\end{proof}

As it was observed by M.~Eisermann in $\QND$, we have:
\begin{corollary}\label{CorollaryCoveringInducedAction} A rack covering $f\colon{A \to B}$ induces a surjective morphism $\bar{f} \colon{\Pth(B) \to \Inn(A)}$ such that $\vec{f} \bar{f} = s$ and thus induces an action of $\Pth(B)$ on $A$ given for $g_B \in \Pth(B)$ and $x \in A$ by $x \cdot g_B \defeq x \cdot g_A$, where $g_A$ is any element in the pre-image $\vec{f}^{-1}(g_b)$.
\end{corollary}

Observe that an easy way to obtain a rack covering is by constructing a quotient $f \colon {A \twoheadrightarrow B}$ such that $\vec{f}$ is an isomorphism.

\begin{example}\label{ExampleUnitFrqIsCovering}
The components of the unit $^r\eta^q$ of the $\Frq$ adjunction are rack coverings. Indeed, we discussed in Paragraph \ref{SubsectionRckAndQndHaveTheSameGrpOfPth} that $\Pth \Frq = \Pth$, also see Paragraph \ref{ParagraphIdempotencyInRacks}. In particular, we look at the one element set $1$ and consider the map $f \defeq\, {^r\eta^q_{\Fr(1)}} \colon {\Fr(1) \to \Fq(1) = 1}$. We then compute that $\vec{f} = \Pth(^r\eta^q_{\Fr(1)})$ and $\Inn(f) = \Inn(^r \eta^q_{\Fr(1)})$ are respectively the morphisms \[ \xymatrix{\Pth(\Fr(1)) = \Z \ar[r]^-{\id_{\Z}}  & \Pth(\Fq(1)) = \Z} \text{ and } \xymatrix{\Inn(\Fr(1)) = \Z_3 \ar[r]^-{} & \Inn(\Fq(1)) = \lbrace e \rbrace}, \] where $\Z$ is the infinite cyclic group, $\Z_3 = \Z/3\Z$ is the cyclic group with $3$ elements and $\lbrace e \rbrace$ the trivial group. In this case $\vec{f}$ is an isomorphism, but $\Inn(f)$ is not.
\end{example}

\begin{remark} In the article \cite{BLRY2010}, Theorem 4.2 says that quandle coverings (such as in $(3)$ of Proposition \ref{PropositionCoveringsAsOrbitCongruences} above) should coincide with \emph{rigid quotients} of quandles, i.e.~surjective morphisms $f\colon{A \to B}$ which induce an isomorphism $\Inn(f) \colon \Inn(A) \to \Inn(B)$. Looking at the proof on page 1150, the authors assume ``by construction'' that the map $\eta$ (between the \emph{excess} of $Q$ and $R$ \cite{FenRou1992}) is surjective, which is equivalent to asking for the bottom right-hand square $c_R\ Adconj(h) = \Inn(h)\ c_Q$ to be a pushout. This doesn't seem to hold in the generality asked for in \cite{BLRY2010}. Note that these results are presented in such a way that they should also hold in $\RCK$, since the idempotency axiom is never used. Then the example above provides a counter-example to \cite[Theorem 4.2]{BLRY2010} in $\RCK$. We further give a counter-example in $\QND$, which shows that \cite[Theorem 4.2]{BLRY2010} must be incorrect.
\end{remark}

\begin{example}
Consider the quandle $Q_{ab\star}$ from Example \ref{ExampleaConnectedComponents}, which by Example \ref{ExamplePthOfQab} is such that $\Pth(Q_{ab\star})= \Z \times \Z$ with $\gr a = \gr b = (1,0)$ and $\gr \star = (0,1)$. Moreover, observe that the trivial quandle with two elements $\pi_0(Q_{ab\star})$ is also such that $\Pth(\pi_0(Q_{ab\star})) = \Fa(\lbrace [a], [\star] \rbrace) = \Z \times \Z$ where $\gr {[a]} = (1,0)$ and $\gr {[\star]} = (0,1)$. Hence the morphism of quandles $f \defeq \eta_{Q_{ab\star}}\colon{Q_{ab\star} \to \pi_0(Q_{ab\star})}$ is such that $\vec{f}= \id_{\Z \times \Z}$. In particular $\Ker(\vec{f}) = \lbrace e \rbrace$ is the trivial group, but $\Inn(f)\colon{\Z/2\Z \to \lbrace e \rbrace}$ is not an isomorphism. 

Other such examples can be built using morphisms between quandles from Example 1.3, as well as Proposition 2.27 and Remark 2.28 in \cite{Eis2014}.
\end{example}

\subsubsection{Visualizing coverings}\label{ParagraphVisualizingCoverings} Coverings are characterized by the trivial action of $f$-symmetric paths, which are the elements $g=g_ag_b^{-1} \in \Pth(A)$ such that $g_a$ and $g_b$ are $f$-symmetric to each other. Notice that an $f$-symmetric pair $g_a$, $g_b$ is obtained from the projections of a primitive path in $\Eq(f)$. We emphasize the geometrical aspect of these 2-dimensional primitive paths by defining \emph{membranes} and \emph{horns}. An $f$-symmetric trail is a compact 1-dimensional concept which remains so when generalized to higher dimensions. The concept of $f$-horn allows for a more visual, geometrical and elementary description of these ingredients as well as their higher-dimensional generalizations.

\begin{definition}\label{DefinitionFMembrane}
Given a morphism $f\colon {A \to B}$ in $\RCK$ (or $\QND$), we define an $f$-membrane $M = ((a_0,b_0),((a_i,b_i),\delta_i)_{1\leq i \leq n})$ to be the data of a primitive trail in $\Eq(f)$ (see Paragraph \ref{ParagraphConnectedComponents}. We call such an $f$-membrane $M$ a $f$-horn if $ a_0=b_0 \eqqcolon x$ which we denote $M = (x, (a_i,b_i, \delta_i)_{1 \leq i \leq n})$. The \emph{associated $f$-symmetric pair} of the membrane or horn $M$ is given by the paths $g^M_a \defeq \gr{a_1}^{\delta_1} \cdots  \gr{a_n}^{\delta_n}$ and $g^M_b \defeq \gr{b_1}^{\delta_1} \cdots  \gr{b_n}^{\delta_n}$ in $\Pth(A)$. The \emph{top trail} is $t_a=(a_0,g^M_a)$ and the \emph{bottom trail} is $t_b=(b_0,g^M_b)$. The \emph{endpoints} of the membrane or horn are given by $a_M = a_0 \cdot g^M_a$ and $b_M = b_0 \cdot g^M_b$. 
\end{definition}

Given an $f$-symmetric trail $(x,g)$ for $g = g_ag_b^{-1} \in \Ker(\vec{f})$ as before, there is an $f$-horn such that its associated $f$-symmetric pair is given by $g_a$ and $g_b$ (in particular the associated $f$-symmetric trail is then $(x,g)$). Given a horn $M =  (x, (a_i,b_i, \delta_i)_{1\leq i \leq n})$, we represent it (with $n=3$ and $\delta_i = 1$ for $1 \leq i \leq 3$) as in the left-hand diagram below. 
\begin{definition} 
A horn $M = (x, (a_i,b_i, \delta_i)_{1\leq i \leq n})$ is said to \emph{close} (into a \emph{disk}) if its endpoints are equal $a^M = x\cdot g^M_a = x\cdot g^M_b = b^M$. The horn $M$ is said to \emph{retract} if for each $ 1 \leq k  \leq n$, the truncated horn $M_{\leq k} \defeq (x, (a_i,b_i,\delta_i)_{1\leq i \leq k})$ closes.
\end{definition}
\[\vcenter{\xymatrix@C=2pt @R=6pt {& & & x  \ar@{{}{-}{}}[dddlll]|-{\dir{>}} _(0.3){\gr{a_1}\, }_(0.5){\gr{a_2}\, }_(0.7){\gr{a_3}\, }   \ar@{{}{-}{}}[dddrrr]|{\dir{>}}^(0.3){\, \gr{b_1}}^(0.5){\, \gr{b_2}}^(0.7){\, \gr{b_3}}     \\
& & \ar@{{}{-}{}}[rr]|-{f} & & & & \\
&  \ar@<0.7ex>@{{}{-}{}}[rrrr]|-{f} & &&& & \\
x \cdot (\gr{a_1}\,\gr{a_2}\,\gr{a_3} ) \ar@<2.3ex>@{{}{-}{}}[rrrrrr]|-{f}  & & & & &  & x \cdot (\gr{b_1}\,\gr{b_2}\,\gr{b_3} )
}} \qquad  \vcenter{\xymatrix@C=4pt @R=6pt {& & x  \ar@/_8ex/@{{}{-}{}}[ddd]|{\dir{>}} _(0.3){\gr{a_1}}_(0.48){\gr{a_2}\ }_(0.65){\gr{a_3}}    \ar@/^8ex/@{{}{-}{}}[ddd]|{\dir{>}}^(0.3){\gr{b_1}}^(0.48){\ \gr{b_2}}^(0.65){\gr{b_3}}    \\
& \ar@<0.52ex>@{{}{-}{}}[rr]|-{f} & & &  \\
& \ar@<1.3ex>@{{}{-}{}}[rr]|-{f} \ar@<-0.95ex>@{{}{-}{}}[rr]|-{f} & & & \\
& & a^M=b^M
}}  \qquad  \vcenter{\xymatrix@C=5pt @R=6pt {x  \ar@<-.5ex>@{{}{-}{}}[ddd]|-{\dir{>}} _(0.2){\gr{a_1}\, }_(0.43){\gr{a_2}\, }_(0.66){\gr{a_3}\, }  \ar@<.5ex>@{{}{-}{}}[ddd]|-{\dir{>}}^(0.2){\, \gr{b_1}}^(0.43){\, \gr{b_2}}^(0.66){\, \gr{b_3}}    \\
 \\
\\
a^M=b^M
}} \]

\begin{corollary}\label{CorollaryCoveringsAsHornRetractors}
A surjective morphism $f\colon {A \twoheadrightarrow B}$ in $\RCK$ (or $\QND$) is a covering if and only if every $f$-horn retracts (or equivalently, if every $f$-horn closes into a disk).
\end{corollary}

\subsubsection{Visualizing normal extensions}
Normal extensions of quandles are described by V.~Even in \cite{Eve2014}. The same description works in racks. We reinterpret it using our own terminology.

\begin{definition} Given a surjective morphism $f\colon{A \to B}$ in $\RCK$, together with an $f$-membrane $M=(a_i,b_i, \delta_i)_{0\leq i \leq n}$, we say that the membrane $M$ \emph{forms a cylinder} if both the top and the bottom trails of $M$ are loops.
\end{definition}

\begin{proposition}
A surjective morphism $f\colon{A \to B}$ in $\RCK$ (or $\QND$) is a normal extension if and only if $f$-membranes are \emph{rigid}, i.e.~if and only if given any $f$-membrane $M = (a_i,b_i, \delta_i)_{0\leq i \leq n}$, $M$ forms a cylinder as soon as either the top or the bottom trail of $M$ is a loop.
\end{proposition}
\begin{proof}
The surjection $f$ is normal if and only if the projections $\pi_1,\pi_2 \colon{\Eq(f) \rightrightarrows A}$ of the kernel pair of $f$ are trivial. Such projections are trivial if and only if they reflect loops. The $\pi_1$ (resp.~$\pi_2$) projection of a trail $t=((a_0,b_0), h)$ in $\Eq(f)$ loops if and only if there is an $f$-membrane $M=((a_0,b_0),((a_i,b_i),\delta_i)_{1\leq i \leq n})$ such that $\vec{\pi}_1(h)= g^M_a$, $\vec{\pi}_2(h)= g^M_b$ and the top (resp.~bottom) trail of $M$ loops (see also \cite[Proposition 3.2.3]{Eve2014}). 
\end{proof}

\subsection{Characterizing central extensions}\label{SectionCharacterizingCentralExtensions}

V.~Even's strategy to prove the characterization is to split coverings along the weakly universal covers constructed by M.~Eisermann. These weakly universal covers can be understood as the centralization of the canonical projective presentations (using free objects -- see Section \ref{SectionWucAndFundamentalGroupoid}). Their structure and properties used to show V.~Even's result derive from the structure and properties of the free objects we described before. Thus even though V.~Even's proof can be translated to the context of racks, we prefer to work directly with free objects in the alternative proof below. This approach then easily generalizes to higher dimensions without us having to build the weakly universal higher-dimensional coverings from scratch.

\begin{proposition}\label{PropositionCoversingsAboveFreeObjects}
Any rack-covering with free codomain $f \colon {A \to \Fr(B)}$ is a trivial extension.
\end{proposition}
\begin{proof}
In order to test whether $f$ is a trivial extension, consider $x \in A$ and $g = \gr{a_1}^{\delta_1} \cdots \gr{a_n}^{\delta_n} \in \Pth(A)$ for $n\in \N$, $a_1$, $\ldots$, $a_n$ in $A$ and $\delta_1$, $\ldots$, $\delta_n$ in $\lbrace -1,1 \rbrace$. Assume that $f$ sends the trail $(x,g)$ to the loop $(f(x),\vec{f}(g))$:
\[f(a) \cdot (\gr{f(a_1)}^{\delta_1} \cdots \gr{f(a_n)}^{\delta_n}) = f(x) \qndop^{\delta_1} f(a_1) \cdots \qndop^{\delta_n} f(a_n)  = f(x \qndop^{\delta_1} a_1  \cdots  \qndop^{\delta_n} a_n  ) =  f(x), \]
where we write $\gr{f(a_i)} \defeq \pth_{\Fr(B)}(f(a_i))$ (which does not mean that $f(a_i)$ is in $B$).
We have to show that $(x,g)$ was a loop in the first place: \[ x \cdot g = x \qndop^{\delta_1} a_1  \cdots  \qndop^{\delta_n} a_n = x.\]

$(*)$ Since $\Fr(B)$ is projective (with respect to surjective morphisms) and $f$ is surjective, there is a morphism of racks $s\colon{\Fr(B) \to A}$ such that $fs=1_{ \Fr(B)}$. Then $s$ induces a group homomorphism $\vec{s} \colon {\Pth(\Fr(B)) \to \Pth(A)}$ such that for each $1 \leq i \leq n$, $\vec{s}[\gr{f(a_i)}] = \pth_A(sf(a_i)) \eqqcolon \gr{sf(a_i)}$ (see Paragraph \ref{ParagraphUnitAndUniversalPropertyPth}), and thus 
\[e = \vec{s}[\gr{f(a_1)}]^{\delta_1} \cdots \vec{s}[\gr{f(a_n)}]^{\delta_n} = \gr{sf(a_1)}^{\delta_1} \cdots \gr{sf(a_n)}^{\delta_n}.\]
Hence in particular we have 
\[ x \qndop^{\delta_1} sf(a_1) \cdots \qndop^{\delta_n} sf(a_n) = x \cdot ( \gr{sf(a_1)}^{\delta_1} \cdots \gr{sf(a_n)}^{\delta_n}) =\ x \cdot e = x. \]

Finally since for each $1 \leq i \leq n$ we have $f(sf(a_i))=f(a_i)$, $M=(x,(a_i,sf(a_i),\delta_i)_{1\leq i \leq n})$ is an $f$-horn, which has to retract since $f$ is a covering:
\[ x \qndop^{\delta_1} a_1  \cdots  \qndop^{\delta_n} a_n  =x  \qndop^{\delta_1} sf(a_1)  \cdots \qndop^{\delta_n} sf(a_n) = x. \qedhere\] 
\end{proof}

In this one-dimensional context, the characterization of coverings from Proposition \ref{PropositionCoveringsAsOrbitCongruences} allows for a shorter version of this proof. Since a direct generalization of Proposition \ref{PropositionCoveringsAsOrbitCongruences} in higher dimensions is not yet clear to us, we prefer to keep the previous, more visual version of the proof as our main reference. However, you may want to replace what follows $(*)$ in the previous proof by:

\begin{proof} $[...]$ $(*)$ Now since the action of $\Pth(\Fr(B))$ on $\Fr(B)$ is free, any loop in $\Pth(\Fr(B))$ must be trivial, and in particular $\gr{f(a_1)}^{\delta_1} \cdots \gr{f(a_n)}^{\delta_n} = e$. Hence $g \in \Ker(\vec{f})$, and thus by Proposition \ref{PropositionCoveringsAsOrbitCongruences}, $x \cdot g = x$, which concludes the proof.
\end{proof}

Note finally that the exact same proofs work for quandle coverings, using the fact that if $A$ is a quandle, we may then always choose $a_i$'s and $\delta_i$'s such that $\sum_i \delta_i =0$. Then $\gr{f(a_1)}^{\delta_1} \cdots \gr{f(a_n)}^{\delta_n}$ is in $\Pth\degree(\Fq(B))$ which acts freely on $\Fq(B)$. The rest of both proofs remain identical.

\begin{proposition}
If a quandle-covering $f \colon {A \to \Fq(B)}$ has a free codomain, then it is a trivial extension. 
\end{proposition}

By Lemma \ref{LemmaCoveringsAndPullbacks}, and the previous propositions, the strategy of Section \ref{SectionStrategy} yields Theorem 2 from \cite{Eve2014}, as well as:

\begin{theorem}
Rack coverings are the same as central extensions of racks. 
\end{theorem}

\subsection{Comparing admissible adjunctions by factorization}\label{SectionComparingAdjunctionsByFactorization}
The notions of trivial object and connectedness, or trivialising relation $\Co$, coincide in racks and quandles. These are understood as the zero-dimensional central extensions and centralizing relations. In dimension 1, the notions of central extensions in racks and quandles also coincide. Further we also have coincidence of the centralizing relations and the corresponding notions in dimension 2. Before we move on, we show how these results are no coincidence and can be studied systematically as a consequence of the tight relationship between the $\pi_0$-admissible adjunctions of interest.

Expanding on Paragraph \ref{ParagraphRckAndQndHvTheSameConnComp} we get a factorization as in \ref{SubsectionRckAndQndHaveTheSameGrpOfPth}, where all triangles commute and all the adjunctions are admissible:
\[\xymatrix@C=35pt@R=20pt{ \RCK \ar@<+3pt>@/^7pt/[rr]^-{\Frq}  \ar@<-3pt>@/_7pt/[rdd]|(.4){\pi_0} \ar@{{}{}{}}[rr]|-{\perp}  & & \QND  \ar@<+3pt>@/^7pt/[ll]^-{\I} \ar@<-3pt>@/_7pt/[ddl]|(.63){\pi_0} \\
\\
 & \SET. \ar@<-3pt>@/_7pt/[uur]|(.61){\I} \ar@<-3pt>@/_7pt/[uul]|(.38){\I}  \ar@{{}{}{}}[uur]|(.5){\dashv}  \ar@{{}{}{}}[uul]|(.5){\dashv} &
}\]
Since we are dealing here with several different Galois structures: $\Gamma$ from $\RCK$ to $\SET$, $_r\Gamma_q$ from $\RCK$ to $\QND$ and say $\Gamma_q \defeq (\QND$, $\SET$, $\pi_0$, $\I$, $\eta$, $\epsilon$, $\EE)$; we specify the Galois structure with respect to which the concepts of interest are discussed. 
\begin{lemma}\label{LemmaGammaTrivialImpliesAllTrivial}
If $f\colon{A \to B}$ is a $\Gamma$-trivial extension, then $f$ is also $_r\Gamma_q$-trivial, and the image $\Frq(f)$ of $f$ is a $\Gamma_q$-trivial extension in $\QND$.
\end{lemma}
\begin{proof}
The $\Gamma$-canonical square of $f$ in $\RCK$ is given on the left, and factorizes into the composite of double extensions on the right:
\[ \vcenter{\xymatrix@C=15pt@R=15pt{ A \ar[d]_-{f} \ar[rr]^-{\eta_A} & & \pi_0(A) \ar[d]^-{\pi_0(f)}  \\
B \ar[rr]_-{\eta_A} & & \pi_0(B),}} \qquad  \vcenter{\xymatrix@C=30pt@R=15pt{ A \ar[d]_-{f} \ar[r]^-{^r\eta^q_A} & \Frq(A) \ar[d]_-{\Frq(f)} \ar[r]^-{\eta_{\Frq(A)}} & \pi_0(A) = \pi_0(\Frq(A))  \ar[d]^-{\pi_0(f)} \\
B \ar[r]_-{^r\eta^q_B} & \Frq(B) \ar[r]_-{\eta_{\Frq(B)}} & \pi_0(B) = \pi_0(\Frq(B)).}} \]
Hence if $f$ is a trivial extension, then this composite is a pullback square. The composite of two double extensions is a pullback if and only if both double extensions are pullbacks themselves. 
\end{proof}

\begin{lemma}\label{LemmaInclusionOfQndCentralAndTrivialExtensionInRck}
An extension $f\colon{A \to B}$ in $\QND$ is 
\begin{enumerate}[label=(\roman*)]
\item $\Gamma_q$-trivial in $\QND$ if and only if $\I(f)$ is $\Gamma$-trivial in $\RCK$;
\item $\Gamma_q$-central in $\QND$ if and only if $\I(f)$ is $\Gamma$-central in $\RCK$.
\end{enumerate}
\end{lemma}
\begin{proof}
The first point $(i)$ is immediate by the previous lemma, and the fact that the $\pi_0$-canonical squares of $\I(f)$ in $\RCK$ is the same as the image by $\I$ of the $\Gamma_q$-canonical square of $f$ in $\QND$. Note also that $\I$ preserves and reflects pullbacks.

For the second statement $(ii)$, if $f$ is $\Gamma_q$-central, then there is an extension $p\colon{E \to B}$ such that the pullback of $f$ along $p$ is $\Gamma_q$-trivial. We may conclude by taking the image by $\I$ of this pullback square. Now if $\I(f)$ is $\Gamma$-central in $\RCK$, there exists $p\colon{E \to B}$ in $\RCK$ such that the pullback $t$ of $\I(f)$ along $p$ is $\Gamma$-trivial in $\RCK$. Taking the quotient along $^r\eta^q$ of this pullback square (1) yields a factorization of (1):
\[ \xymatrix@C=30pt@R=15pt{ E \times_B A \ar[d]_-{t} \ar[r]^-{^r\eta^q_P} & \Frq(E \times_B A) \ar[d]_-{\Frq(t)} \ar[r]^-{ } & A \ar[d]^-{f} \\
E \ar[r]_-{^r\eta^q_E} & \Frq(E) \ar[r]_-{\Frq(p)} & B.
} \]
Again, since the left hand square is a double extension, and the composite is a pullback, both squares are actually pullbacks and thus $f$ is $\Gamma_q$-central.
\end{proof}

Now since the $\pi_0$-adjunction is strongly Birkhoff (both in $\RCK$ and $\QND$), central extensions are closed by quotients along double extensions in $\EXT\RCK$ (or $\EXT\QND$ -- see also Proposition \ref{PropositionClosureOfCentralExtensionsAlongDoubleExtensions}).

\begin{corollary}\label{CorollaryImageOfACentralExtensionByFrq}
The image by $\Frq$ of a $\Gamma$-central extension $f\colon{A \to B}$ in $\RCK$ is a $\Gamma_q$-central extension in $\QND$.
\end{corollary}
\begin{proof}
The image $\Frq(f)$ is $\Gamma_q$-central extension if and only if $\I(\Frq(f))$ is $\Gamma$-central. Since $\SET$ is strongly Birkhoff in $\RCK$, $\I(\Frq(f))$ is the quotient of a $\Gamma$-central extension in $\RCK$ along a double extension and thus is still $\Gamma$-central in $\RCK$. 
\end{proof}

\begin{proposition}
If the image by $\Frq$ of an $_r\Gamma_q$-trivial extension $f\colon{A \to B}$ in $\RCK$ is a $\Gamma_q$-central extension in $\QND$, then $f$ is $\Gamma$-central in $\RCK$. 
\end{proposition}
\begin{proof}
Consider the following commutative cube in $\RCK$ where we omit the inclusion $\I\colon{\QND \to \RCK}$. The back face is a pullback by construction. The right hand face is a pullback by assumption, and the left hand face is a pullback by Proposition \ref{PropositionAllSurjAreFrqCentral}. We deduce that the front face is a pullback as well.
\[\xymatrix@1@!0@=38pt{
& P_1 \ar[rr]^-{} \ar[dd]^(.25){t}|-{\hole} \ar[ld]_-{^r\eta^q_{P_1}} && A \ar[dd]^-{f} \ar[ld]|-{^r\eta^q_A} \\
\Frq(P_1) \ar[rr]^(.75){} \ar[dd]_-{\Frq(t)} && \Frq(A) \ar[dd]^(.25){\Frq(f)} \\
& \Fr(B) \ar[ld]|-{^r\eta^q_{\Fr(B)}} \ar[rr]^(.3){\epsilon^r_B}|(.5){\hole} && B \ar[ld]^-{^r\eta^q_B} \\
\Fq(B) \ar[rr]_-{\Frq(\epsilon^r_B)} && \Frq(B)}\]
Since $\Frq(f)$ is $\Gamma_q$-central by assumption, and since $\Frq(\epsilon^r_B)\colon{\Fq(B) = \Frq(\Fr(B)) \to \Frq(B)}$ factorizes as 
\[ \xymatrix@C=50pt{\Fq(B) \ar[r]^-{\Fq(\, ^r\eta^q_B)} & \Fq(\Frq(B))  \ar[r]^-{\Fq(\epsilon^q_{\Frq(B)})} & \Frq(B), } \]
both $\Frq(t)$ and $t$ are $\Gamma$-trivial as the pullback of a trivial extension.
\end{proof}

\begin{example}
Some extensions of racks which are not central, still have central images under $\Frq$. Define the involutive rack with underlying set $\lbrace a, a^2 , b, b^2, 1, 2 \rbrace$, and an operation $\qndop$ such that $a$, $a^2$, $b$ and $b^2$ have the same action and,  moreover,
\begin{center}
\begin{tabular}{c | c c c c c c}
$\triangleleft$    & $a$ & $a^2$ & $b$ & $b^2$ & $1$ & $2$ \\
\hline
$a$ & $a^2$ & $a$ & $b^2$ & $b$ & $1$ & $2$ \\
$1$ & $b$ & $b^2$ & $a$ & $a^2$ & $1$ & $2$ \\
$2$ & $b^2$ & $b$ & $a^2$ & $a$ & $1$ & $2$ \\
\end{tabular}
\end{center}
We may check by hand that the axioms $\R{1}$ and $\R{2}$ are satisfied. Then define the morphism of racks $f$, with codomain the trivial rack $\lbrace x, 1 \rbrace$ and which sends letters to $x$ and numbers to $\star$. We have that $a \qndop 1 = b \neq b^2 = a \qndop 2 $, and thus $f$ is not central. However we compute the morphism $\Frq(f)\colon{Q_{ab\star\star} \to \lbrace x, \star \rbrace}$, where $Q_{ab\star\star}$ is as in Example \ref{ExampleaConnectedComponents} but with two distinct $\star$'s which act in the same way. This morphism merely identifies the letters and the stars and thus it is central.

Of course some rack homomorphisms which are not $_r\Gamma_q$-trivial are still $\Gamma$-central: we already mentioned the important example of $^r\eta^q_A$ for any rack $A$.
\end{example}

Before even studying the next steps of the covering theory, we can predict that what happens in $\QND$ directly follows from what happens in $\RCK$.
\begin{corollary}\label{CorollaryCoveringsAreReflective}
If the full subcategory $\CEXT\RCK$ of central extensions of racks is reflective within the category of extensions $\EXT\RCK$ (see Theorem \ref{TheoremEpiReflectivityOfCExt} for details), then also $\CEXT\QND$ is reflective in $\EXT\QND$ and the reflection is computed as in $\EXT\RCK$, via the inclusion $\I\colon{\QND \to \RCK}$.
\end{corollary}
\begin{proof}
Since $\QND$ is closed under quotients in $\RCK$, the centralization of an extension in $\QND \preceq \RCK$ yields an extension in $\QND$ which is moreover central by Lemma \ref{LemmaInclusionOfQndCentralAndTrivialExtensionInRck}. The universality in $\CEXT\QND$ directly derives from the universality in $\CEXT\RCK$ by the same arguments.
\end{proof}

\subsection{Centralizing extensions}\label{SectionCentralizingExtensions}
We adapt the result from \cite{DuEvMo2017}, showing the reflectivity of quandle coverings in the category of extensions, to the context of racks. We put the emphasis on our new characterizations of the centralizing relation which works the same for racks and for quandles. We also prepare the ingredients to show the admissibility of coverings within extensions, and the forthcoming covering theory in dimension 2. 

Let us define $\EE_1$ to be the the class of double extensions in $\EXT\RCK$.
\begin{theorem}\label{TheoremEpiReflectivityOfCExt}
The category $\CEXT\RCK$ is an ($\EE_1$)-reflective subcategory of the category $\EXT\RCK$ with left adjoint $\Fi$ and unit $\eta^1$ defined for an object $f\colon{A \to B}$ in $\EXT\RCK$ by $\eta^1_{f} \defeq (\eta^1_A, \id_B)$, where $\eta^1_A\colon{A \to \Fii(A)}$ is the quotient of $A$ by the \emph{centralizing} congruence $\Ci(f)$, which can be defined in the following equivalent ways:
\begin{enumerate}[label=(\roman*)]
\item $\Ci(f)$ is the equivalence relation on $A$ generated by the pairs $(x\qndop a \qndiop b, x)$ for $x$, $a$, and $b$ in $A$ such that $f(a) = f(b)$,
\item a pair $(a,b)$ of elements from $A$ is in the equivalence relation $\Ci(f)$ if and only if $a$ and $b$ are the endpoints of a horn, i.e.~there exists a horn $M=(x,(a_i,b_i,\delta_i)_{1\leq i \leq n})$ such that $x \cdot g^M_a = a$ and $x \cdot g^M_b = b$,
\item $\Ci(f) $ is the orbit relation $ \sim_{\Ker(\vec{f})}$ (or equivalently $\sim_{\K{\vec{f}}}$) induced by the action of the kernel of $\vec{f}$ (i.e. the group of $f$-symmetric paths).
\end{enumerate}
Observing that $\Ci(f) \leq \Eq(f)$, the image of $f$ by $\Fi$ is defined as the unique factorization of $f$ through this quotient:
\[\xymatrix@R=10pt{
A \ar[rd]_-{\eta^1_{A}} \ar[rr]^-{f} && B\\
& \Fii(A) \ar[ur]_-{\Fi(f)}  &
}\]
The definition of $\Fi$ on morphisms $\alpha = (\alpha_1,\alpha_0)\colon{f_A \to f_B}$ decomposes into the top (initial) component $\Fii(\alpha_1)$ defined by the universal property of the quotients $\eta^1_{A_1}$ for $f_A \colon {A_1 \to A_0}$; and the bottom component $\Fio(\alpha_0) = \alpha_0$ which is the identity.
\end{theorem}
\begin{proof}
Using definition $(i)$ for the centralizing relation, the proof of Theorem 5.5 in \cite{DuEvMo2017} easily translates to the context of racks. Then given an extension $f\colon{A \to B}$, the unit $\eta^1_f =(\eta^1_A,\id_B)$ is indeed a double extension since its bottom component is an isomorphism. It remains to show that the definitions $(ii)$ and $(i)$ are equivalent, since $(iii)$ is equivalent to $(ii)$ by Proposition \ref{PropositionKernelOfPthF}.

First we show by induction on $n \in \N$ that $\Ci(f)$, defined as in $(i)$, contains all pairs that are endpoints of a horn. Then we show that the collection of such pairs defines a congruence containing the generators of $\Ci(f)$. This then concludes the proof.

Step $0$ is satisfied by reflexivity of $\Ci(f)$. Now assume that if $(a,b)$ is a pair of elements in $A$, which are endpoints of a horn $M=(x,(a_i,b_i,\delta_i)_{1\leq i \leq n})$ of length $n \leq k$, for some fixed natural number $k$, then $(a,b) \in \Ci(f)$. We show that the endpoints $a \defeq x \cdot g^M_a$ and $b \defeq x \cdot g^M_b$ of any given horn $M=(x,(a_i,b_i,\delta_i)_{1\leq i \leq k+1})$ of length $k+1$ are in relation by $\Ci(f)$. Indeed, define $a' = a \qndop^{-\delta_{k+1}} a_{k+1} $ and $b' = b \qndop^{-\delta_{k+1}} b_{k+1}$. Then we have that $(a',b') \in \Ci(f)$ by assumption and, moreover,
\[ \left( a = a' \qndop^{\delta_{k+1}} a_{k+1} \right)  \quad  \Ci(f) \quad  \left( b' \qndop^{\delta_{k+1}} a_{k+1} \right)  \quad\Ci(f) \quad \left(  b' \qndop^{\delta_{k+1}} b_{k+1} = b\right), \]
by compatibility of $\Ci(f)$ with the rack operation, together with reflexivity, and further by definition $(i)$ of $\Ci(f)$. We may conclude by transitivity of $\Ci(f)$.

Now define the symmetric set relation $S$ as the subset of $A\times A$, given by pairs of endpoints of $f$-horns. Looking at horns of length $0$ and $1$, $S$ defines a reflexive relation containing the generators of $\Ci(f)$. It is also easy to observe that it is compatible with the rack operation. Thus it remains to show transitivity. In order to do so, for $k$ and $n$ in $\N$, consider a horn $M=(x,(a_i,b_i,\delta_i)_{1\leq i \leq k})$, and its endpoints $a$ and $b$ as before, as well as a horn $N=(z,(c_i,d_i,\gamma_i)_{1\leq i \leq n})$ with endpoints $c = z \cdot g^N_a$ and $d = z \cdot g^N_b$. If $b = c$ then also $(a,d)$ is in $S$ since: 
\begin{align*}
&a = x \qndop^{\delta_1} a_1 \cdots \qndop^{\delta_{k}}  a_{k} \qndop^{-\gamma_{n}} c_{n}   \cdots \qndop^{-\gamma_1} c_1 \qndop^{\gamma_{1}} c_{1}  \cdots \qndop^{\gamma_n}  c_n, \\
&d = x \qndop^{\delta_1} b_1 \cdots \qndop^{\delta_{k}}  b_{k} \qndop^{-\gamma_{n}} c_{n}   \cdots \qndop^{-\gamma_1} c_1 \qndop^{\gamma_{1}} d_{1}   \cdots \qndop^{\gamma_n} d_n.
\qedhere
\end{align*}
\end{proof}

By Corollary \ref{CorollaryCoveringsAreReflective}, what we deduced about the functor $\Fi$ restricts to the domain $\CEXT\QND$, and so also describes the left adjoint to the inclusion in $\EXT\QND$ from Theorem 5.5. in \cite{DuEvMo2017}. In addition to Corollary \ref{CorollaryCoveringsAreReflective}, we further describe how centralization behaves with respect to $\Frq$. 

\subsubsection{Navigating between racks and quandles}
Observe that the adjunction $\Frq \colon{ \RCK \rightleftarrows \QND }\colon \I$ induces (in the obvious way) an adjunction $\Frq^1 \colon{ \EXT\RCK \rightleftarrows \EXT\QND }\colon\I$ with unit given by $_1^r\eta^q = (^r\eta^q,\, ^r\eta^q)$. Then by Corollary \ref{CorollaryImageOfACentralExtensionByFrq} this adjunction restricts to the full subcategories $\Frq^1 \colon{ \CEXT\RCK \rightleftarrows \CEXT\QND }\colon\I$.

\begin{proposition}
We have the following square of adjunctions, in which all possible squares of functors commute (up to isomorphism):
\[\xymatrix@C=20pt@R=15pt{ \EXT\RCK \ar@{{}{}{}}[dd]|-{\dashv}  \ar@<-3pt>@/_7pt/[rr]_-{\Frq^1}  \ar@<-3pt>@/_7pt/[dd]_-{\Fi} \ar@{{}{}{}}[rr]|-{\top}  & & \EXT\QND \ar@{{}{}{}}[dd]|-{\dashv}   \ar@<-3pt>@/_7pt/[ll]_-{\I} \ar@<-3pt>@/_7pt/[dd]_-{\Fi} \\
\\
 \CEXT\RCK  \ar@{{}{}{}}[rr]|-{\top}  \ar@<-3pt>@/_7pt/[uu]_-{\I} \ar@<-3pt>@/_7pt/[rr]_-{\Frq^1}   & &  \CEXT\QND.  \ar@<-3pt>@/_7pt/[ll]_-{\I} \ar@<-3pt>@/_7pt/[uu]_-{\I}
}\]
\end{proposition}
\begin{proof}
Corollary \ref{CorollaryComputingCentralizationUsingFrq} gives commutativity of the square $\Fi\I = \I\Fi$ from the top right to the bottom left. In the opposite direction, $\I\Frq^1 = \Frq^1\I$ by Corollary \ref{CorollaryImageOfACentralExtensionByFrq} again. Finally bottom-right to top-left $\I \I = \I \I$ commutes trivially, from which we can deduce, by uniqueness of left adjoints, that $\Frq^1\Fi = \Fi\Frq^1$.
\end{proof}

In particular we have:
\begin{corollary}\label{CorollaryComputingCentralizationUsingFrq}
If $f\colon{A \to B}$ is a morphism of racks, then the centralization \[\Fi(\Frq(f)) \colon{\Fii(\Frq(A)) \to \Frq(B)}\] of $\Frq(f)$ is equal (up to isomorphism) to the reflection $\Frq(\Fi(f)) \colon { \Frq(\Fii(A)) \to \Frq(B)}$ of the centralization $\Fi(f)$ of $f$.
\end{corollary}

\subsubsection{Towards admissibility in dimension 2}\label{SectionTowardsAdmissibilityInD2}
A reflector such as $\Fi$, of a subcategory of morphisms containing the identities into a larger class of morphisms can always be chosen such that the bottom component of the unit of the adjunction is the identity \cite[Corollary 5.2]{ImKel1986}. This is important in order to obtain higher order reflections and admissibility, for we relate certain problems back to the first level context (which has the advantage of being complete, cocomplete and Barr-exact). For dimension 2, we need this reflection to be strongly Birkhoff. Below we have the results we need for the permutability condition on the kernel pair of the unit (``strongly'') and for the closure by quotients of central extensions (``Birkhoff'').

\begin{proposition}\label{PropositionPermutabilityKernelPairOfTheUnit}
Given a rack extension $f\colon{A \to B}$ (or in particular an extension in $\QND$) as before, the kernel pair $\Ci(f)$ of the domain-component $\eta^1_A$ of the unit $\eta^1_f \defeq (\eta^1_A, \id_B)$, commutes with all congruences on $A$, in $\RCK$ (and so also in particular in $\QND$).
\end{proposition}
\begin{proof}
By Theorem \ref{TheoremEpiReflectivityOfCExt}, the centralizing relation $\Ci(f)$ is an orbit congruence which thus commutes with any other congruence on $A$.
\end{proof}

As we shall see in Part II and III, the following property is a consequence of the fact that the Galois structure $\Gamma$, in dimension 0, is strongly Birkhoff. For now we show by hand:

\begin{proposition}\label{PropositionClosureOfCentralExtensionsAlongDoubleExtensions}
If $\alpha$ is a double extension of racks (or in particular quandles)
\[ \vcenter{\xymatrix @R=1pt @ C=11pt{
A_{\ttop} \ar[rrr]^{\alpha_{\ttop}} \ar[rd] |-{p}  \ar[dddd]_{f_A} & & & B_{\ttop}  \ar[dddd]^-{f_B} \\
 & A_{\pperp} \times_{B_{\pperp}} B_{\ttop} \ar[rru]|-{\pi_2} \ar[lddd]|-{\pi_1}  \\
\\
\\
A_{\pperp}  \ar[rrr]_{\alpha_{\pperp}} &  & & B_{\pperp}
}}
\]
then the morphism $\bar{\alpha}$ induced between the centralizing relations $\Ci(f_A)$ and $\Ci(f_B)$ is a regular epimorphism.
Moreover, if $f_A$ is a central extension then $f_B$ is a central extension. 
\end{proposition}
\begin{proof}
Certainly if we show that $\bar{\alpha}$ is a regular epimorphism, then assuming that $f_A$ is central, then its centralizing relation is trivial, hence the centralizing relation of $f_B$ is trivial, showing that $f_B$ is central (note that in this context, it is enough to have preservation of centrality by quotients along double extensions in order to have surjectivity of $\bar{\alpha}$, see Part II and III).

We pick a pair $(x \qndop y, x \qndop z)$ amongst the generators of $\Ci(f_B)$ (i.e.~with $f_B(y)=f_B(z)$). Since $\alpha_{\pperp}$ is surjective we get $a \in A_{\pperp}$ such that $\alpha_{\pperp}(a) = f_B(y)$. Now both pairs $(a,y)$ and $(a,z)$ are in the pullback $A_{\pperp} \times_{B_{\pperp}} B_{\ttop} $ hence there exist $t$ and $s$ in $A_{\ttop}$ such that $\alpha_{\ttop}(t) = y$, $\alpha_{\ttop}(s) = z$ and $f_A(t) = f_A(s) =a$, by surjectivity of $p$. Now there is also $u \in A_{\ttop}$ such that $t(u) = x$ and the pair $(u \qndop t, u \qndop s)$ is a generator of $\Ci(f_A)$ by definition. It is also sent to $(x \qndop y, x \qndop z) \in \Ci(f_B)$ by $\bar{\alpha}$ by construction. All generators of $\Ci(f_B)$ are thus in the image of $\bar{\alpha}$, and this concludes the proof.
\end{proof}

\begin{corollary}\label{CorollaryStronglyBirkhoffLemma}
Given a morphism $\alpha = (\alpha_{\ttop}, \alpha_{\pperp})\colon{f_A \to f_B}$ in $\EXT\RCK$ such that $\alpha_{\ttop}$ and $\alpha_{\pperp}$ are surjections, then the square below (where $P \defeq \Fii(A_{\ttop}) \times_{\Fii(B_{\ttop})} B_{\ttop}$) is a double extension of racks. Similarly in $\EXT\QND$.
\[\vcenter{\xymatrix @R=1pt @ C=11pt{
A_{\ttop} \ar[rrr]^{\alpha_{\ttop}} \ar[rd] |-{p}  \ar[dddd]_-{\eta^1_{A_{\ttop}}}  & & &  B_{\ttop}  \ar[dddd]^{\eta^1_{B_{\ttop}}} \\
 & P  \ar[rru]|-{\pi_2} \ar[lddd]|-{\pi_1}  \\
\\
\\
\Fii(A_{\ttop})  \ar[rrr]_{\Fii(\alpha_{\ttop})} &  & & \Fii(B_{\ttop}) 
}}
\]
\end{corollary}
\begin{proof}
By Lemma 1.2 in \cite{Bou2003}, this square is a pushout as a consequence of Proposition \ref{PropositionClosureOfCentralExtensionsAlongDoubleExtensions}. Then by Proposition 5.4 in \cite{CKP1993}, $p$ is a surjection as well, making $\alpha$ into a double extension.
\end{proof}

In Part II we complete the proof that $\Gamma_1 = (\EXT\RCK,\CEXT\RCK,\Fi,\I,\eta^1,\epsilon^1, \EE^1)$ forms an admissible Galois structure such that morphisms in $\EE^1$ are of \emph{effective $\EE^1$-descent} \cite{JanTho1994,JanSoTho2004}.

\subsection{Weakly universal covers and the fundamental groupoid}\label{SectionWucAndFundamentalGroupoid}
We insist on the importance of the new results of this Section and the following, in achieving a precise theoretical understanding and expansion of M.~Eisermann's covering theory of quandles (as a continuation of V.~Even's contributions).
\subsubsection{Centralizing the canonical presentations}
Weakly universal covers (w.u.c.) for quandles were described by M.~Eisermann. He also indicated how to adapt his theory to the case of racks. In this section, we recover his constructions from the centralization of the canonical projective presentations as explained in the introduction. Note that the difference between the w.u.c.~in racks and in quandles is then due to the difference between the canonical projective presentations rather than the centralizations which are the same. 

Given the canonical projective presentation of a rack $\epsilon^r_X \colon{\Fr(X) \to X}$, we saw in Paragraph \ref{ParagraphConstructionOfPthAsColimit} that the induced morphism $\vec{\epsilon}^{\ r}_X$ is actually the quotient map $\vec{\epsilon}^{\ r}_X = q_X\colon{\Fg(X) \to \Pth(X)}$ from Subsection \ref{SubsectionGroupOfPath}. Hence the kernel of $\vec{\epsilon}^{\ r}_X$ is given by \[\Ker(\vec{\epsilon}^{\ r}_X) = \langle \langle  \gr{c}^{-1}\gr{a}^{-1}\gr{x}\,\gr{a} \mid a,x,c \in X \text{ and } c= x \qndop a \rangle \rangle_{\Fg(X)}.\]
Since the action of $\Pth(\Fr(X)) = \Fg(X)$ is by right multiplication, two elements $(a,g)$ and $(b,h)$ in $\Fr(X)$ are identified by the centralizing relation $\Ci(\epsilon^r_X)$ if and only if $a=b$ and there is $k \in \Ker(\vec{\epsilon}^{\ r}_X)$ such that $g = hk$. In other words, the domain component $\eta^1_{\Fr(X)}$ of the centralization unit is given by the product 
\[ \xymatrix@C=50pt{ X \rtimes \Fg(X) \ar[r]^-{\id_X \times q_X} & X \rtimes \Pth(X),}\] where the operation in $\tilde{X} \defeq X \rtimes \Pth(X)$ is defined as in Paragraph \ref{ParagraphTerminologyAndVisualRepr}, Equation~\eqref{EquationActionByCodomainRacks}.
\begin{definition}
Given a rack $X$, we define the associated weakly universal cover of $X$ to be the centralised map $\omega_X \defeq \Fi(\epsilon^r_X)$ \[\xymatrix{\tilde X \defeq X \rtimes \Pth(X) \ar[r]^-{\omega_X} & X,}\]
where $\omega_X$ sends a trail $(a,g) \in \tilde{X}$ to its endpoint $a \cdot g$, and trails in $\tilde{X}$ ``act by endpoint'' as in $\Fr(X)$. Note that this construction is functorial in $X$, yielding a functor $\tilde{-} \colon \RCK \to \RCK$ which sends a morphism of racks $f\colon{A \to B}$ to the morphism $\tilde{f} \defeq f \times \vec{f} \colon {\tilde{A} \to \tilde{B}}$; and a natural transformation $\omega \colon{\tilde{-} \to \id_{\RCK}}$, whose component at $X$ is $\omega_X$.
\end{definition}
Then the action of $\Pth(X)$ induced by the covering $\omega_X$ on $\tilde{X} = X \rtimes \Pth(X)$ is by right multiplication, and is thus free. Given any other covering $f\colon{B \to X}$, together with a splitting function $s\colon{X \to B}$ in $\SET$ such that $fs=\id_X$, a factorization $\omega_f\colon{\tilde{X} \to B}$ of $\omega_X$ through $f$ is given by $\omega_f(a,e) \defeq s(a)$ and compatibility with the action of $\Pth(X)$ on $\tilde{X}$ and $B$ (see Corollary \ref{CorollaryCoveringInducedAction}).

Starting with the canonical projective presentation of a quandle $\epsilon^q_X \colon{X \rtimes \Pth\degree(X) \to X}$, the same reasoning yields a w.u.c. with the same properties
\[ \xymatrix{\tilde{X}\degree \defeq X \rtimes \Pth\degree(X) \ar[r]^-{\omega^q_X}  & X,}\]
such that the quandle structure on $X \rtimes \Pth\degree(X)$ is as for $\Fq(X)$ (Paragraph \ref{ParagraphDefinitionFreeQuandle}). As in the case of racks, this describes a functor as well as a natural transformation whose component at any quandle $A$ is $\omega^q_A$. Observe that Corollary \ref{CorollaryComputingCentralizationUsingFrq} implies that $\tilde{X}\degree \defeq X \rtimes \Pth\degree(X)$ is actually the free quandle on the rack $\tilde{X} \defeq  X \rtimes \Pth(X)$ and thus if $X$ is a quandle, then $\omega^q_X$ is merely the image of $\omega_X$ by $\Frq$.

As it was proved by V.~Even \cite{Eve2015}, every covering of $X$ is split by $\omega^q_X$ in $\QND$ and a similar argument shows that every covering of $X$ is split by $\omega_X$ in $\RCK$. This derives more generally from Corollary \ref{CorollaryStronglyBirkhoffLemma}:
\begin{proposition}\label{PropositionSplitByCentralization}
If the extension $c\colon A \to B$ is split by an extension $e\colon E \to B$, then it is also split by the centralization $\Fi(e)\colon E \to B$ of this extension $e$. As a consequence, $c$ must be split by any weakly universal cover above $B$.
\end{proposition}
\begin{proof}
Consider the reflection by $\Fi$ of the pullback $P$ of $e$ and $c$ as on the right-hand side of Diagram \eqref{DiagramReflectionOfSplittingSquare}.
\begin{equation}\label{DiagramReflectionOfSplittingSquare}
\vcenter{\xymatrix@1@!0@=36pt{
& \pi_0(P) \ar@{{<}{-}{}}[rr]^-{\eta_P} \ar[dd]|(.3){\pi_0(t)}|-{\hole} \ar@{{}{=}{}}[ld]_-{} & & P \pullback \ar[rr]^-{f} \ar[dd]^(.25){t}|-{\hole} \ar[ld]|-{\eta^1_{P}} & & A \ar[dd]^-{c} \ar@{{}{=}{}}[ld]|-{} \\
\pi_0(P) \ar@{{<}{-}{}}[rr]_(.70){\eta_{\Fii(p)}} \ar[dd]_-{\pi_0(t)=\pi_0(\Fii(t))} & &\Fii(P) \ar[rr]_(.75){\Fi(f)} \ar[dd]^(.3){\Fii(t)} && A \ar[dd]^(.25){c} \\
& \pi_0(E) \ar@{{}{=}{}}[ld]_-{} \ar@{{<}{-}{}}[rr]^(.3){\eta_E}|(.5){\hole} & & E \ar[ld]|-{\eta^1_{E}} \ar[rr]^(.3){e}|(.5){\hole} && B \ar@{{}{=}{}}[ld]^-{} \\
\pi_0(E) \ar@{{<}{-}{}}[rr]_-{\eta_{\Fii(E)}} && \Fii(E) \ar[rr]_-{\Fi(e)} && B}}
\end{equation}
Since the composite of two double extensions is a pullback if and only if both double extensions are pullbacks themselves, Corollary \ref{CorollaryStronglyBirkhoffLemma} implies that the commutative squares $\Fii(t) \eta^1_P = \eta^1_E t$ and $c\Fi(f) = \Fi(e) \Fii(t)$ are pullback squares. Hence, since $\eta_E = \eta_{\Fii(E)} \eta^1_E$, and similarly $\eta_P = \eta_{\Fii(P)} \eta^1_P$, the $F$-reflection square $\eta_E t = \pi_0(t) \eta_P$ at $t$ (which is a pullback by assumption) factors through the $F$-reflection square $\pi_0(\Fii(t)) \eta_{\Fii(P)} = \eta_{\Fii(E)} \Fii(t)$ at $\Fii(t)$ via the pullback square $\Fii(t) \eta^1_P = \eta^1_E t$. Since the square $\pi_0(\Fii(t)) \eta_{\Fii(P)} = \eta_{\Fii(E)} \Fii(t)$ is a double extension, it is actually a pullback, which shows that $\Fii(t)$ is a trivial extension. We conclude by observing that a weakly universal cover above $B$ factors through $\Fi(e)$ and trivial extensions are stable by pullbacks (see also Diagram \ref{DiagramSplitByProjectivePresentations}).
\end{proof}
Given any $X$ in $\RCK$ (respectively $\QND$), the covering $\omega_X$ (respectively $\omega^q_X$) is split by itself and thus it is a normal covering. Hence its kernel pair is sent to a groupoid by the reflection $\pi_0$ (see \cite[Lemma 5.1.22]{BorJan1994}) and thus we can construct the \emph{fundamental groupoid} (see \emph{Galois groupoid} of a \emph{weakly universal central extension} as in \cite{BorJan1994}) yielding functors $\pi^r_1 \colon{\RCK \to \GRPD}$ and $\pi^q_1 \colon{\QND \to \GRPD}$, with codomain the category of ordinary groupoids $\GRPD$ (i.e. the category of internal groupoids in $\SET$).

\begin{definition}
The functor $\pi_1 \colon{\RCK \to \GRPD}$ is defined on objects by sending a rack $X$ to $\pi^r_1(X)$, the image by $\pi_0$ of the groupoid induced by taking the kernel pair of $\omega_X$. Functoriality is induced by functoriality of $\omega$.

Similarly the functor $\pi^q_1 \colon{\QND \to \GRPD}$ is defined by sending a quandle $X$ to $\pi^q_1(X)$, the image by $\pi_0$ of the groupoid induced by taking the kernel pair of $\omega^q_X$. 
\end{definition}

From there, the Galois theorem yields an equivalence of categories between the category of coverings of $X$ and the category of internal covariant presheaves over $\pi_1(X)$ (and similarly for $\QND$, see Section \ref{SectionCategoricalGaloisTheory} and references).

\subsubsection{The fundamental groupoid}

We show that the fundamental groupoid $\pi_1(X)$ (respectively $\pi_1^q(X)$) for an object $X$ in the category $\RCK$ (respectively $\QND$) is indeed the groupoid induced by the action of $\Pth(X)$ (respectively $\Pth\degree(X)$) on $X$, as suggested in M.~Eisermann's work (see \cite[Section 8]{Eis2014}). As was mentioned in the introduction, these results, and categorical Galois theory, give a positive answer to M.~Eisermann's questions about the relevance of his analogies with topology. Results about the fundamental group of a connected pointed quandle were given by V.~Even in \cite{Eve2014}. We generalize these results to the non-connected, non-pointed context in both categories $\RCK$ and $\QND$. Exploiting the analogy with the covering theory of locally connected topological spaces, this result confirms the intuition that the elements of the group $\Pth(X)$ (respectively $\Pth\degree(X)$) are representatives of the classes of \emph{homotopically equivalent paths} which connect elements in the rack (respectively quandle) $X$. 

\begin{definition}\label{DefinitionCanonicalGroupoidFromGrpAction}
Given a set $X$ and a group $G$ together with an action of $G$ on $X$, we build the ordinary groupoid $\mathcal{G}_{(X,G)}$ (in $\SET$)
\[\xymatrix@C=60pt { 
X_2 \ar@<2ex>[r]^-{p_1} \ar@<-2ex>[r]_-{p_2} \ar[r]|-{m} & X_1 \ar@(ur,ul)[]|{-1} \ar@<2ex>[r]^{c} \ar@<-2ex>[r]_{d} & X \ar[l] |{i}
}\] 
where $X_0 \defeq X$, $X_1 \defeq X \times G$ and for $a \in X_0$, $(a,g) \in X_1$,
\[ d(a,g) \defeq a; \quad c(a,g) \defeq a \cdot g; \quad i(a) \defeq (a,e);  \quad (a,g)^{-1} \defeq (a \cdot g,g^{-1});\]
$p_1, p_2 \colon {X_2 \rightrightarrows X_1}$ form the pullback of $c$ and $d$; and $m$ is the composition function defined for $\langle(a,g),(b,h)\rangle$ in $X_2$ by $m\langle(a,g),(b,h)\rangle \defeq (a,g) \cdot (b,h) \defeq (a,gh)$. 

Note that this construction actually defines a functor from the category of group actions to the category of ordinary groupoids.
\end{definition}

\begin{theorem}\label{TheoremCharactOfFundGrpdLem}
Given an object $X$ in $\RCK$ (respectively $\QND$), the fundamental groupoid $\pi_1(X)$ (resp.~$\pi^q_1(X)$) is given by the set groupoid $\mathcal{G}_{(X,\Pth(X))}$ (resp.~$\mathcal{G}_{(X,\Pth\degree(X))}$). Moreover, the groupoid morphisms induced by $f\colon{X \to Y}$ via $\Pth$ (resp.~$\Pth\degree$) and $\mathcal{G}$ correspond to $\pi_1(f)$ (resp.~$\pi^q_1(f)$).
\end{theorem}
\begin{proof}
Given the kernel pair $d_1, d_2 \colon {X_1' \rightrightarrows X'}$ of the weakly universal cover $\omega_X\colon{\tilde{X} \to X}$ (resp.~$\omega^q_X \colon{\tilde{X}\degree \to X}$), we define the groupoid $\mathcal{G}$ as:
\[\xymatrix@C=60pt { 
X_2' \ar@<2ex>[r]^-{p_1'} \ar@<-2ex>[r]_-{p_2'} \ar[r]|-{m'} & X_1' \ar@(ur,ul)[]|{-1} \ar@<2ex>[r]^{d_2} \ar@<-2ex>[r]_{d_1} & X'\ar[l] |{u}
}\] 
where $X_2'$ is the pullback of $d_2$ and $d_1$, and $m'$ is the composition function defined by the unique factorization of $d_2 \circ p_2', d_1 \circ p_1' \colon {X_2'   \rightrightarrows X' }$ through $d_2, d_1 \colon {X_1'\rightrightarrows X'}$.

Remember that a trail $(a,g) \in X'$ is represented as an arrow $g\colon \xymatrix@C=15pt{a \ar@{{}{-}{}}[r]|-{\dir{>}} & a\cdot g}$; and the action of a trail on another is as in Paragraph \ref{ParagraphTerminologyAndVisualRepr}, Equation~\eqref{EquationActionByCodomainRacks}, where the composition of arrows is understood by multiplication in $\Pth(X)$ (resp.~$\Pth\degree(X)$). 

By definition, the elements in $X_1'$ are then pairs of trails with same endpoint (diagram on the left), and the rack (resp.~quandle) operation is defined component-wise such that we have the equality on the right:
\begin{equation}\label{EquationCharactOfFundGrpd}
\vcenter{\xymatrix@C=5pt @R=15pt { & a \cdot g = b \cdot h &  \\
a \ar@{{}{-}{}}[ur]|{\dir{>}}^(0.3){g}   &  & b \ar@{{}{-}{}}[lu]|-{\dir{>}} _(0.3){h} 
}} \quad ; \qquad \vcenter{\xymatrix@C=5pt @R=15pt { & a'\cdot h' = b' \cdot g' \ar@{{}{}{}}[d] |{\bigtriangledown} & & & & a \cdot (h k) = b \cdot (g k)  &   \\
a' \ar@{{}{-}{}}[ur]|{\dir{>}}^(0.3){h'}   &  a \cdot h = b \cdot g & b' \ar@{{}{-}{}}[lu]|-{\dir{>}} _(0.3){g'} & = & &  a \cdot h = b \cdot g \ar@{{}{-}{}}[u]|{\dir{>}}^-{k }&  \\
a \ar@{{}{-}{}}[ur]|{\dir{>}}^(0.3){h}   &  & b \ar@{{}{-}{}}[lu]|-{\dir{>}} _(0.3){g} &  & a \ar@{{}{-}{}}[ur]|{\dir{>}}^(0.3){h}   &  & b \ar@{{}{-}{}}[lu]|-{\dir{>}} _(0.3){g} & 
}}
\end{equation}
where $k \defeq (h')^{-1} \gr{a'} h'$ (resp.~$k \defeq \gr{(a \cdot h)}^{-1} (h')^{-1} \gr{a'} h'$).
Finally observe that $X_2'$ is composed of pairs of elements in $X_1'$ with one matching leg (such as represented on the left), which images by $m'$ are given as in the right-hand diagram:
\[\vcenter{\xymatrix@C=7pt @R=12pt { & a\cdot g = b \cdot h = a' \cdot g'  & \\
a \ar@{{}{-}{}}[ur]|{\dir{>}}^(0.3){g}   &  b \ar@{{}{-}{}}[u]|{\dir{>}}^-{h} & a' \ar@{{}{-}{}}[lu]|-{\dir{>}}_(0.3){g'} 
}} \xymatrix{\ \ar@{{|}{-}{>}}[r]^-{m'} & \ } \vcenter{\xymatrix@C=7pt @R=12pt { & a \cdot g = a' \cdot g'  & \\
a \ar@{{}{-}{}}[ur]|{\dir{>}}^(0.3){g}   &   & a' \ar@{{}{-}{}}[lu]|-{\dir{>}} _(0.3){g'} 
}} \] 
Again the operation in $X_2'$ is defined component-wise and behaves as in $X_1'$. 

We compute the image $\pi_0(\mathcal{G})$ which is $\pi_1(X)$ (resp.~$\pi_1^q(X)$) by definition. Working on each object separately, first observe that as for $\Fr(X)$ (resp.~$\Fq$), the unit $\eta_{X'} \colon$ ${X' \to \pi_0(X') = X}$ sends a trail $(a,g) \in X \rtimes \Pth(X)$ (resp.~in $X \rtimes \Pth\degree(X)$) to its head $a \in X$, i.e.~$\eta_{X'}$ is given by the product projection on $X$.
Now for each pair of trails $\alpha =\langle (a,g),(b,h) \rangle$ in $X_1'$, we define the trail $\mu(\alpha) \defeq (a,gh^{-1})$ in $X'$: 
\[\alpha =\quad \vcenter{\xymatrix@C=7pt @R=12pt { & a \cdot g = b \cdot h &  \\
a \ar@{{}{-}{}}[ur]|{\dir{>}}^(0.3){g}   &  & b \ar@{{}{-}{}}[lu]|-{\dir{>}} _(0.3){h} 
}}\quad \mapsto\quad \vcenter{\xymatrix@C=7pt @R=12pt { & a \cdot g = b \cdot h \ar@{{}{-}{}}[rd]|-{\dir{>}} ^(0.7){h^{-1}}  &  \\
a \ar@{{}{-}{}}[ur]|{\dir{>}}^(0.3){g}   &  & b
}} \quad \eqqcolon \mu(\alpha). \]
Observe that this trail $\mu(\alpha)$ is invariant under the action on $\alpha$, of other pairs $\beta = \langle (a',g'),(b',h') \rangle$ in $X_1'$, since $\mu(\alpha \qndop \beta) = (a, hkk^{-1}g^{-1})= \mu(\alpha)$, where $k = (h')^{-1} \gr{a'} h'$ (resp.~$k=\gr{(a \cdot h)}^{-1}$ $(h')^{-1} \gr{a'} h'$) is the common part of both left and right legs as in Equation~\eqref{EquationCharactOfFundGrpd}. Conversely suppose that $\alpha$, $\alpha'$ in $X_1'$ have the same image by $\mu$, we show that $\alpha$ and $\alpha'$ are connected in $X_1'$. Indeed, $\alpha$ and $\alpha'$ must then be of the form $\alpha = \langle(a,g),(b,h)\rangle$ and $\alpha' = \langle(a,g'),(b,h')\rangle$, such that moreover  $gh^{-1} = g'h'^{-1}$. Then the path $l \defeq h^{-1}h'= g^{-1}g' \in \Pth(X)$ (resp.~in $\Pth\degree(X)$) decomposes as a product $l = \gr{x_0}^{\delta_0}\cdots\gr{x_n}^{\delta_n}$, such that all the pairs $\langle (x_i,e), (x_i,e) \rangle$ are in $X_1'$ (and we have moreover $\sum_{i=0}^n \delta_i = 0$ in the context of $\QND$). By acting with these pairs ``$- \qndop^{\delta_i} \langle (x_i,e), (x_i,e) \rangle$'' on $\alpha$, we may obtain $\alpha'$ as in the diagram on the right:
\[\alpha \defeq \vcenter{\xymatrix@C=7pt @R=15pt { & a \cdot g = b \cdot h &  \\
a \ar@{{}{-}{}}[ur]|{\dir{>}}^(0.3){g}   &  & b \ar@{{}{-}{}}[lu]|-{\dir{>}} _(0.3){h} 
}} \text{ and }\quad \alpha' \defeq \vcenter{\xymatrix@C=7pt @R=15pt { & a \cdot g' = b \cdot h' &  \\
a \ar@{{}{-}{}}[ur]|{\dir{>}}^(0.3){g'}   &  & b \ar@{{}{-}{}}[lu]|-{\dir{>}} _(0.3){h'} 
}} \quad = \quad \vcenter{\xymatrix@C=7pt @R=12pt { & a \cdot (gl) = a\cdot (hl) & \\
 & a \cdot g = b \cdot h \ar@{{}{-}{}}[u]|{\dir{>}}^-{l} &  \\
a \ar@{{}{-}{}}[ur]|{\dir{>}}^(0.3){g}   &  & b \ar@{{}{-}{}}[lu]|-{\dir{>}} _(0.3){h} 
}}\]
Hence we have the unit morphism $\eta_{X_1'} = \mu\colon{ X_1' \to \pi_0(X_1')}$ where $\pi_0(X_1')$ is $\pi_0(\Eq(\omega_X)) = X \times \Pth(X)$ (resp.~$\pi_0(\Eq(\omega^q_X)) =  X \times \Pth\degree(X)$). We may then compute $\pi_0(d_2) = c$, $\pi_0(d_1) = d$, $\pi_0(i) = u$ and $\pi_0(-1) = -1$, as displayed in the commutative diagram of plain arrows,
\[ 
\xymatrix@C=60pt @R=20pt { 
X_2' \ar@{{}{..}{>}}[d]_-{\eta_{X_2'} = \mu \times \mu} \ar@<1.5ex>[r]^-{p_1'} \ar@<-1.5ex>[r]_-{p_2'} \ar[r]|-{m'} & X_1' \ar[d]|-{\eta_{X_1'} = \mu} \ar@(ur,ul)[]|{-1} \ar@<1.5ex>[r]^{d_2} \ar@<-1.5ex>[r]_{d_1} & X' \ar[d]^-{\eta_{X'} = d} \ar[l] |{u}  \ar[r]^-{\omega_X \text{ (resp. }\omega^q_X\text{)}} & X\\
X_2 \ar@<1.5ex>[r]^-{p_1} \ar@<-1.5ex>[r]_-{p_2} \ar[r]|-{m} & X_1 \ar@(dr,dl)[]|{-1} \ar@<1.5ex>[r]^-{c} \ar@<-1.5ex>[r]_-{d} & X  \ar[l] |-{i} &
} \]
where the bottom groupoid is the inclusion in $\RCK$ (resp.~$\QND$) of the groupoid $\mathcal{G}_{(X,\Pth(X))}$ (resp.~$\mathcal{G}_{(X,\Pth\degree(X))}$) from $\SET$. Hence $X_1 = X \times \Pth(X)$ (resp.~$X_1 = X \times \Pth\degree(X)$) has the same underlying set as $X'$, and the underlying functions of $\eta_{X'}$ and $d$ are both given by ``projection on $X$''.

Then since $\omega_X$ (resp.~$\omega^q_X$) is a normal covering, $d_1$ and $d_2$ are trivial extensions, so that the commutative squares $d d_1 = d \mu$ and $d d_2 = c \mu$ are actually pullback squares. Hence the pullback $p_1', p_2'\colon{X_2' \rightrightarrows X_1'}$ of $d_2$ and $d_1$ and the pullback $p_1, p_2\colon{X_2 \rightrightarrows X_1}$ of $c$ and $d$, induce a morphism $f\colon{X_2' \to X_2}$ which is thus the pullback of $\eta_{X_1'} = \mu$ and computed component-wise as $f = \mu \times \mu$. By admissibility of the Galois structure $\Gamma$ (see Paragraph \ref{ParagraphAdmissibilityForGalois} and \cite{JK1994}), this morphism is also the unit component $f = \eta_{X_2'}$. Finally the commutativity of the square $\mu m' = m \eta_{X_2'}$ is given by construction (and easy to check by hand), which concludes the proof that $\pi_1(X) = \pi_0(\mathcal{G}) = \mathcal{G}_{(X,\Pth(X))}$ (resp.~$\pi^q_1(X) = \mathcal{G}_{(X,\Pth\degree(X))}$ in $\QND$). 
\end{proof}

\paragraph{Remarks}\label{ParagraphRamarksFundamentalGroupoid} Remember from Paragraph \ref{ParagraphConnectedComponentsAreNotConnected} that the notion of connectedness is not local. Now relate this fact to the \emph{regularity} of the fundamental groupoid of a rack, whose domain map is the projection map of a cartesian product: given a rack $A$, the set of homotopy classes of paths of a given domain $a \in A$ is always $\Pth(A)$ and thus independent of the domain $a$. Since every path is invertible, the same is true for the homotopy classes of paths of a given endpoint. 

One of D.E.~Joyce's main results is to show that the \emph{knot quandle} is a complete invariant for oriented knots. Now the \emph{knot group} \cite{Rei1928} of an oriented knot, which is the fundamental group of the ambient space of the knot, is also computed as the group of paths of the knot quandle. In other words, the knot group is the \emph{fundamental group} of the knot quandle, in the sense of the covering theory of racks (not in the sense of the covering theory of quandles).

Finally observe that $\pi_1(X)$ (resp.~$\pi^q_1(X)$) can be equipped with a non-trivial ad-hoc structure of rack (resp.~quandle) making it into an internal groupoid in $\RCK$ (resp.~$\QND$) with internal object of objects the rack (resp.~quandle) $X$. Given two trails $(a,g)$ and $(b,h)$ in $X_1$, define $(a,g) \qndop (b,h) \defeq (a \qndop b, \gr{b}^{-1} g h^{-1} \gr{b} h)$ (note that if $g$, $h \in \Pth\degree(X)$, then $\gr{b}^{-1} g h^{-1} \gr{b} h \in \Pth\degree(X)$). Unlike in $\hat{X}$ (resp.~$\hat{X}\degree$), trails act on each other with both their heads and end-points, which means that both projections to $X$ are morphisms in $\RCK$ (resp.~$\QND$). The rest of the structure is easy to derive. 

\paragraph{Working with skeletons} As we shall see in the next section, we are interested in the fundamental groupoid, up to equivalence. Given a rack $A$, we thus also describe a \emph{skeleton} $S$ of $\pi_1(A)$ (in the sense of \cite[Section IV.4]{McLane1997}). The resulting groupoid $S$ is not regular like $\pi_1(A)$, it is totally disconnected and its vertices are the connected components of $A$. With the objective of interpreting the fundamental theorem of Galois theory, the homotopical information contained in $\pi_1(A)$ can be made more explicit using its skeleton. 

\begin{definition}\label{DefinitionLoopGroup}
Given an object $A$ in $\RCK$ (respectively in $\QND$), we call a \emph{pointing} of $A$ any choice of representatives $I \defeq \lbrace a_i \rbrace_{i \in \pi_0(A)} \subseteq A$ such that $\eta_A(a_i) = [a_i] = i$ for each equivalence class $i \in \pi_0(A)$. Then for any element $a \in A$, define $\Loop_{a}$ as the group of loops $l \in \Pth(A)$ (resp.~$l \in \Pth\degree(A)$) such that $a\cdot l = a$. Observe that if $[a] = [b]$, for some $a$ and $b$ in $A$, then there is $g \in \Pth(A)$ (resp.~$g \in \Pth\degree(A)$) such that $a = b \cdot g$ and thus the subgroups $\Loop_{a}$ and $\Loop_{b}$ are isomorphic, via the automorphism of $\Pth(A)$ (resp.~$\Pth\degree$) given by conjugation with $g$.

Let us fix a pointing $I \defeq \lbrace a_i \rbrace_{i \in \pi_0(A)} \subseteq A$ of $A$, then we define the groupoid $\pi_1(A,I)$ (resp.~$\pi^q_1(A,I)$) as
\[ 
\xymatrix@C=50pt @R=20pt {
A_2 \ar@<1.5ex>[r]^-{p_1} \ar@<-1.5ex>[r]_-{p_2} \ar[r]|-{m} & A_1 \ar@(ul,ur)[]|{-1} \ar@<1.5ex>[r]^-{c} \ar@<-1.5ex>[r]_-{d} & \pi_0(A),  \ar[l] |-{i}
} \]
where $A_1 \defeq \coprod_{i\in \pi_0(A)}\Loop_{a_i}$ is defined as the disjoint union, of the underlying sets of $\Loop_{a_i}$'s indexed by $i \in \pi_0(A)$. The domain and codomain maps send a loop  $l \in \Loop_{a_i}$ to the index $i \in \pi_0(A)$. The set $A_2$ is then the disjoint union of products $A_2 \defeq \coprod_{i\in \pi_0(A)} (\Loop_{a_i} \times \Loop_{a_i})$ and $m$ is defined by multiplication in $\Loop_{a_i} \leq \Pth(A)$ (resp.~$\Loop_{a_i} \leq \Pth\degree(A)$). 
\end{definition}

From the description of the skeleton of a groupoid obtained as in Definition \ref{DefinitionCanonicalGroupoidFromGrpAction}, we deduce:

\begin{lemma}
For each $I$ pointing of $A$ object of $\RCK$ (respectively of $\QND$),  $\pi_1(A,I)$ (respectively $\pi^q_1(A,I)$) is a skeleton of the fundamental groupoid $\pi_1(A)$ (respectively $\pi^q_1(A)$).
\end{lemma}

\subsection{The fundamental theorem of categorical Galois theory}\label{SectionFundamentalTheorem} In sections 5, 6 and 7 of \cite{Eis2014}, M.~Eisermann studies in detail different classification results for quandle coverings. We will not go into so much depth ourselves, however we show how to recover and extend the main theorems from these sections using categorical Galois theory.

Given an object $A$ in $\RCK$ (respectively $\QND$), the category of \emph{internal covariant presheaves} over $\pi \defeq \pi_1(A)$ (resp.~$\pi \defeq \pi^q_1(A)$) are externally described as the category of functors from $\pi$ to $\SET$ and thus as the category of $\pi$-groupoid actions on sets $\SET^{\pi}$. Given a pointing $I$ of $A$, define $\pi(I) \defeq \pi_1(A,I)$ (resp.~$\pi(I) \defeq  \pi^q_1(A,I)$ and deduce from $\pi(I) \cong \pi$ that $\SET^{\pi} \cong \SET^{\pi(I)}$. Now $\pi(I)$ is totally disconnected, thus the category of $\pi(I)$-actions is equivalent to the category $\coprod_{i\in \pi_0(A)} \SET^{\Loop_{a_i}}$ whose objects are sequences of $\Loop_{a_i}$-group actions (see Definition \ref{DefinitionLoopGroup}), indexed by $i \in \pi_0(A)$, and morphisms between these are $\pi_0$-indexed sums of group-action morphisms. From the fundamental theorem of categorical Galois theory (see for instance \cite[Theorem 6.2]{JK1994}), classifying central extensions above an object we deduce in particular:

\begin{theorem}
Given an object $A$ in $\RCK$ and a pointing $I \defeq \lbrace a_i \rbrace_{i \in \pi_0(A)} \subseteq A$ of $A$, there is a natural equivalence of categories between the category $\CEXT(A)$ of central extensions above $A$ and the category $\SET^{\pi_1(A)}$. The latter category is then also equivalent (but not naturally) to $\SET^{\pi_1(A,I)} \cong \coprod_{i\in \pi_0(A)} \SET^{\Loop_{a_i}}$. The same theorem holds in $\QND$, using the appropriate definition of $\Loop_{a_i}$ and using $\pi^q_1(A)$ and $\pi^q_1(A,I)$ instead of $\pi_1(A)$ and $\pi_1(A,I)$. 
\end{theorem}

\begin{corollary}
The category of central extensions above a connected rack $A$ is equivalent to the category of $\Loop_{a}$-actions (from Definition \ref{DefinitionLoopGroup}), for any given element $a \in A$. The same is true in $\QND$.
\end{corollary}

\begin{example}
We illustrate this result on a trivial example, to show the difference between the context of $\RCK$ and that of $\QND$. Consider the one element set $1$. The coverings above $1$ in $\QND$ should all be surjective maps to $1$ in $\SET$, whereas the coverings above $1$ in $\RCK$ include for instance the unit morphism $^r\eta^q_{\Fr(1)} = \eta_{\Fr_1} \colon \Fr(1) \to \Fq(1) = 1$, whose domain is not a set. Then observe that $\Pth(1) = \Z$ and thus $\Pth\degree(1) = \lbrace e \rbrace$ and since there is only one element $* \in 1$, $\Loop_*$ is the former in $\RCK$ and the latter in $\QND$. Hence the category of coverings above $1$ in $\QND$ is $\SET^{\lbrace e \rbrace}$ which is indeed equivalent to $\SET$. The category of coverings above $1$ in $\RCK$ is given by $\SET^{\Z}$, the category of $\Z$-actions on sets, where $\Z$ is the additive group of integers.
\end{example}

\subsection{Relationship to groups and abelianization}\label{SectionRelationshipToAbelianisation}

The following relationship between $\pi_0 \dashv \I$ in $\RCK$ (or $\QND$) and the abelianization in groups has played an important role in the study of the present paper, and in the identification of the relevant centrality conditions in higher dimensions described in Part II and III. 

Let us comment first of all that the subvariety of sets is absolutely not a \emph{Mal'tsev category}, and the adjunction $\pi_0 \dashv \I$ does not arise from an \emph{abelianization adjunction} like, for instance, in the case of \emph{abelian sym quandles} studied in \cite{EveGrMo2016} (a quandle is \emph{sym} if $\qndop$ is commutative). For instance, the distinction with the study in \cite{EveGrMo2016} is clear since the only connected \emph{sym quandle} is $\lbrace \ast \rbrace$, also the only group whose conjugation is sym is the trivial group $\lbrace e \rbrace$. The relation between centrality in racks/quandles and the classical notions of centrality induced by Mal'tsev or partial Mal'tsev contexts, appears to us as more subtle than: one being merely an example of the other. 

The following comments also apply to the context of $\QND$, however we like to work in the more ``primitive'' context of $\RCK$ considering the role of $\Pth$ in the comparison with groups, and its tight relationship with the axioms of racks.

We study which squares of functors commute in the following square of adjunctions. Starting with an abelian group $G$, conjugation in $G$ is trivial hence $\Conj(\I(G))$ is the trivial quandle on the underlying set of $G$.  Since also both composites send a morphism to the underlying function we have $\Conj \I = \I \U$ and thus the restriction of $\Conj$ to abelian groups gives the forgetful functor to $\SET$. By uniqueness of left adjoints we must also have $\Fa \pi_0 = \ab \Pth$. A direct proof easily follows from the corresponding group presentations. 
\[\xymatrix@C=25pt@R=15pt{ \RCK \ar@{{}{}{}}[dd]|-{\dashv}  \ar@<-3pt>@/_7pt/[rr]_-{\pi_0}  \ar@<-3pt>@/_7pt/[dd]_-{\Pth} \ar@{{}{}{}}[rr]|-{\top}  & & \SET \ar@{{}{}{}}[dd]|-{\dashv}   \ar@<-3pt>@/_7pt/[ll]_-{\I} \ar@<-3pt>@/_7pt/[dd]_-{\Fa} \\
\\
 \GRP  \ar@{{}{}{}}[rr]|-{\top}  \ar@<-3pt>@/_7pt/[uu]_-{\Conj} \ar@<-3pt>@/_7pt/[rr]_-{\ab}   & &  \AB  \ar@<-3pt>@/_7pt/[ll]_-{\I} \ar@<-3pt>@/_7pt/[uu]_-{\U}
}\]

Now starting with a set $X$ in $\SET$ we may consider it as a trivial quandle by application of $\I$. Then we compute 
\[ \Pth(\I(X)) \defeq \Fg(X)/\langle (x\qndop a)^{-1}a^{-1} xa | a,x \in X \rangle = \Fg(X)/\langle  x^{-1}a^{-1} xa | a,x \in X \rangle, \]
which shows that for each set $X$ we have $\Pth(\I(X)) = \I\Fa(X)$, which then easily gives $\Pth \I = \I \Fa$, i.e.~the restriction of $\Pth$ to trivial racks gives the free abelian group functor. 

Observe that we cannot use uniqueness of adjoints to derive that $\pi_0 \Conj$ is the same as $\U \ab$. Indeed we compute that ($\pi_0$, $\Conj$, $\U$, $\ab$) is the only square of functors that doesn't commute. Given a group $G$, the image $\pi_0(\Conj(G))$ is given by the set of conjugacy classes. The corresponding congruence in $\QND$ is given by 
\begin{equation}\label{EquationConnCompCongruenceForGroups}
a \sim b \Leftrightarrow (\exists c \in G)(c^{-1} a c = b). 
\end{equation} 
Then the abelianization $\ab(G)$ is the quotient of $G$ by the congruence generated  in $\GRP$ by the identities $\lbrace c^{-1}ac = a\ |\ a,c \in G \rbrace$. We may show that in general the equivalence relation defined in \eqref{EquationConnCompCongruenceForGroups} does not define a group congruence. A counter-example is given by the group of permutations $S_3$. It has three conjugacy classes given by cycles, two permutations and the unit. The derived subgroup is the alternating group $A_3$ which is of order 2. This shows  that there are less elements in the abelianization of $S_3$ than there are conjugacy classes in $S_3$. 

Understand that an ``image'' of the covering theory in $\RCK$, arising from the adjunction $\pi_0 \dashv \I$ can be studied in groups through the functor $\Pth$ and its restriction to sets $\Fa$. Note that $\Pth$ is neither full, or faithful, nor essentially surjective. The functor $\Fa$ is full and faithful. We will not study what information to extract from this image. Yet, again, we have been using ingredients of this image to describe centrality in $\RCK$ such as in Theorem \ref{TheoremEpiReflectivityOfCExt}. Observe moreover that any covering in racks induces a central extension between the groups of paths \cite[Proposition 2.39]{Eis2014}. However, certain morphisms, such as $f\colon{Q_{ab\star} \to \lbrace \ast \rbrace}$, which are not central in $\RCK$ (or $\QND$) are sent by $\Pth$ to central extensions of groups, e.g. 
\[\vec{f}\  \colon \ \Pth(Q_{ab\star}) = \Z \times \Z \to \Z = \Pth(\lbrace \ast \rbrace) \ \colon \ (k,l) \mapsto  k+l.  \] 

In the other direction an ``image'' of the theory of central extensions of groups can be studied in $\RCK$ via the ``inclusion'' $\Conj$ and its restriction $\U$ on abelian groups.  Both $\Conj$ and $\U$ are not full but faithful, $\U$ is moreover surjective. Again we shall not develop the full potential of this study. Observe nonetheless that a morphism of groups is central if and only if it gives a covering in racks \cite[Example 2.34]{Eis2014}, see also \cite[Example 1.2]{Eis2014} and comments below. We give some more results about this relationship in Part II.

\section*{Acknowledgements}
Many thanks to my supervisors Tim Van der Linden and Marino Gran for their support, advice and careful proofreading of this work.

\bibliography{Part1.bib}

\providecommand{\bysame}{\leavevmode\hbox to3em{\hrulefill}\thinspace}
\providecommand{\MR}{\relax\ifhmode\unskip\space\fi MR }
\providecommand{\MRhref}[2]{%
  \href{http://www.ams.org/mathscinet-getitem?mr=#1}{#2}
}
\providecommand{\href}[2]{#2}
\begin{thebibliography}{10}

\bibitem{Bar1971}
M.~Barr, \emph{Exact categories}, Lecture Notes Math (Berlin, Heidelberg),
  Exact Categories and Categories of Sheaves, vol. 236, Springer, 1971,
  pp.~1--120.

\bibitem{BarBec1969}
M.~Barr and J.~Beck, \emph{Homology and standard constructions}, Lecture Notes
  Math (Berlin, Heidelberg) (B.~Eckmann, ed.), Seminar on Triples and
  Categorical Homology Theory, vol.~80, Springer, 1969.

\bibitem{BaGrOs1971}
M.~Barr, P.~A. Grillet, and D.~H. van Osdol, \emph{Exact categories and
  categories of sheaves}, Lecture Notes in Math., vol. 236, Springer, 1971.

\bibitem{Ber2008}
W.~Bertram, \emph{Differential geometry, {Lie} groups and symmetric spaces over
  general base fields and rings}, Memoirs of the American Mathematical Society,
  vol. 900, AMS, Providence, R.I., 2008.

\bibitem{BonSta2021}
M.~Bonatto and D.~Stanovský, \emph{Commutator theory for racks and quandles},
  J. Math. Soc. Japan \textbf{73} (2021), no.~1, 41--75.

\bibitem{BorJan1994}
F.~Borceux and G.~Janelidze, \emph{Galois theories}, Cambridge studies in
  advanced mathematics, Vol. 72, Cambridge University Press, 1994.

\bibitem{Bou2003}
D.~Bourn, \emph{The denormalized {$3\times 3$} lemma}, J.~Pure Appl. Algebra
  \textbf{177} (2003), 113--129.

\bibitem{Bri1988}
E.~Brieskorn, \emph{Automorphic sets and braids and singularities}, Braids
  (Providence, RI), Contemporary Mathematics, vol.~78, Amer. Math. Soc., 1988,
  pp.~45--115.

\bibitem{Bro1999}
R.~Brown, \emph{Groupoids and crossed objects in algebraic topology}, Homology,
  Homotopy and Applications \textbf{1} (1999).

\bibitem{Bro2006}
R.~Brown, \emph{Topology and groupoids}, Cambridge studies in advanced
  mathematics, Vol. 72, Booksurge LLC, S. Carolina, 2006.

\bibitem{BroEll1988}
R.~Brown and G.~Ellis, \emph{Hopf formulae for the higher homology of a group},
  Bulletin of the London Mathematical Society \textbf{20} (1988).

\bibitem{BroHiSi2011}
R.~Brown, P.J. Higgins, and R.~Sivera, \emph{Nonabelian algebraic topology:
  Filtered spaces, crossed complexes, cubical homotopy groupoids}, EMS Tracts
  in Mathematics, vol.~15, European Mathematical Society (EMS), Z\"{u}rich,
  2011.

\bibitem{BroJan1999}
R.~Brown and G.~Janelidze, \emph{Galois theory of second order covering maps of
  simplicial sets}, Journal of Pure and Applied Algebra \textbf{135} (1999),
  no.~1, 23 -- 31.

\bibitem{BroJan2004}
R.~Brown and G.~Janelidze, \emph{Galois theory and a new homotopy double
  groupoid of a map of spaces}, Applied Categorical Structures \textbf{12}
  (2004), no.~1, 63 -- 80.

\bibitem{BLRY2010}
E.~Bunch, P.~Lofgren, A.~Rapp, and D.~N. Yetter, \emph{On quotients of
  quandles}, J. Knot Theory Ramifications \textbf{19} (2010), no.~9.

\bibitem{BuSan1981}
S.~Burris and H.~P. Sankappanavar, \emph{A course in universal algebra},
  Graduate Texts in Mathematics, vol.~78, Springer-Verlag, 1981.

\bibitem{CJKP1997}
A.~Carboni, G.~Janelidze, G.~M. Kelly, and R.~Par{\'e}, \emph{On localization
  and stabilization for factorization systems}, Applied Categorical Structures
  \textbf{5} (1997), no.~1, 1--58.

\bibitem{CKP1993}
A.~Carboni, G.~M. Kelly, and M.~C. Pedicchio, \emph{Some remarks on {Maltsev}
  and {Goursat} categories}, Appl. Categ. Struct. \textbf{1} (1993), 385--421.

\bibitem{CLP1991}
A.~Carboni, J.~Lambek, and M.C. Pedicchio, \emph{Diagram chasing in {Mal'cev}
  categories}, Journal of Pure and Applied Algebra \textbf{69} (1991), no.~3,
  271 -- 284.

\bibitem{CPP1992}
A.~Carboni, M.~C. Pedicchio, and N.~Pirovano, \emph{Internal graphs and
  internal groupoids in {M}al'cev categories}, Am. Math. Soc. for the Canad.
  Math. Soc., Providence, 1992, pp.~97--109.

\bibitem{CaKaSa2001}
J.~S. Carter, S.~Kamada, and M.~Saito, \emph{Geometric interpretations of
  quandle homology}, Journal of Knot Theory and Its Ramifications \textbf{10}
  (2001), no.~03, 345--386.

\bibitem{CJKLS2003}
J.S. Carter, D.~Jelsovsky, S.~Kamada, L.~Langford, and M.~Saito, \emph{Quandle
  cohomology and state-sum invariants of knotted curves and surfaces}, Trans.
  Amer. Math. Soc. \textbf{355} (2003), no.~4, 3947–3989.

\bibitem{CaHeKe1985}
C.~Cassidy, M.~H\'ebert, and G.~M. Kelly, \emph{Reflective subcategories,
  localizations and factorizations systems}, J. Austral. Math. Soc. \textbf{38}
  (1985), 287--329.

\bibitem{dCosta2013}
I.A. da~Costa, \emph{Reflections of universal algebras into semilattices, their
  {Galois} theories and related factorization systems.}, Ph.D. thesis,
  Universidade de Aveiro, 2013.

\bibitem{Dri1992}
V.~G. Drinfeld, \emph{On some unsolved problems in quantum group theory},
  Quantum Groups (Berlin, Heidelberg) (Petr~P. Kulish, ed.), Springer Berlin
  Heidelberg, 1992, pp.~1--8.

\bibitem{DuEvMo2017}
M.~Duckerts-Antoine, V.~Even, and A.~Montoli, \emph{How to centralize and
  normalize quandle extensions}, Journal of Knot Theory and Its Ramifications
  \textbf{27} (2018), no.~02, 1850020.

\bibitem{Eis2014}
M.~Eisermann, \emph{Quandle coverings and their {Galois} correspondence}, Fund.
  Math. \textbf{225(1)} (2014), 103--168.

\bibitem{Eve2014}
V.~Even, \emph{A {Galois}-theoretic approach to the covering theory of
  quandles}, Appl. Categ. Struct. \textbf{22} (2014), no.~5, 817--831.

\bibitem{Eve2015}
V.~Even, \emph{Coverings, factorization systems and closure operators in the
  category of quandles}, Ph.D. thesis, Université catholique de Louvain, 2015.

\bibitem{EveGr2014}
V.~Even and M.~Gran, \emph{On factorization systems for surjective quandle
  homomorphisms}, Journal of Knot Theory and Its Ramifications \textbf{23}
  (2014), no.~11, 1450060.

\bibitem{EveGrMo2016}
V.~Even, M.~Gran, and A.~Montoli, \emph{A characterization of central
  extensions in the variety of quandles}, Theory and Applications of Categories
  \textbf{31} (2016), 201--216.

\bibitem{Ever2010}
T.~Everaert, \emph{Higher central extensions and {Hopf} formulae}, Journal of
  Algebra \textbf{324} (2010), no.~8, 1771 -- 1789.

\bibitem{EvGoeVdl2012}
T.~Everaert, J.~Goedecke, and T.~Van~der Linden, \emph{Resolutions, higher
  extensions and the relative {M}al'tsev axiom}, J.~Algebra \textbf{371}
  (2012), 132--155.

\bibitem{EvGrVd2008}
T.~Everaert, M.~Gran, and T.~Van~der Linden, \emph{Higher {Hopf} formulae for
  homology via {Galois} theory}, Advances in Mathematics \textbf{217} (2008),
  no.~5, 2231 -- 2267.

\bibitem{EvVdl2013}
T.~Everaert and T.~Van~der Linden, \emph{Baer invariants in semi-abelian
  categories {II}: Homology}, Theory and Applications of Categories \textbf{12}
  (2004), no.~4, 195--224.

\bibitem{FenRou1992}
R.~Fenn and C.~Rourke, \emph{Racks and links in codimension two}, J. Knot
  Theory Ramifications \textbf{1(4)} (1992), 343--406.

\bibitem{FeRoSa1995}
R.~Fenn, C.~Rourke, and B.~Sanderson, \emph{Trunks and classifying spaces},
  Applied Categorical Structures \textbf{3} (1995), no.~4, 321--356.

\bibitem{FreYet1989}
P.~J. Freyd and D.~N. Yetter, \emph{Braided compact closed categories with
  applications to low dimensional topology}, Advances in Mathematics
  \textbf{77} (1989), no.~2, 156 -- 182.

\bibitem{Froe1963}
A.~Fr{\"o}hlich, \emph{Baer-invariants of algebras}, Trans. Amer. Math. Soc.
  \textbf{109} (1963), 221--244.

\bibitem{GraVdl2008}
M.~Gran and T.~Van~der Linden, \emph{On the second cohomology group in
  semi-abelian categories}, J.~Pure Appl.\ Algebra \textbf{212} (2008),
  636--651.

\bibitem{Hat2001}
A.~Hatcher, \emph{Algebraic topology}, Cambridge University Press, 2001.

\bibitem{Hel2012}
S.~Helgason, \emph{Differential geometry, {L}ie groups, and symmetric spaces},
  Graduate studies in mathematics, vol.~34, American Mathematical Society,
  Providence, RI, 2012.

\bibitem{ImKel1986}
G.~B. Im and G.M. Kelly, \emph{On classes of morphisms closed under limits}, J.
  Korean Math. Soc. \textbf{23} (1986), no.~1, 1 -- 18.

\bibitem{Jan1990}
G.~Janelidze, \emph{Pure {Galois} theory in categories}, Journal of Algebra
  \textbf{132} (1990), no.~2, 270 -- 286.

\bibitem{JanComo1991}
G.~Janelidze, \emph{Precategories and {G}alois theory}, Category {T}heory,
  {P}roceedings {C}omo 1990 (A.~Carboni, M.~C. Pedicchio, and G.~Rosolini,
  eds.), Lecture Notes in Math., vol. 1488, Springer, 1991, pp.~157--173.

\bibitem{Jan1991}
G.~Janelidze, \emph{What is a double central extension ? (the question was
  asked by {Ronald Brown})}, Cah. de Top. et Géom. Diff. Cat. \textbf{32}
  (1991), no.~3, 191--201 (eng).

\bibitem{Jan1995}
G.~Janelidze, \emph{Higher dimensional central extensions: A categorical
  approach to homology theory of groups}, Lecture at the International Category
  Theory Meeting CT95 (Halifax), 1995.

\bibitem{Jan2008}
G.~Janelidze, \emph{Galois groups, abstract commutators, and {Hopf} formula},
  Applied Categorical Structures \textbf{16} (2008), 653--668.

\bibitem{Jan2016}
G.~Janelidze, \emph{A history of selected topics in categorical algebra {I}:
  {F}rom {G}alois theory to abstract commutators and internal groupoids},
  Categories and General Algebraic Structures with Applications \textbf{5}
  (2016), no.~1, 1--54.

\bibitem{JK1994}
G.~Janelidze and G.M. Kelly, \emph{Galois theory and a general notion of
  central extension}, Journal of Pure and Applied Algebra \textbf{97} (1994),
  no.~2, 135 -- 161.

\bibitem{JK1997}
G.~Janelidze and G.M. Kelly, \emph{The reflectivness of covering morphisms in
  algebra and geometry}, Theory and Applications of Categories \textbf{3}
  (1997), no.~6, 132 -- 159.

\bibitem{JK2000b}
G.~Janelidze and G.M. Kelly, \emph{Central extensions in universal algebra: A
  unification of three notions}, Algebra Universalis \textbf{44} (2000),
  123--128.

\bibitem{JanSoTho2004}
G.~Janelidze, M.~Sobral, and W.~Tholen, \emph{Beyond {B}arr exactness:
  Effective descent morphisms}, {C}ategorical Foundations: Special Topics in
  Order, Topology, Algebra and Sheaf Theory (M.~C. Pedicchio and W.~Tholen,
  eds.), Encyclopedia of Math. Appl., vol.~97, Cambridge Univ. Press, 2004,
  pp.~359--405.

\bibitem{JanTho1994}
G.~Janelidze and W.~Tholen, \emph{Facets of descent, {I}}, Applied Categorical
  Structures \textbf{2} (1994), no.~3, 245--281.

\bibitem{John1977}
P.~T. Johnstone, \emph{Topos theory}, L.M.S. Monographs, vol.~10, Academic
  Press, 1977.

\bibitem{Joy1979}
D.~Joyce, \emph{An algebraic approach to symmetry and applications in knot
  theory}, Ph.D. thesis, University of Pennsylvenia, 1979.

\bibitem{Joy1982}
D.~Joyce, \emph{A classifying invariant of knots, the knot quandle}, Journal of
  Pure and Applied Algebra \textbf{23(1)} (1982), 37--65.

\bibitem{Kau2012}
L.H. Kauffman, \emph{Knots and physics (fourth edition)}, Series On Knots And
  Everything, World Scientific Publishing Company, 2012.

\bibitem{Kru1998}
B.~Kr\"{u}ger, \emph{Automorphe {Mengen} und die artinschen {Zopfgruppen}},
  Bonner mathematische Schriften, Dissertation, Rheinische
  Friedrich-Wilhelms-Universit\"{a}t Bonn, vol. 207, 1989.

\bibitem{Loos1969}
O.~Loos, \emph{Symmetric spaces: General theory}, Mathematics Lecture Note
  Series, W. A. Benjamin, 1969.

\bibitem{Lue1967}
A.~S.-T. Lue, \emph{Baer-invariants and extensions relative to a variety},
  Math. Proc. Cambridge Philos. Soc. \textbf{63} (1967), 569--578.

\bibitem{McLane1997}
S.~Mac~Lane, \emph{Categories for the working mathematician}, second edition
  ed., Springer, Chicago, 1997.

\bibitem{MalSbo1954}
A.~I. Mal'tsev, \emph{On the general theory of algebraic systems}, Mat. Sbornik
  N. S. \textbf{35} (1954), no.~6, 3--20.

\bibitem{Mat1984}
S.~V. Matveev, \emph{Distributive groupoids in knot theory}, Mathematics of the
  USSR-Sbornik \textbf{47} (1984).

\bibitem{Mil1971}
J.~Milnor, \emph{Introduction to algebraic {K}-theory}, Princeton University
  Press, 1971.

\bibitem{Rei1928}
K.~Reidemeister, \emph{Ueber {K}notengruppen}, Abh. {M}ath. {S}em. {U}niv.
  {H}amburg \textbf{6} (1928), 56--64.

\bibitem{Ryd1993}
H.J. Ryder, \emph{The structure of racks}, Ph.D. thesis, University of Warwick,
  1993.

\bibitem{Tak1943}
M.~Takasaki, \emph{Abstraction of symmetric transformations}, Tohoku
  Mathematical Journal, First Series \textbf{49} (1943), 145--207.

\end{thebibliography}
\bibliographystyle{amsplain-nodash}
\end{document}